\documentclass[a4paper,12pt]{amsart}
\usepackage{amsmath}
\usepackage{amsthm}
\usepackage{amsfonts}
\usepackage{amssymb,array}
\usepackage{enumerate}
\usepackage[dvips]{graphicx}
\usepackage{comment}
\usepackage{ccaption}
\usepackage[all]{xy}

\usepackage{color}

\theoremstyle{definition}
\newtheorem{define}{Definition}[section]
\theoremstyle{plain}

\theoremstyle{plain}
\newtheorem{theorem}[define]{Theorem}
\theoremstyle{plain}
\newtheorem{lemma}[define]{Lemma}
\theoremstyle{plain}
\newtheorem{proposition}[define]{Proposition}
\theoremstyle{plain}
\newtheorem{corollary}[define]{Corollary}
\theoremstyle{remark}
\newtheorem{remark}[define]{Remark}
\theoremstyle{definition}

\theoremstyle{definition}
\newtheorem{fact}[define]{Fact}
\theoremstyle{definition}
\newtheorem{eg}[define]{Example}
\theoremstyle{definition}
\newtheorem{notation}[define]{Notation}

\theoremstyle{definition}

\theoremstyle{definition}

\theoremstyle{definition}

\theoremstyle{definition}

\theoremstyle{definition}

\numberwithin{equation}{section}
\numberwithin{figure}{section}
\numberwithin{table}{section}

\renewcommand{\labelenumi}{\textup{(\theenumi) }}
\renewcommand{\theenumi}{\arabic{enumi}}

\title[Symmetric pairs and compact symmetric triads]
{A duality between non-compact semisimple symmetric pairs and commutative compact semisimple symmetric triads and its general theory}
\author[K. Baba, O. Ikawa, A. Sasaki]{Kurando Baba, Osamu Ikawa, Atsumu SASAKI}

\thanks{
	The third author 
	was partially supported %by JSPS KAKENHI Grant Number JP16K05128 and the 
	by JSPS KAKENHI Grant Number JP19K03453. 
}


\address[K. Baba]{Faculty of Science and Technology, 
Tokyo University of Science, 
Noda, Chiba, 278-8510, Japan. }
\email{kurando.baba@rs.tus.ac.jp}

\address[O. Ikawa]{Faculty of Arts and Sciences, Kyoto Institute of Technology, 
Matsugasaki, Sakyoku, Kyoto 606-8585, Japan. }
\email{ikawa@kit.ac.jp}

\address[A. Sasaki]{Department of Mathematics, Faculty of Science, Tokai University, 
4-1-1, Kitakaname, Hiratsuka, Kanagawa, 259-1292, Japan, 
and Department of Mathematics, FAU Erlangen-N\"urnberg, 
Cauerstrasse 11, 91058, Erlangen, Germany}
\email{atsumu@tokai-u.jp}

\date{\today}
\begin{document}

\begin{abstract}
The present paper investigates a natural generalization of the duality 
between Riemannian symmetric pairs of compact type and those of non-compact type \`a la \'E. Cartan. 
%to non-compact pseudo-Riemannian semisimple symmetric pairs. 
The main result of this paper is 
to construct an explicit description of a one-to-one correspondence 
between non-compact pseudo-Riemannian semisimple symmetric pairs 
and commutative compact semisimple symmetric triads, which is called the duality theorem. 
Further, we develop a general theory of the duality theorem. 
%This can be applied to giving an alternative proof of Berger's classification of 
%non-compact pseudo-Riemannian symmetric pairs 
%and to orbit geometry of the action of a maximal compact subgroup of the isometry group 
%of a pseudo-Riemannian symmetric space, 
%which will be explained in the forthcoming papers \cite{bis1} and \cite{bis2}. 
\end{abstract}

\maketitle

\tableofcontents

%%%%%%%%%%%%%%%%%%%%%%%%%%%%%%%%%%%%%%%%%%%%%%%%%%%%%%%%%%%%%%%%%%%%%%%%%%%%%%%%%%%%%%%%%%%

\section{Introduction}
\label{sec:intro}

In Riemannian geometry, 
it is a well-known fact due to \'E. Cartan that 
there is a one-to-one correspondence between the set of non-compact Riemannian symmetric spaces 
and the set of locally isomorphism classes of compact Riemannian symmetric spaces. 
This means that there exists a one-to-one correspondence between 
the set of Riemannian symmetric pairs of non-compact type and that of compact type 
on the Lie algebra level. 
This correspondence is usually called the {\it duality} for Riemannian symmetric pairs. 
The aim of this paper is to give a one-to-one correspondence between 
\begin{center}
non-compact semisimple symmetric pairs
\end{center}
and 
\begin{center}
commutative compact semisimple symmetric triads 
\end{center}
as a natural generalization of the duality for Riemannian symmetric pairs 
(see Theorem \ref{thm:duality-thm} for its description) 
and develop its general theory. 
We say that this correspondence is the {\it duality theorem}. 

One of our motivations for the study on the duality theorem 
is to investigate the geometry of isometric group actions on pseudo-Riemannian symmetric spaces. 
In fact, 
the action of a maximal compact subgroup $K$ of the isometry group $G$ 
of a non-compact pseudo-Riemannian symmetric space $G/H$ corresponds to 
the $K$-action on the compact Riemannian symmetric space $G_u/H_u$ 
as an extension of the duality theorem. 
Thanks to this perspective, 
we can show in \cite[Proposition 2]{bis0} 
that each orbit for the $K$-action on $G/H$ is a pseudo-Riemannian submanifold 
and this action admits a flat totally geodesic submanifold as a cross-section. 
These properties could be interpreted 
that they correspond to the properties that 
each orbit for the $K$-action on $G_u/H_u$ is a Riemannian submanifold 
and this admits a flat totally geodesic submanifold as a cross-section. 
%The key machinery is that 
%there is a natural bijection between two orbit spaces $K\backslash G/H$ and $K\backslash G_u/H_u$. 
%From this point of view, 
Furthermore, we can investigate precisely geometric properties 
of each orbit for the $K$-action on $G/H$ via the duality theorem.
For example, we will specify all totally geodesic orbits for this action in the forthcoming paper \cite{bis2}. 
We note that there is literature on the study of Lie group actions 
on Riemannian symmetric spaces using the duality \`a la Cartan 
(see \cite{koike,kollross}). 
%, 
%however our correspondence can deal with all Lie group actions simultaneously on 
%not only Riemannian symmetric spaces but also pseudo-Riemannian symmetric spaces. 

Another motivation of our work is that 
our duality theorem plays a crucial role to study calibrated geometry. 
Harvey--Lawson, 
who have introduced the notion of calibrations on Riemannian manifolds \cite{harvey-lawson}, 
constructed calibrations using various geometric structures for some Riemannian manifolds. 
On the other hand, 
Mealy, 
who has introduced the notion of calibrations on pseudo-Riemannian manifolds \cite{mealy}, 
constructed calibrations using various geometric structures for some pseudo-Riemannian manifolds. 
Due to our duality theorem, 
we can understand the relationship 
between the calibrations on Riemannian manifolds and those on pseudo-Riemannian manifolds. 
Namely, 
we can regard the calibration on a Riemannian manifold given explicitly by Harvey--Lawson 
as an element of the space of functions on a certain compact symmetric space 
which are invariant under the action of the Lie group 
preserving some geometric structures, 
and the calibration on a pseudo-Riemannian manifold given explicitly by Mealy 
as an element of the space of functions on a certain pseudo-Riemannian symmetric space 
which are invariant under the action of the Lie group preserving some geometric structures, 
and then our duality theorem gives rise to a correspondence between two spaces of functions. 
%
%Namely, 
%we can {\color{red}immerse the space of} the calibrations on a Riemannian manifold given explicitly by Harvey--Lawson 
%{\color{red}into the space of} functions on a certain compact symmetric space which are invariant under the action of the Lie group 
%preserving some geometric structures, 
%and {\color{red}space of} the calibrations on a pseudo-Riemannian manifold given explicitly by Mealy 
%{\color{red}into the space of} functions on a certain pseudo-Riemannian symmetric space 
%which are invariant under the action of the Lie group preserving some geometric structures, 
%and then our duality theorem gives rise to a correspondence between {\color{red}two spaces of functions}. 

Now, 
let us give a quick review on this paper as follows. 
%the definition of non-compact semisimple symmetric pairs and 
%that of commutative compact semisimple symmetric triads. 
A non-compact semisimple symmetric pair is a pair $(\mathfrak{g}_0,\sigma )$ 
which consists of a non-compact real semisimple Lie algebra $\mathfrak{g}_0$ 
and an automorphism $\sigma $ on $\mathfrak{g}_0$ satisfying $\sigma ^2=\operatorname{id}$, 
that is, an involutive automorphism. 
If $\sigma$ coincides with a Cartan involution $\theta $ of $\mathfrak{g}_0$, 
the pair $(\mathfrak{g}_0,\theta )$ is usually said to be a non-compact Riemannian symmetric pair. 
On the other hand, 
a commutative compact semisimple symmetric triad is a triplet $(\mathfrak{g},\theta _1,\theta _2)$ 
with a compact semisimple Lie algebra $\mathfrak{g}$ 
and two involutive automorphisms $\theta _1,\theta _2$ on $\mathfrak{g}$ 
satisfying $\theta _1\theta _2=\theta _2\theta _1$. 
Clearly, $(\mathfrak{g},\theta ,\theta )$ is a commutative compact semisimple symmetric triad 
for any involutive automorphism $\theta$, 
and it seems to be a compact Riemannian symmetric pair $(\mathfrak{g},\theta )$. 

The present paper provides an explicit description of a one-to-one correspondence between 
the set of equivalence classes in non-compact semisimple symmetric pairs 
and the set of equivalence classes in commutative compact semisimple symmetric triads 
(the duality theorem, Theorem \ref{thm:duality-thm}). 
%We shall say that this correspondence is the {\it duality theorem}. 
If $(\mathfrak{g}_0,\sigma )$ corresponds to $(\mathfrak{g},\theta _1,\theta _2)$ via our duality theorem, 
we shall call $(\mathfrak{g},\theta _1,\theta _2)$ (resp. $(\mathfrak{g}_0,\sigma )$) the dual 
of $(\mathfrak{g}_0,\sigma )$ (resp. $(\mathfrak{g},\theta _1,\theta _2)$) 
and use the notation $(\mathfrak{g},\theta _1,\theta _2)=(\mathfrak{g}_0,\sigma )^*$ 
and $(\mathfrak{g}_0,\sigma )=(\mathfrak{g},\theta _1,\theta _2)^*$ (see Notation \ref{notation:*}). 

Our duality theorem might be a natural generalization of the duality for Riemannian symmetric pairs. 
Namely, 
when $\sigma $ is a Cartan involution $\theta $ of $\mathfrak{g}_0$, 
the corresponding commutative compact semisimple symmetric triad is $(\mathfrak{g},\theta ,\theta )$ 
and then essentially the same as a compact Riemannian symmetric pair $(\mathfrak{g},\theta )$. 
Our perspective is explained in Section \ref{subsec:riemannian}. 

We would emphasize that 
our duality theorem is compatible with the fundamental properties on non-compact semisimple symmetric pairs 
and with those on commutative compact semisimple symmetric triads. 

First, there is a notion of the associated symmetric pair $(\mathfrak{g}_0,\sigma )^a$ 
(Definition \ref{def:associated}) 
and the dual symmetric pair $(\mathfrak{g}_0,\sigma )^d$ (Definition \ref{def:dual}) 
of a non-compact semisimple symmetric pair $(\mathfrak{g}_0,\sigma )$, 
whereas, there is a notion of the associated symmetric triad $(\mathfrak{g},\theta _1,\theta _2 )^a$ 
(Definition \ref{def:associated-triad}) and 
the dual symmetric triad $(\mathfrak{g},\theta _1,\theta _2)^d$ 
(Definition \ref{def:dual-triad}) of 
a commutative compact semisimple symmetric triad $(\mathfrak{g},\theta _1,\theta _2 )$. 
Then, our duality theorem satisfies the compatible conditions 
on these notions in the sense of Propositions \ref{prop:compatible-c} and \ref{prop:compatible-n}. 
For example, the dual of $(\mathfrak{g}_0,\sigma )^{a}$, 
denoted by $(\mathfrak{g}_0,\sigma )^{a*}$, equals 
the associated symmetric triad of $(\mathfrak{g}_0,\sigma )^{*}$, %of $(\mathfrak{g}_0,\sigma )$, 
denoted by $(\mathfrak{g}_0,\sigma )^{*a}$. 

Second, our duality theorem preserves the property of irreducibility, namely, 
an irreducible non-compact semisimple symmetric pair (Definition \ref{def:irr-pair}) 
corresponds to an irreducible commutative compact semisimple symmetric triad 
(Definition \ref{def:irr-triad}) and vice versa (see Theorem \ref{thm:corresp-irreducible}). 
Further, 
any irreducible non-compact semisimple symmetric pair belongs to 
either a non-compact simple symmetric pair without complex structures or one of four %(abstract) 
subclasses 
(we shall use the labels (P-\ref{item:cpx-anti})--(P-\ref{item:non-linear}), see Corollary \ref{cor:irr-pair}), 
and any irreducible commutative compact semisimple symmetric triad belongs to 
either a compact simple symmetric triad or one of four %(abstract) 
subclasses 
(the labels (T-\ref{type:i2})--(T-\ref{type:ii2d}), see Proposition \ref{prop:irr-triad}). 
Then, our duality theorem induces a one-to-one correspondence 
between non-compact simple symmetric pairs without complex structures 
and commutative compact simple symmetric triads (see Proposition \ref{prop:simple}) 
and that between the subclasses (P-\ref{item:cpx-anti})--(P-\ref{item:non-linear}) 
and the subclasses (T-\ref{type:i2})--(T-\ref{type:ii2d}) 
%Then, our duality theorem can specify the `type' of an irreducible $(\mathfrak{g}_0,\sigma )$ 
%by the `type' of the dual $(\mathfrak{g},\theta _1,\theta _2)=(\mathfrak{g}_0,\sigma )^*$ 
%and vise versa 
(see Theorem \ref{thm:irreducible}). 
We remark that 
our definition for $(\mathfrak{g}_0,\sigma )$ to be irreducible might be different from 
the definition given by \cite[p.435]{oshima-sekiguchi}. 
The properness of our definition and the comparison of two definitions are explained 
in Section \ref{sec:appendix} as an appendix. 

Third, 
we provide a new characterization for a non-compact semisimple symmetric pair $(\mathfrak{g}_0,\sigma )$ 
to be of type $K_{\varepsilon }$ 
(see \cite[Proposition 2.1]{kaneyuki96}, also Definition \ref{def:typeK}) 
by the corresponding commutative compact semisimple symmetric triad 
$(\mathfrak{g},\theta _1,\theta _2)=(\mathfrak{g}_0,\sigma )^*$ via our duality theorem, 
namely, 
$\theta _1$ is conjugate to $\theta _2$ by inner automorphisms (see Theorem \ref{thm:typeK}). 

Here is a remark that 
we can offer an alternative proof for the classification 
of non-compact semisimple symmetric pairs, 
of which a concrete description %of the classification of non-compact semisimple symmetric pairs 
was given by Berger \cite{berger}, 
from the viewpoint that 
any non-compact semisimple symmetric pair is the dual of some commutative compact semisimple symmetric triad. 
An explicit correspondence has been given in \cite[Table 1]{bis0}. 
The significance of our proof is to bring the classification of compact semisimple symmetric pairs 
%which consists of 19 series of pairs 
(\cite[Table V in p.518]{helgason}) 
to the classification of non-compact semisimple symmetric pairs 
%which includes at least 150 series of irreducible non-compact simple symmetric pairs 
(\cite[Tableau II in pp.157--161]{berger}). % via the duality theorem. 
%The detail will be explained in the next paper \cite{bis1}, 
For details, we refer to the next paper \cite{bis1}, 
and we summarize the mechanism of our idea as follows. 
The classification of compact semisimple symmetric triads 
is gained in view of compact semisimple symmetric pairs. 
Further, the classification of non-compact semisimple symmetric pairs is derived 
from the former via the duality theorem. 
What bridges between the classification of compact semisimple symmetric triads 
and that of non-compact semisimple symmetric pairs is 
the notion of symmetric triads with multiplicities 
whose classification is accomplished by the second author (cf. \cite{ikawa-jms}). 
%which we will give an explanation of the mechanism for our method. 
We note that we find another approach to Berger's classification by Helminck \cite{helminck}, 
however in a completely different way. 

From this kind of circumstances, 
this paper is the foundation of the studies mentioned above. 

%Based on the above background, 
%this paper will also develop a general theory of the duality theorem 
%between non-compact semisimple symmetric pairs and commutative compact semisimple symmetric triads. 

%%%%%%%%%%%%%%%%%%%%%%%%%%%%%%%%%%%%%%%%%%%%%%%%%%%%%%%%%%%%%%%%%%%%%%%%%%%%%%%%%%%%%%%%%%

\section{Preliminaries}
\label{sec:preliminaries}

This section provides a quick review on non-compact semisimple symmetric pairs 
and commutative compact semisimple symmetric triads. 
We refer to the references \cite{berger,conlon1,ikawa-jms}, for example. 

\subsection{Pseudo-Riemannian semisimple symmetric pair}
\label{subsec:pair}

We begin with a summary on non-compact semisimple symmetric pairs. 

Let $\mathfrak{g}_0$ be a non-compact real semisimple Lie algebra. 
Suppose we are given an involutive automorphism (involution, for short) $\sigma $ on $\mathfrak{g}_0$. 
The pair $(\mathfrak{g}_0,\sigma )$ is called a {\it non-compact semisimple symmetric pair}. 
We denote by $\mathfrak{g}_0^{\sigma}:=\{ X\in \mathfrak{g}_0:\sigma (X)=X\} $ 
the fixed point set of $\sigma$ in $\mathfrak{g}_0$. 
Then, $\mathfrak{g}_0^{\sigma }$ is a subalgebra of $\mathfrak{g}_0$. 
We note that $\mathfrak{g}_0^{\sigma}$ is a maximal compact subalgebra if and only if 
$\sigma$ is a Cartan involution of $\mathfrak{g}_0$. 
In this paper, 
we say that a non-compact semisimple symmetric pair $(\mathfrak{g}_0,\sigma)$ is {\it Riemannian} 
if $\sigma$ is a Cartan involution of $\mathfrak{g}_0$, 
and {\it pseudo-Riemannian} semisimple symmetric pair if not. 

The terminology of non-compact (pseudo-)Riemannian semisimple symmetric pairs 
comes from the corresponding (pseudo-)Riemannian symmetric spaces 
in the sense as follows. 
Let $G$ be a non-compact connected real semisimple Lie group with finite center 
whose Lie algebra is $\mathfrak{g}_0$. 
For a Cartan involution $\theta$ of $G$, 
the fixed point set $K=G^{\theta}:=\{ g\in G:\theta (g)=g\} $ %of $\theta$ in $G$ 
is a maximal compact subgroup of $G$ (\cite[Chap. VI, Theorem 1.1]{helgason}), 
and then $G/K$ has a structure of a Riemannian symmetric space of non-compact type. 
On the other hand, 
for a non-trivial involution $\sigma$ which is different from Cartan involutions of $G$, 
the fixed point set $H:=G^{\sigma }$ is non-compact. 
Then, $G/H$ has a structure of a pseudo-Riemannian symmetric space. 

We introduce an equivalence relation on non-compact semisimple symmetric pairs as follows. 
Two non-compact semisimple symmetric pairs $(\mathfrak{g}_0,\sigma )$, $(\mathfrak{g}_0',\sigma ')$ 
satisfy $(\mathfrak{g}_0,\sigma )\equiv (\mathfrak{g}_0',\sigma ')$ 
if there exists a Lie algebra isomorphism $\varphi :\mathfrak{g}_0\to \mathfrak{g}_0'$ 
such that $\varphi \sigma =\sigma '\varphi $. 
Non-compact semisimple symmetric pairs are classified by Berger \cite{berger} 
up to the equivalence relation. 

\subsection{Pseudo-Riemannian semisimple symmetric pair equipped with Cartan involution}
\label{subsec:pair-cartan}

Let $(\mathfrak{g}_0,\sigma )$ be a non-compact semisimple symmetric pair. 
It is known by \cite{berger} that 
there exists a Cartan involution $\theta $ of $\mathfrak{g}_0$ commuting with $\sigma$ 
(cf. \cite[Theorem 6.16]{knapp}). 

\begin{define}
\label{def:cartan}
We call such a triplet $(\mathfrak{g}_0,\sigma ;\theta )$ 
a {\it non-compact semisimple symmetric pair equipped with a Cartan involution}, 
and say that it is Riemannian (resp. pseudo-Riemannian) 
if $(\mathfrak{g}_0,\sigma )$ is Riemannian (resp. pseudo-Riemannian). 
\end{define}

Two non-compact semisimple symmetric pairs equipped with Cartan involutions 
$(\mathfrak{g}_0,\sigma ;\theta ),(\mathfrak{g}_0',\sigma ';\theta ')$ satisfy 
$(\mathfrak{g}_0,\sigma ;\theta )\equiv (\mathfrak{g}_0',\sigma ';\theta ')$ 
if there exists a Lie algebra isomorphism $\varphi :\mathfrak{g}_0\to \mathfrak{g}_0'$ 
such that $\varphi \sigma =\sigma '\varphi $ and $\varphi \theta =\theta '\varphi $. 
Clearly, $(\mathfrak{g}_0,\sigma )\equiv (\mathfrak{g}_0',\sigma ')$ holds 
if $(\mathfrak{g}_0,\sigma ;\theta )\equiv (\mathfrak{g}_0',\sigma ';\theta ')$. 

Given a non-compact semisimple symmetric pair equipped with a Cartan involution, 
we can construct other ones, namely, the associated and the dual non-compact semisimple symmetric pairs 
equipped with Cartan involutions of it, 
whose notions are based on those \`a la Berger \cite{berger} 
(see also \cite[Definition 1.3]{oshima-sekiguchi}). 

Let $(\mathfrak{g}_0,\sigma ;\theta )$ be a non-compact semisimple symmetric pair 
equipped with a Cartan involution. 
We set $\mathfrak{k}_0:=\mathfrak{g}_0^{\theta }$ and $\mathfrak{p}_0:=\mathfrak{g}_0^{-\theta }$. 
Then, 
$\mathfrak{g}_0=\mathfrak{k}_0+\mathfrak{p}_0$ is the corresponding Cartan decomposition. 
Similarly, 
we put $\mathfrak{h}_0:=\mathfrak{g}_0^{\sigma }$ and $\mathfrak{q}_0:=\mathfrak{g}_0^{-\sigma}$. 
Then, 
$\mathfrak{g}_0=\mathfrak{h}_0+\mathfrak{q}_0$ is the $\sigma$-eigenspace decomposition. 
As $\theta \sigma =\sigma \theta $, 
the Lie algebra $\mathfrak{g}_0$ is decomposed into the direct sum as follows: 
\begin{align*}
\mathfrak{g}_0
&=\mathfrak{k}_0\cap \mathfrak{h}_0+\mathfrak{k}_0\cap \mathfrak{q}_0
+\mathfrak{p}_0\cap \mathfrak{h}_0+\mathfrak{p}_0\cap \mathfrak{q}_0. 
\end{align*}
%In particular, we obtain 
%\begin{align*}
%\mathfrak{g}_0^{\theta ,\sigma }=\mathfrak{k}_0\cap \mathfrak{h}_0,\ 
%\mathfrak{g}_0^{\theta ,-\sigma }=\mathfrak{k}_0\cap \mathfrak{q}_0,\ 
%\mathfrak{g}_0^{-\theta ,\sigma }=\mathfrak{p}_0\cap \mathfrak{h}_0,\ 
%\mathfrak{g}_0^{-\theta ,-\sigma }=\mathfrak{p}_0\cap \mathfrak{q}_0. 
%\end{align*}
%Here, we shall use the notation for two maps $\nu ,\nu '$ on $\mathfrak{g}_0$ as 
%\begin{align*}
%\mathfrak{g}_0^{\nu ,\nu '}=\{ X\in \mathfrak{g}_0:\nu (X)=\nu '(X)=X\} . 
%\end{align*}

The commutativity $\theta \sigma =\sigma \theta$ 
gives rise to another involution $\theta \sigma$ on $\mathfrak{g}_0$. 
Clearly, $\theta \sigma$ commutes with $\theta$, 
from which we obtain another non-compact semisimple symmetric pair equipped with a Cartan involution 
$(\mathfrak{g}_0,\theta \sigma ;\theta )$. 

Cartan involutions of $\mathfrak{g}_0$ commuting with $\sigma$ are unique up to conjugations 
in the sense as follows. 

\begin{fact}[{\cite[p.153]{loos}}]
\label{fact:loos}
For any two Cartan involutions $\theta ,\theta '$ of $\mathfrak{g}_0$ commuting with $\sigma$, 
there exists $X\in \mathfrak{g}_0^{\sigma }$ such that 
$\theta '=e^{\operatorname{ad}X}\theta e^{-\operatorname{ad}X}$. 
\end{fact}

Then, two involutions $\theta \sigma$ and $\theta '\sigma$ are conjugate in the following sense: 

\begin{lemma}
\label{lem:conjugate-associated}
Retain the setting of Fact \ref{fact:loos}. 
Then, we have $\sigma e^{\operatorname{ad}X}=e^{\operatorname{ad}X}\sigma $ and 
$(\theta '\sigma )e^{\operatorname{ad}X}=e^{\operatorname{ad}X}(\theta \sigma )$. 
\end{lemma}

\begin{proof}
As 
% $\sigma ^2=\operatorname{id}$ and 
$\sigma (X)=X$, we have 
$(\operatorname{ad}X)\sigma (Y)=\sigma (\operatorname{ad}X)Y$ for any $Y\in \mathfrak{g}_0$. 
This means $(\operatorname{ad}X)\sigma =\sigma (\operatorname{ad}X)$. 

Let us observe the one-parameter transformation group 
$\{ \sigma e^{t\operatorname{ad}X}\sigma ^{-1}:t\in \mathbb{R}\}$. 
Since the following relation holds: 
\begin{align*}
\left. \frac{d}{dt}\sigma e^{t\operatorname{ad}X}\sigma ^{-1}\right| _{t=0}
=\sigma (\operatorname{ad}X)\sigma ^{-1}=\operatorname{ad}X, 
\end{align*}
we obtain 
\begin{align}
\sigma e^{t\operatorname{ad}X}\sigma ^{-1}=e^{t\operatorname{ad}X}. 
\label{eq:commute sigma}
\end{align}
In particular, $\sigma e^{\operatorname{ad}X}=e^{\operatorname{ad}X}\sigma$. 
Hence, we obtain 
\begin{align*}
(\theta '\sigma )e^{\operatorname{ad}X}=\theta 'e^{\operatorname{ad}X}\sigma 
=e^{\operatorname{ad}X}(\theta \sigma ). 
\end{align*}
Therefore, Lemma \ref{lem:conjugate-associated} has been proved. 
\end{proof}

Lemma \ref{lem:conjugate-associated} explains that 
$(\mathfrak{g}_0,\sigma ;\theta )\equiv (\mathfrak{g}_0,\sigma ;\theta ')$ 
and $(\mathfrak{g}_0,\theta \sigma ;\theta )\equiv (\mathfrak{g}_0,\theta '\sigma ;\theta ')$ 
for any Cartan involutions $\theta ,\theta '$ of $\mathfrak{g}_0$ commuting with $\sigma$ 
by the automorphism $\mathfrak{g}_0\to \mathfrak{g}_0,~Y\mapsto e^{\operatorname{ad}X}Y$. 

\begin{define}[cf. {\cite{berger}}]
\label{def:associated}
The triplet $(\mathfrak{g}_0,\theta \sigma ;\theta )$ is called 
the {\it associated} symmetric pair equipped with a Cartan involution 
of $(\mathfrak{g}_0,\sigma ;\theta)$. 
If $(\mathfrak{g}_0,\theta \sigma ;\theta )\equiv (\mathfrak{g}_0,\sigma ;\theta )$, 
then $(\mathfrak{g}_0,\sigma ;\theta )$ is said to be {\it self-associated}. 
\end{define}

In this context, we write 
\begin{align}
(\mathfrak{g}_0,\sigma ;\theta )^a:=(\mathfrak{g}_0,\theta \sigma ;\theta ). 
\label{eq:associated noncompact}
\end{align}

\begin{lemma}
\label{lem:associated}
If $(\mathfrak{g}_0,\sigma ;\theta )\equiv (\mathfrak{g}_0',\sigma ';\theta ')$, 
then $(\mathfrak{g}_0,\sigma ;\theta )^a\equiv (\mathfrak{g}_0',\sigma ';\theta ')^a$. 
\end{lemma}

\begin{proof}
Suppose $(\mathfrak{g}_0,\sigma ;\theta )\equiv (\mathfrak{g}_0',\sigma ';\theta ')$. 
We take a Lie algebra isomorphism $\varphi :\mathfrak{g}_0\to \mathfrak{g}_0'$ 
such that $\varphi \sigma =\sigma '\varphi $ and $\varphi \theta =\theta '\varphi $. 
Then, we have $\varphi (\theta \sigma )=\theta '\varphi \sigma =(\theta '\sigma ')\varphi $. 
This implies $(\mathfrak{g}_0,\theta \sigma ;\theta )\equiv (\mathfrak{g}_0',\theta '\sigma ';\theta ')$. 
\end{proof}

Let $\mathfrak{g}_{\mathbb{C}}=\mathfrak{g}_0+\sqrt{-1}\mathfrak{g}_0$ 
be the complexification of $\mathfrak{g}_0$. 
Then, $\mathfrak{g}_{\mathbb{C}}$ is a complex semisimple Lie algebra 
and its subalgebra 
\begin{align}
\mathfrak{g}:=\mathfrak{k}_0+\sqrt{-1}\mathfrak{p}_0
=\mathfrak{g}_0^{\theta }+\sqrt{-1}\mathfrak{g}_0^{-\theta } 
\label{eq:compact real form}
\end{align}
is a compact real form of $\mathfrak{g}_{\mathbb{C}}$. 
We denote by $\mu$ the complex conjugation of $\mathfrak{g}_{\mathbb{C}}$ with respect to $\mathfrak{g}$. 
We extend $\theta $ and $\sigma$ 
to $\mathbb{C}$-linear involutions on $\mathfrak{g}_{\mathbb{C}}$ 
which we use the same letters $\theta $ and $\sigma$ to denote, respectively. 
Clearly, $\mu$ commutes with $\sigma $, 
which gives rise to another involution $\mu \sigma $ on $\mathfrak{g}_{\mathbb{C}}$. 
Then, $\mu \sigma $ is anti-linear because 
$\mu$ is anti-linear and $\sigma $ is $\mathbb{C}$-linear. 
Hence, we obtain a new real form of $\mathfrak{g}_{\mathbb{C}}$ as follows: 
\begin{align}
\label{eq:dual}
\mathfrak{g}_{\mathbb{C}}^{\mu \sigma }
&=\mathfrak{k}_0\cap \mathfrak{h}_0+\sqrt{-1}(\mathfrak{k}_0\cap \mathfrak{q}_0)
+\sqrt{-1}(\mathfrak{p}_0\cap \mathfrak{h}_0)+\mathfrak{p}_0\cap \mathfrak{q}_0. 
\end{align}
This $\mathfrak{g}_{\mathbb{C}}^{\mu \sigma }$ is a non-compact real semisimple Lie algebra 
if $\sigma $ is not the identity map. 
The restriction of $\theta $ to $\mathfrak{g}_{\mathbb{C}}^{\mu \sigma }$ defines an involution 
on $\mathfrak{g}_{\mathbb{C}}^{\mu \sigma }$. 

We remark that the subalgebra 
\begin{align}
\mathfrak{k}_0^d:=\mathfrak{k}_0\cap \mathfrak{h}_0+\sqrt{-1}(\mathfrak{p}_0\cap \mathfrak{h}_0)
\label{eq:maximal-dual}
\end{align}
is a maximal compact subalgebra of $\mathfrak{g}_{\mathbb{C}}^{\mu \sigma }$. 

\begin{lemma}
\label{lem:dual-cartan}
The restriction of $\sigma $ to $\mathfrak{g}_{\mathbb{C}}^{\mu \sigma }$ 
becomes a Cartan involution of $\mathfrak{g}_{\mathbb{C}}^{\mu \sigma }$. 
In particular, we obtain $(\mathfrak{g}_{\mathbb{C}}^{\mu \sigma })^{\sigma }=\mathfrak{k}_0^d$. 
\end{lemma}

Lemma \ref{lem:dual-cartan} implies 
a new non-compact semisimple symmetric pair equipped with a Cartan involution 
$(\mathfrak{g}_{\mathbb{C}}^{\mu \sigma },\theta ;\sigma)$. 

For another Cartan involution $\theta '$ of $\mathfrak{g}_0$ commuting with $\sigma$, 
let $\mu '$ be the complex conjugation of $\mathfrak{g}_{\mathbb{C}}$ 
with respect to $\mathfrak{g}':=\mathfrak{g}_0^{\theta '}+\sqrt{-1}\mathfrak{g}_0^{-\theta '}$. 
Then, we obtain 
a non-compact semisimple symmetric pair equipped with a Cartan involution 
$(\mathfrak{g}_{\mathbb{C}}^{\mu '\sigma },\theta ';\sigma )$. 
It follows from Fact \ref{fact:loos} that 
$\theta '=e^{\operatorname{ad}X}\theta e^{-\operatorname{ad}X}$ 
for some $X\in \mathfrak{g}_0^{\sigma }$. 
Then, $e^{-\operatorname{ad}X}$ defines an isomorphism 
from $\mathfrak{g}_0^{\theta '}$ to $\mathfrak{g}_0^{\theta }$ 
and that from $\mathfrak{g}_0^{-\theta '}$ to $\mathfrak{g}_0^{-\theta }$. 
Moreover, %as $e^{-\operatorname{ad}X}\sigma =\sigma e^{-\operatorname{ad}X}$ 
Lemma \ref{lem:conjugate-associated} implies that 
$e^{-\operatorname{ad}X}$ becomes automorphisms 
on both $\mathfrak{g}_0^{\sigma }$ and $\mathfrak{g}_0^{-\sigma }$. 
Hence, $\mathfrak{g}_{\mathbb{C}}^{\mu '\sigma }$ is isomorphic to $\mathfrak{g}_{\mathbb{C}}^{\mu \sigma }$ 
via $e^{-\operatorname{ad}X}$ and then $(\mathfrak{g}_{\mathbb{C}}^{\mu '\sigma },\theta ';\sigma )\equiv 
(\mathfrak{g}_{\mathbb{C}}^{\mu \sigma },\theta ;\sigma )$. 

\begin{define}[cf. {\cite{berger}}]
\label{def:dual}
The triplet $(\mathfrak{g}_{\mathbb{C}}^{\mu \sigma },\theta ;\sigma )$ is called 
the {\it dual} symmetric pair equipped with a Cartan involution of $(\mathfrak{g}_0,\sigma ;\theta )$. 
Further, $(\mathfrak{g}_0,\sigma ;\theta )$ is said to be {\it self-dual} 
if $(\mathfrak{g}_{\mathbb{C}}^{\mu \sigma },\theta ;\sigma )\equiv (\mathfrak{g}_0,\sigma ;\theta )$ holds. 
\end{define}

Hereafter, we write $\mathfrak{g}_0^d:=\mathfrak{g}_{\mathbb{C}}^{\mu \sigma }$ and 
\begin{align}
(\mathfrak{g}_0,\sigma ;\theta )^d:=(\mathfrak{g}_0^d,\theta ;\sigma ). 
\label{eq:dual noncompact}
\end{align}

\begin{lemma}
\label{lem:dual}
If $(\mathfrak{g}_0,\sigma ;\theta )\equiv (\mathfrak{g}_0',\sigma ';\theta ')$, 
then $(\mathfrak{g}_0,\sigma ;\theta )^d\equiv (\mathfrak{g}_0',\sigma ';\theta ')^d$. 
\end{lemma}

\begin{proof}
Let $\varphi :\mathfrak{g}_0\to \mathfrak{g}_0'$ be a Lie algebra isomorphism 
such that $\varphi \sigma =\sigma '\varphi $ and $\varphi \theta =\theta '\varphi $. 
We extend $\varphi $ to a $\mathbb{C}$-linear isomorphism 
from $\mathfrak{g}_{\mathbb{C}}=\mathfrak{g}_0+\sqrt{-1}\mathfrak{g}_0$ to 
$\mathfrak{g}_{\mathbb{C}}'=\mathfrak{g}_0'+\sqrt{-1}\mathfrak{g}_0'$. 
Let $\mu$ be the complex conjugation of $\mathfrak{g}_{\mathbb{C}}$ with respect to 
$\mathfrak{g}:=\mathfrak{g}_0^{\theta }+\sqrt{-1}\mathfrak{g}_0^{-\theta }$ 
and $\mu '$ be that of $\mathfrak{g}_{\mathbb{C}}'$ with respect to 
$\mathfrak{g}':=(\mathfrak{g}_0')^{\theta '}+\sqrt{-1}(\mathfrak{g}_0')^{-\theta '}$. 
As $\varphi \theta =\theta '\varphi $, 
we have $\varphi (\mathfrak{g}_0^{\pm \theta })=(\mathfrak{g}_0')^{\pm \theta '}$. 
This implies that $\varphi (\mathfrak{g}_{\mathbb{C}}^{\pm \mu })
=(\mathfrak{g}_{\mathbb{C}}')^{\pm \mu '}$. 
Thus, we obtain $\varphi \mu =\mu '\varphi $, 
from which $\varphi (\mu \sigma )=(\mu '\sigma ')\varphi $. 
Hence, we have verified that 
$\varphi :
\mathfrak{g}_{\mathbb{C}}^{\mu \sigma }\to (\mathfrak{g}_{\mathbb{C}}')^{\mu '\sigma '}$ 
is a Lie algebra isomorphism satisfying $\varphi \sigma =\sigma '\varphi $ 
and $\varphi \theta =\theta '\varphi$. 
As $\mathfrak{g}_0^d=\mathfrak{g}_{\mathbb{C}}^{\mu \sigma }$, 
we conclude 
$(\mathfrak{g}_0^d,\theta ;\sigma )\equiv ((\mathfrak{g}_0')^d,\theta ';\sigma ')$. 
\end{proof}

Now, let us give some examples of non-compact semisimple symmetric pairs equipped with Cartan involutions. 
In the following, we will use the notation as follows. 
Let $I_m$ be the unit matrix of degree $m$ and 
\begin{align}
\label{eq:J_m}
J_m=\left( 
	\begin{array}{cc}
	O & -I_m \\
	I_m & O
	\end{array}
\right) \in M(2m,\mathbb{R}). 
\end{align}
For positive integers $m,n$, 
we write 
\begin{align}
\label{eq:I_{m,n}}
I_{m,n}:=\left( 
	\begin{array}{cc}
	I_{m} & O \\
	O & -I_{n}
	\end{array}
\right) \in M(m+n,\mathbb{R})
\end{align}
and 
\begin{align}
\label{eq:J_{m,n}}
J_{m,n}:=\left( 
	\begin{array}{cc}
	J_m & O \\
	O & J_n
	\end{array}
\right) \in M(2m+2n,\mathbb{R}). 
\end{align}

\begin{eg}
\label{eg:so-u}
For positive integers $p,q$, 
let $(\mathfrak{g}_0,\sigma )$ be a non-compact semisimple symmetric pair as 
$\mathfrak{g}_0=\mathfrak{so}(2p,2q):=\{ X\in M(2p+2q,\mathbb{R}):{}^tXI_{2p,2q}+I_{2p,2q}X=O\} $ 
and $\sigma (X):=J_{p,q}XJ_{p,q}^{-1}$ ($X\in \mathfrak{g}_0$). 
Here, ${}^tX$ denotes the transposed matrix of $X$. 

The non-compact real simple Lie algebra $\mathfrak{so}(2p,2q)$ is of the form 
\begin{align}
\label{eq:so(2p,2q)}
\mathfrak{so}(2p,2q)
&=\left\{ 
	\left( 
		\begin{array}{cc}
		A & B \\
		{}^tB & D
		\end{array}
	\right) :
	\begin{array}{c}
	A\in \mathfrak{so}(2p),~D\in \mathfrak{so}(2q),\\
	B\in M(2p,2q\,;\mathbb{R})
	\end{array}
\right\} 
\end{align}
where $\mathfrak{so}(m):=\{ X\in M(m,\mathbb{R}):{}^tX=-X\} $ is a compact Lie algebra. 
In this case, the fixed point set $\mathfrak{g}_0^{\sigma }$ is given by 
\begin{align*}
\mathfrak{g}_0^{\sigma }
&=\left\{ 
	\left(
		\begin{array}{cccc}
		A_{1} & -A_{2} & B_1 & -B_2 \\
		A_2 & A_1 & B_2 & B_1 \\
		{}^tB_1 & {}^t B_2 & D_1 & -D_2 \\
		-{}^tB_2 & {}^tB_1 & D_2 & D_1
		\end{array}
	\right) :
	\begin{array}{c}
	A_1\in \mathfrak{so}(p),A_2\in \operatorname{Sym}(p,\mathbb{R}),\\
	D_1\in \mathfrak{so}(q),D_2\in \operatorname{Sym}(q,\mathbb{R}),\\
	B_1,B_2\in M(p,q\,;\mathbb{R}) 
	\end{array}
\right\} 
\end{align*}
where $\operatorname{Sym}(m,\mathbb{R}):=\{ X\in M(m,\mathbb{R}):{}^tX=X\} $ 
is a vector space over $\mathbb{R}$. 
Then, $\mathfrak{g}_0^{\sigma }$ is isomorphic to the non-compact reductive Lie algebra 
$\mathfrak{u}(p,q):=\{ X\in M(p+q,\mathbb{C}):{}^t\overline{X}I_{p,q}+I_{p,q}X=O\} $ 
by the Lie algebra isomorphism $\mathfrak{g}_0^{\sigma }\stackrel{\sim }{\to }\mathfrak{u}(p,q)$ given by 
\begin{align*}
\left(
	\begin{array}{cccc}
	A_{1} & -A_{2} & B_1 & -B_2 \\
	A_2 & A_1 & B_2 & B_1 \\
	{}^tB_1 & {}^t B_2 & D_1 & -D_2 \\
	-{}^tB_2 & {}^tB_1 & D_2 & D_1
	\end{array}
\right) \mapsto 
\left( 
	\begin{array}{cc}
	A_1+\sqrt{-1}A_2 & B_1+\sqrt{-1}B_2 \\
	{}^tB_1-\sqrt{-1}\,{}^tB_2 & D_1+\sqrt{-1}D_2
	\end{array}
\right) . 
\end{align*}

We take an involution $\theta $ on $\mathfrak{g}_0$ as $\theta (X)=I_{2p,2q}XI_{2p,2q}$ 
($X\in \mathfrak{g}_0$). 
Then, the fixed point set $\mathfrak{g}_0^{\theta }$ is of the form 
\begin{align}
\label{eq:so(2p)+so(2q)}
\mathfrak{g}_0^{\theta }
&=\left\{ 
	\left( 
		\begin{array}{cc}
		A & O \\
		O & B 
		\end{array}
	\right) :
	A\in \mathfrak{so}(2p),~B\in \mathfrak{so}(2q)
\right\} . 
\end{align}
This implies that $\mathfrak{g}_0^{\theta }=\mathfrak{so}(2p)+\mathfrak{so}(2q)$ 
is a maximal compact subalgebra of $\mathfrak{g}_0$, 
and then $\theta $ is a Cartan involution of $\mathfrak{g}_0$. 
As $I_{2p,2q}J_{p,q}=J_{p,q}I_{2p,2q}$, 
the Cartan involution $\theta $ commutes with $\sigma$. 
Hence, $(\mathfrak{g}_0,\sigma ;\theta )$ is a non-compact semisimple symmetric pair 
equipped with a Cartan involution. 
\end{eg}

The next example uses the matrix $J_{2p}'$ defined by 
\begin{align}
\label{eq:J_{2p}'}
J_{2p}':=\left( 
	\begin{array}{cc}
	O & I_{2p} \\
	I_{2p} & O
	\end{array}
\right) . 
\end{align}

\begin{eg}
\label{eg:so-gl}
Let $p$ be a positive integer. 
We set $\mathfrak{g}_0:=\mathfrak{so}(2p,2p)$ and an involution $\sigma $ on $\mathfrak{g}_0$ as 
$\sigma (X)=J_{2p}'XJ_{2p}'$ ($X\in \mathfrak{g}_0$). 
Under the realization of $\mathfrak{so}(2p,2p)$ as in (\ref{eq:so(2p,2q)}), 
the fixed point set $\mathfrak{g}_0^{\sigma }$ forms 
\begin{align}
\label{eq:gl(2p)}
\mathfrak{g}_0^{\sigma }
=\left\{ 
	\left( 
		\begin{array}{cc}
		A & B \\
		B & A
		\end{array}
	\right) :A\in \mathfrak{so}(2p),B\in \operatorname{Sym}(2p,\mathbb{R})
\right\} . 
\end{align}
This is isomorphic to $\mathfrak{gl}(2p,\mathbb{R})$ by Lie algebra isomorphism 
\begin{align}
\label{eq:gl(2p)-isom}
\iota :\mathfrak{g}_0^{\sigma }\stackrel{\sim}{\to }\mathfrak{gl}(2p,\mathbb{R}),\quad 
\left( 
	\begin{array}{cc}
	A & B \\
	B & A
	\end{array}
\right) \mapsto & A+B. 
\end{align}

We take an involution $\theta $ on $\mathfrak{g}_0$ as 
$\theta (X)=I_{2p,2p}XI_{2p,2p}$ ($X\in \mathfrak{g}_0$). 
%Since $\mathfrak{g}^{\theta }=\mathfrak{so}(2p)+\mathfrak{so}(2p)$ 
%is a maximal compact subalgebra of $\mathfrak{g}_0$, 
Then, $\theta $ is a Cartan involution of $\mathfrak{g}_0$ and $(\mathfrak{g}_0,\sigma ;\theta )$ 
is a non-compact semisimple symmetric pair equipped with a Cartan involution. 
\end{eg}

Here, let us compare the non-compact semisimple symmetric pair equipped with a Cartan involution 
$(\mathfrak{g}_0,\sigma _1;\theta )$ discussed in Example \ref{eg:so-u} in case of $p=q$ 
with $(\mathfrak{g}_0,\sigma _2;\theta )$ in Example \ref{eg:so-gl}. 
Namely, $\mathfrak{g}_0=\mathfrak{so}(2p,2p)$, $\theta (X)=I_{2p,2p}XI_{2p,2p}$ 
and $\sigma _1(X)=J_{p,p}XJ_{p,p}^{-1}$, $\sigma _2(X)=J_{2p}XJ_{2p}^{-1}$ ($X\in \mathfrak{g}_0$). 
We will focus on two fixed point sets $\mathfrak{g}_0^{\sigma _1}$ and $\mathfrak{g}_0^{\sigma _2}$. 
Then, $\mathfrak{g}_0^{\sigma _1}$ is isomorphic to $\mathfrak{u}(p,p)$ 
and $\mathfrak{g}_0^{\sigma _2}$ to $\mathfrak{gl}(2p,\mathbb{R})$. 
If $p\geq 2$, 
then $\mathfrak{u}(p,p)$ is not isomorphic to $\mathfrak{gl}(2p,\mathbb{R})$, 
from which $(\mathfrak{g}_0,\sigma _1;\theta )\not \equiv (\mathfrak{g}_0,\sigma _2;\theta )$. 
In the special case $p=1$, $\mathfrak{u}(1,1)$ is isomorphic to $\mathfrak{gl}(2,\mathbb{R})$, 
nevertheless, 
$(\mathfrak{so}(2,2),\sigma _1;\theta )$ is not equivalent to $(\mathfrak{so}(2,2),\sigma _2;\theta )$. 
Indeed, 
the center of $\mathfrak{u}(1,1)$ is contained in $\mathfrak{g}_0^{\theta }$, 
whereas, the center of $\mathfrak{gl}(2,\mathbb{R})$ is contained in $\mathfrak{g}_0^{-\theta }$. 

\subsection{Compact semisimple symmetric triad}
\label{subsec:triad}

In this subsection, we give a brief summary on compact semisimple symmetric triads 
whose definition is given in the following. 

\begin{define}
\label{def:triad}
A triplet $(\mathfrak{g},\theta _1,\theta _2)$ of a compact semisimple Lie algebra $\mathfrak{g}$ 
and two involutions $\theta _1,\theta _2$ on $\mathfrak{g}$ 
is called a {\it compact semisimple symmetric triad}. 
Further, $(\mathfrak{g},\theta _1,\theta _2)$ is {\it commutative} 
if $\theta _1\theta _2=\theta _2\theta _1$. 
\end{define}

In a special case where $\theta _1$ coincides with $\theta _2$, 
a compact semisimple symmetric triad $(\mathfrak{g},\theta _1,\theta _1)$ means 
a compact semisimple symmetric pair $(\mathfrak{g},\theta _1)$. 
We shall sometimes identify $(\mathfrak{g},\theta _1,\theta _1)$ with $(\mathfrak{g},\theta _1)$ 
which will be treated again in Section \ref{subsec:riemannian}. 

Two compact semisimple symmetric triads $(\mathfrak{g},\theta _1,\theta _2)$, 
$(\mathfrak{g}',\theta _1',\theta _2')$ are equivalent, denoted by 
$(\mathfrak{g},\theta _1,\theta _2)\equiv (\mathfrak{g}',\theta '_1,\theta '_2)$, 
if there exists a Lie algebra isomorphism $\psi :\mathfrak{g}\to \mathfrak{g}'$ such that 
two equalities $\psi \theta _1=\theta _1'\psi $, $\psi \theta _2=\theta _2'\psi $ hold. 
This notion defines an equivalence relation on the set of compact semisimple symmetric triads. 
Further, this is compatible with the commutativity, 
namely, 
if $(\mathfrak{g},\theta _1,\theta _2)\equiv (\mathfrak{g},\theta _1',\theta _2')$ and 
$\theta _1\theta _2=\theta _2\theta _1$, 
then $\theta _1'\theta _2'=\theta _2'\theta _1'$ holds. 

We notice that 
compact semisimple symmetric triads may not be always commutative. 
In fact, 
Matsuki \cite[Remark 1.2]{matsuki} has pointed out that 
no pair $(\theta _1,\theta _2)$ of involutions on $\mathfrak{g}=\mathfrak{so}(2n)$ satisfying 
$\mathfrak{g}^{\theta _1}\simeq \mathfrak{so}(2p+1)+\mathfrak{so}(2n-2p-1)$ and 
$\mathfrak{g}^{\theta _2}\simeq \mathfrak{u}(n)$ with $2p+1<n$ satisfies $\theta _1\theta _2
=\theta _2\theta _1$. 
In the argument of this paper, 
we will always consider compact semisimple symmetric triads $(\mathfrak{g},\theta _1,\theta _2)$ 
with $\theta _1\theta _2=\theta _2\theta _1$. 

Let $(\mathfrak{g},\theta _1,\theta _2)$ be a commutative compact semisimple symmetric triad. 
Then, $\theta _1\theta _2$ is another involution on $\mathfrak{g}$. 
Obviously, $\theta _1\theta _2$ commutes with $\theta _1$ and $\theta _2$. 
Thus, we obtain a new commutative compact semisimple symmetric triad 
$(\mathfrak{g},\theta _1,\theta _1\theta _2)$. 

\begin{define}
\label{def:associated-triad}
We say that $(\mathfrak{g},\theta _1,\theta _1\theta _2)$ 
is the {\it associated} compact symmetric triad of $(\mathfrak{g},\theta _1,\theta _2)$. 
Further, $(\mathfrak{g},\theta _1,\theta _2)$ is {\it self-associated} 
if $(\mathfrak{g},\theta _1,\theta _1\theta _2) \equiv (\mathfrak{g},\theta _1,\theta _2)$. 
\end{define}

We write 
\begin{align*}
(\mathfrak{g},\theta _1,\theta _2)^a:=(\mathfrak{g},\theta _1,\theta _1\theta _2). 
\end{align*}

The fixed point set $\mathfrak{g}^{\theta _1\theta _2}$ is expressed as follows. 
We set $\mathfrak{k}_i:=\mathfrak{g}^{\theta _i}$ and $\mathfrak{p}_i:=\mathfrak{g}^{-\theta _i}$ 
($i=1,2$). 
Since $\theta _1\theta _2=\theta _2\theta _1$, we decompose $\mathfrak{g}$ into as follows 
\begin{align}
\label{eq:decomp}
\mathfrak{g}
&=\mathfrak{k}_1\cap \mathfrak{k}_2+\mathfrak{k}_1\cap \mathfrak{p}_2
+\mathfrak{p}_1\cap \mathfrak{k}_2+\mathfrak{p}_1\cap \mathfrak{p}_2. 
\end{align}
Hence, we obtain 
\begin{align}
\mathfrak{g}^{\theta _1\theta _2}=\mathfrak{k}_1\cap \mathfrak{k}_2+\mathfrak{p}_1\cap \mathfrak{p}_2. 
\label{eq:associated}
\end{align}

It is natural to regard that two compact semisimple symmetric triads 
$(\mathfrak{g},\theta _1,\theta _2)$ and $(\mathfrak{g},\theta _2,\theta _1)$ are the same. 
However, we shall distinguish them in the sense as follows: 

\begin{define}
\label{def:dual-triad}
We say that $(\mathfrak{g},\theta _2,\theta _1)$ is the {\it dual} compact symmetric triad 
of $(\mathfrak{g},\theta _1,\theta _2)$. 
Further, 
$(\mathfrak{g},\theta _1,\theta _2)$ is {\it self-dual} 
if $(\mathfrak{g},\theta _2,\theta _1)\equiv (\mathfrak{g},\theta _1,\theta _2)$. 
\end{define}

We write 
\begin{align}
(\mathfrak{g},\theta _1,\theta _2)^d:=(\mathfrak{g},\theta _2,\theta _1). 
\label{eq:dual compact}
\end{align}
A direct computation shows that 
$(\mathfrak{g},\theta _1,\theta _2)^{ada}=(\mathfrak{g},\theta _1,\theta _2)^{dad}$. 

\begin{eg}
\label{eg:so-so-u}
Let $p,q$ be positive integers. 
We take a compact semisimple Lie algebra $\mathfrak{g}$ as 
$\mathfrak{so}(2p+2q)=\{ X\in M(2p+2q,\mathbb{R}):{}^tX=-X\} $ 
and two involutions $\theta _1,\theta _2$ on $\mathfrak{g}$ as 
$\theta _1(X)=I_{2p,2q}XI_{2p,2q}$ and $\theta _2(X)=J_{p,q}XJ_{p,q}^{-1}$ ($X\in \mathfrak{g}$). 
Here, two matrices $I_{2p,2q}$ and $J_{p,q}$ are given by (\ref{eq:I_{m,n}}) and (\ref{eq:J_{m,n}}), 
respectively. 
Then, $(\mathfrak{g},\theta _1,\theta _2)$ is a commutative compact semisimple symmetric triad. 

The fixed point set $\mathfrak{g}^{\theta _1}$ is of the form 
\begin{align}
\label{eq:k_1}
\mathfrak{g}^{\theta _1}
=\left\{ 
	\left( 
		\begin{array}{cc}
		A & O \\ 
		O & D
		\end{array}
	\right) :A,D\in \mathfrak{so}(2p)
\right\} =\mathfrak{so}(2p)+\mathfrak{so}(2p). 
\end{align}
and $\mathfrak{g}^{\theta _2}$ is described as follows 
\begin{align}
\label{eq:k_2}
\mathfrak{g}^{\theta _2}
&=\left\{ 
	\left(
		\begin{array}{cccc}
		A_{1} & -A_{2} & -B_1 & B_2 \\
		A_2 & A_1 & -B_2 & -B_1 \\
		{}^tB_1 & {}^t B_2 & D_1 & -D_2 \\
		-{}^tB_2 & {}^tB_1 & D_2 & D_1
		\end{array}
	\right) :
	\begin{array}{c}
	A_1\in \mathfrak{so}(p),A_2\in \operatorname{Sym}(p,\mathbb{R}),\\
	D_1\in \mathfrak{so}(q),D_2\in \operatorname{Sym}(q,\mathbb{R}),\\
	B_1,B_2\in M(p,q\,;\mathbb{R}). 
	\end{array}
\right\} . 
\end{align}
Then, $\mathfrak{g}^{\theta _2}$ is isomorphic to the compact reductive Lie algebra 
$\mathfrak{u}(p+q)=\{ X\in M(p+q,\mathbb{C}):{}^t\overline{X}=-X\} $ 
by the Lie algebra isomorphism $\mathfrak{g}^{\theta _2}\stackrel{\sim }{\to }\mathfrak{u}(p+q)$ given by 
\begin{align*}
\left(
	\begin{array}{cccc}
	A_{1} & -A_{2} & -B_1 & B_2 \\
	A_2 & A_1 & -B_2 & -B_1 \\
	{}^tB_1 & {}^t B_2 & D_1 & -D_2 \\
	-{}^tB_2 & {}^tB_1 & D_2 & D_1
	\end{array}
\right) \mapsto \left( 
	\begin{array}{cc}
	A_1+\sqrt{-1}A_2 & -B_1-\sqrt{-1}B_2 \\
	{}^tB_1-\sqrt{-1}\,{}^tB_2 & D_1+\sqrt{-1}D_2 
	\end{array}
\right) . 
\end{align*}
\end{eg}

\begin{eg}
\label{eg:so-so-u'}
For a positive integer $p$, 
let $\mathfrak{g}$ be $\mathfrak{so}(4p)$ and $\theta _1,\theta _2$ be two involutions on $\mathfrak{g}$ 
defined by $\theta _1(X)=I_{2p,2p}XI_{2p,2p}$ and $\theta _2(X)=J_{2p}XJ_{2p}^{-1}$ ($X\in \mathfrak{g}$). 
Here, $J_{2p}$ is given by (\ref{eq:J_m}). 
Then, $(\mathfrak{g},\theta _1,\theta _2)$ is a commutative compact semisimple symmetric triad 
and $\mathfrak{g}^{\theta _1}=\mathfrak{so}(2p)+\mathfrak{so}(2p)$. 
The fixed point set $\mathfrak{g}^{\theta _2}$ is 
\begin{align}
\label{eq:k_2'}
\mathfrak{g}^{\theta _2}
=\left\{ 
	\left( 
		\begin{array}{cc}
		A & -B \\
		B & A
		\end{array}
	\right) :A\in \mathfrak{so}(2p),B\in \operatorname{Sym}(2p,\mathbb{R})
\right\} . 
\end{align}
Thus, $\mathfrak{g}^{\theta _2}\simeq \mathfrak{u}(2p)$ 
because of the following Lie algebra isomorphism 
\begin{align*}
\mathfrak{g}^{\theta _2}\stackrel{\sim}{\to } \mathfrak{u}(2p),\quad 
\left( 
	\begin{array}{cc}
	A & -B \\
	B & A
	\end{array}
\right) \mapsto A+\sqrt{-1}B. 
\end{align*}
\end{eg}

\begin{eg}
\label{eg:equivalent}
It is noteworthy to mention here the relations between two commutative compact semisimple symmetric triads 
treated in Examples \ref{eg:so-so-u} and \ref{eg:so-so-u'}. 
Let us consider three involutions $\theta _1,\theta _2,\theta _2'$ 
on the compact simple Lie algebra $\mathfrak{g}=\mathfrak{so}(4p)$ 
defined by $\theta _1(X)=I_{2p,2p}XI_{2p,2p}$, $\theta _2(X)=J_{p,p}XJ_{p,p}^{-1}$, 
$\theta _2'(X)=J_{2p}XJ_{2p}^{-1}$ ($X\in \mathfrak{g}$). 
As we have seen in Examples \ref{eg:so-so-u} and \ref{eg:so-so-u'}, 
$\mathfrak{g}^{\theta _1}=\mathfrak{so}(2p)+\mathfrak{so}(2p)$ 
and $\mathfrak{g}^{\theta _2}\simeq \mathfrak{g}^{\theta _2'}\simeq \mathfrak{u}(2p)$. 
Nevertheless, $(\mathfrak{g},\theta _1,\theta _2)$ is not equivalent to $(\mathfrak{g},\theta _1,\theta _2')$. 
To see it, we focus on $\mathfrak{g}^{\theta _1}\cap \mathfrak{g}^{\theta _2}$ 
and $\mathfrak{g}^{\theta _1}\cap \mathfrak{g}^{\theta _2'}$. 
In view of (\ref{eq:k_1}) and (\ref{eq:k_2}), 
it turns out that 
$\mathfrak{g}^{\theta _1}\cap \mathfrak{g}^{\theta _2}\simeq \mathfrak{u}(p)+\mathfrak{u}(p)$, 
and it follows from (\ref{eq:k_1}) and (\ref{eq:k_2'}) that 
$\mathfrak{g}^{\theta _1}\cap \mathfrak{g}^{\theta _2'}\simeq \mathfrak{so}(2p)$. 
This means that $\mathfrak{g}^{\theta _1}\cap \mathfrak{g}^{\theta _2}$ is not isomorphic to 
$\mathfrak{g}^{\theta _1}\cap \mathfrak{g}^{\theta _2'}$. 
Due to Lemma \ref{lem:k_1k_2} below, 
we conclude $(\mathfrak{g},\theta _1,\theta _2)\not \equiv (\mathfrak{g},\theta _1,\theta _2')$. 
\end{eg}

%More precisely, 
%we take three involutions $\theta _1,\theta _2,\theta _2'$ on $\mathfrak{g}=\mathfrak{so}(4p)$ 
%as $\theta _1(X)=I_{2p,2p}XI_{2p,2p}$, $\theta _2(X)=J_{p,p}XJ_{p,p}^{-1}$, 
%$\theta _2'(X)=J_{2p}XJ_{2p}^{-1}$ ($X\in \mathfrak{g}$). 
%Then, we have $\mathfrak{g}^{\theta _1}=\mathfrak{so}(2p)+\mathfrak{so}(2p)$ 
%and $\mathfrak{g}^{\theta _2}\simeq \mathfrak{g}^{\theta _2'}\simeq \mathfrak{u}(2p)$. 
%Nevertheless, it is known that $(\mathfrak{g},\theta _1,\theta _2)
%\not \equiv (\mathfrak{g},\theta _1,\theta _2')$. 
%This fact might be understood by comparing 
%the duality of $(\mathfrak{g},\theta _1,\theta _2)$ with 
%the duality of $(\mathfrak{g},\theta _1,\theta _2')$ according to Theorem \ref{thm:duality-thm} below, 
%which we will explain in Examples \ref{eg:duality-so-u} and \ref{eg:duality-so-gl}. 

We will here explain Lemma \ref{lem:k_1k_2} used in Example \ref{eg:equivalent}. 
This provides a necessary condition for two commutative compact semisimple symmetric triads to be equivalent. 

\begin{lemma}
\label{lem:k_1k_2}
If two compact semisimple symmetric triads $(\mathfrak{g},\theta _1,\theta _2)$, 
$(\mathfrak{g}',\theta _1',\theta _2')$ are equivalent, 
then $\mathfrak{g}^{\theta _1}\cap \mathfrak{g}^{\theta _2}$ is isomorphic to 
$(\mathfrak{g}')^{\theta _1'}\cap (\mathfrak{g}')^{\theta _2'}$. 
\end{lemma}

\begin{proof}
We take a Lie algebra isomorphism $\varphi :\mathfrak{g}\to \mathfrak{g}'$ such that 
$\theta _1'=\varphi \theta _1\varphi ^{-1}$ and $\theta _2'=\varphi \theta _2\varphi ^{-1}$. 
Then, $\varphi $ induces Lie algebra isomorphisms 
from $\mathfrak{g}^{\theta _1}$ to $(\mathfrak{g}')^{\theta _1'}$ 
and from $\mathfrak{g}^{\theta _2}$ to $(\mathfrak{g}')^{\theta _2'}$. 
Since $(\mathfrak{g}')^{\theta _1'}\cap (\mathfrak{g}')^{\theta _2'}
=\varphi (\mathfrak{g}^{\theta _1})\cap \varphi (\mathfrak{g}^{\theta _2})
=\varphi (\mathfrak{g}^{\theta _1}\cap \mathfrak{g}^{\theta _2})$, 
the Lie algebra $\mathfrak{g}^{\theta _1}\cap \mathfrak{g}^{\theta _2}$ must be isomorphic to 
$(\mathfrak{g}')^{\theta _1'}\cap (\mathfrak{g}')^{\theta _2'}$. 
\end{proof}

%%%%%%%%%%%%%%%%%%%%%%%%%%%%%%%%%%%%%%%%%%%%%%%%%%%%%%%%%%%%%%%%%%%%%%%%%%%%%%%%%%%%%%%%%%

\section{Duality between non-compact semisimple symmetric pairs and commutative compact semisimple symmetric triads}
\label{sec:duality}

This section is one of the main parts of this paper. 
Namely, we construct a one-to-one correspondence 
between non-compact semisimple symmetric pairs and commutative compact semisimple symmetric triads. 

From now on, we denote 
by $\mathfrak{P}$ the set of non-compact semisimple symmetric pairs, 
by $\mathfrak{P}_c$ the set of non-compact semisimple symmetric pairs equipped with Cartan involutions 
and by $\mathfrak{T}$ the set of commutative compact semisimple symmetric triads. 
The aim of this section is to give a bijection 
between the set $\mathfrak{P}/{\equiv }$ of equivalence classes in $\mathfrak{P}$ 
and the set $\mathfrak{T}/{\equiv }$ of equivalence classes in $\mathfrak{T}$. 

\subsection{One-to-one correspondence between $\mathfrak{P}_c/{\equiv }$ and $\mathfrak{P}/{\equiv }$}
\label{subsec:cartan}

First of all, 
we will consider the relation between $\mathfrak{P}_c/{\equiv }$ and $\mathfrak{P}/{\equiv }$. 

Let $p$ be the projection from $\mathfrak{P}_c$ to $\mathfrak{P}$ defined by 
\begin{align*}
p:\mathfrak{P}_c\to \mathfrak{P},\quad (\mathfrak{g}_0,\sigma ;\theta )\mapsto (\mathfrak{g}_0,\sigma ). 
\end{align*}
Clearly, 
$p(\mathfrak{g}_0,\sigma ;\theta )\equiv p(\mathfrak{g}'_0,\sigma ';\theta ')$ 
if $(\mathfrak{g}_0,\sigma ;\theta )\equiv (\mathfrak{g}'_0,\sigma ';\theta ')$, 
which gives rise to the map $\widetilde{p}$ from $\mathfrak{P}_c/{\equiv }$ to $\mathfrak{P}/{\equiv }$. 

\begin{lemma}
\label{lem:cartan-involution}
The map $\widetilde{p}:\mathfrak{P}_c/{\equiv }\to \mathfrak{P}/{\equiv }$ is bijective. 
\end{lemma}

\begin{proof}
The surjectivity of $\widetilde{p}$ follows from that of $p$. 
Then, let us show $\widetilde{p}$ is injective. 

Assume that 
$\widetilde{p}(\mathfrak{g}_0,\sigma ;\theta )\equiv \widetilde{p}(\mathfrak{g}_0',\sigma ';\theta ')$ 
for $(\mathfrak{g}_0,\sigma ;\theta ),(\mathfrak{g}_0',\sigma ';\theta ')\in \mathfrak{P}_c$, 
namely, $(\mathfrak{g}_0,\sigma )\equiv (\mathfrak{g}_0',\sigma ')$. 
We take a Lie algebra isomorphism $\varphi :\mathfrak{g}_0\to \mathfrak{g}_0'$ 
satisfying $\varphi \sigma =\sigma '\varphi $. 
We define an involution $\widetilde{\theta }$ on $\mathfrak{g}_0$ 
by $\widetilde{\theta }=\varphi ^{-1}\theta '\varphi $. 
Then, $\widetilde{\theta }$ commutes with $\sigma$ because 
\begin{align*}
\widetilde{\theta }\sigma 
=(\varphi ^{-1}\theta '\varphi )\sigma 
=\varphi ^{-1}\theta '(\sigma '\varphi )
=\varphi ^{-1}(\sigma '\theta ')\varphi 
=\sigma (\varphi ^{-1}\theta '\varphi )
=\sigma \widetilde{\theta }. 
\end{align*}
Further, the fixed point set $\mathfrak{g}_0^{\widetilde{\theta }}$ 
is equal to $\varphi ^{-1}((\mathfrak{g}_0')^{\theta '})$, 
equivalently, 
$\varphi (\mathfrak{g}_0^{\widetilde{\theta }})=(\mathfrak{g}_0')^{\theta '}$. 
Since $\varphi $ is a Lie algebra isomorphism 
and $(\mathfrak{g}_0')^{\theta '}$ is a maximal compact subalgebra of $\mathfrak{g}_0'$, 
the set $\mathfrak{g}_0^{\widetilde{\theta }}$ is a maximal compact subalgebra of $\mathfrak{g}_0$. 
Hence, $\widetilde{\theta }$ is also a Cartan involution of $\mathfrak{g}_0$ 
with $\widetilde{\theta }\sigma =\sigma \widetilde{\theta }$. 
It follows from Fact \ref{fact:loos} that 
there exists $X\in \mathfrak{g}_0^{\sigma }$ such that 
%$\widetilde{\theta }$ is written by 
%$\widetilde{\theta }=e^{\operatorname{ad}X}\theta e^{-\operatorname{ad}X}$ 
%for some $X\in \mathfrak{g}_0^{\sigma }$. 
%Therefore, we obtain the equality via $\widetilde{\theta }$ as follows: 
\begin{align}
\varphi ^{-1}\theta '\varphi 
=e^{\operatorname{ad}X}\theta e^{-\operatorname{ad}X}. 
\label{eq:psi}
\end{align}

Here, we set $\psi :=\varphi e^{\operatorname{ad}X}$. 
Then, $\psi $ is a Lie algebra isomorphism 
from $\mathfrak{g}_0$ to $\mathfrak{g}_0'$. 
By the condition (\ref{eq:commute sigma}) and the equality $\varphi \sigma =\sigma '\varphi $, 
we have 
\begin{align*}
\psi \sigma 
=(\varphi e^{\operatorname{ad}X})\sigma 
=\varphi (\sigma e^{\operatorname{ad}X})
=(\sigma '\varphi )e^{\operatorname{ad}X}
=\sigma '\psi . 
\end{align*}
On the other hand, 
the equality (\ref{eq:psi}) implies $\psi \theta =\theta '\psi $. 
Hence, we have shown $(\mathfrak{g}_0,\sigma ;\theta )\equiv (\mathfrak{g}'_0,\sigma ';\theta ')$, 
from which $\widetilde{p}$ is injective. 

Therefore, we have proved Lemma \ref{lem:cartan-involution}. 
\end{proof}

Let us assume that $(\mathfrak{g}_0,\sigma )\equiv (\mathfrak{g}_0',\sigma ')$. 
Combining Lemma \ref{lem:cartan-involution} with Lemma \ref{lem:conjugate-associated}, 
we have $(\mathfrak{g}_0,\sigma ;\theta )\equiv (\mathfrak{g}_0',\sigma ';\theta ')$ 
for any Cartan involution $\theta $ of $\mathfrak{g}_0$ commuting with $\sigma$ 
and any Cartan involution $\theta '$ of $\mathfrak{g}_0'$ commuting with $\sigma '$. 
Hence, it follows from Lemma \ref{lem:associated} that 
$(\mathfrak{g}_0,\sigma ;\theta )^a\equiv (\mathfrak{g}_0',\sigma ';\theta ')^a$ 
and $(\mathfrak{g}_0,\sigma ;\theta )^d\equiv (\mathfrak{g}_0',\sigma ';\theta ')^d$. 

\begin{define}
\label{def:a-d}
We write $[(\mathfrak{g}_0,\sigma )]\in \mathfrak{P}/{\equiv }$ 
for the equivalence class containing $(\mathfrak{g}_0,\sigma )\in \mathfrak{P}$. 
We define $[(\mathfrak{g}_0,\sigma )]^a$ and $[(\mathfrak{g}_0,\sigma )]^d$ by 
\begin{align}
[(\mathfrak{g}_0,\sigma )]^a:=[p((\mathfrak{g}_0,\sigma ;\theta )^a)],\quad 
[(\mathfrak{g}_0,\sigma )]^d:=[p((\mathfrak{g}_0,\sigma ;\theta )^d)]. 
\end{align}
Moreover, 
we say that $[(\mathfrak{g}_0,\sigma )]$ is {\it self-associated} 
if $[(\mathfrak{g}_0,\sigma )]^a=[(\mathfrak{g}_0,\sigma )]$, 
and {\it self-dual} if $[(\mathfrak{g}_0,\sigma )]^d=[(\mathfrak{g}_0,\sigma )]$. 
\end{define}

\subsection{Correspondence between $\mathfrak{P}_c$ and $\mathfrak{T}$}
\label{subsec:correspondence}

In this subsection, 
we construct a map 
$\Phi :\mathfrak{T}\to \mathfrak{P}_c$ and $\Psi :\mathfrak{P}_c\to \mathfrak{T}$ explicitly. 

First, we give a map $\Phi :\mathfrak{T}\to \mathfrak{P}_c$ as follows. 
Let $(\mathfrak{g},\theta _1,\theta _2)$ be a commutative compact semisimple symmetric triad. 
We will use the notation as in Section \ref{subsec:triad}. 
Let $\mu$ be the complex conjugation of $\mathfrak{g}_{\mathbb{C}}=\mathfrak{g}+\sqrt{-1}\mathfrak{g}$ 
with respect to $\mathfrak{g}$. 
We extend $\theta _1$ and $\theta _2$ to $\mathbb{C}$-linear automorphisms 
on $\mathfrak{g}_{\mathbb{C}}$, respectively. 
Then, $\theta _1$ commutes with $\mu$, 
from which $\tau :=\mu \theta _1$ defines another anti-linear involution 
on $\mathfrak{g}_{\mathbb{C}}$. 

The fixed point set $\mathfrak{g}_0:=\mathfrak{g}_{\mathbb{C}}^{\tau }$ 
is a non-compact real form of $\mathfrak{g}_{\mathbb{C}}$ 
and is expressed as follows. 

\begin{lemma}
\label{lem:g0}
$\mathfrak{g}_0
=\mathfrak{g}^{\theta _1}+\sqrt{-1}\mathfrak{g}^{-\theta _1}
=\mathfrak{k}_1+\sqrt{-1}\mathfrak{p}_1$. 
\end{lemma}

\begin{proof}
$\mathfrak{g}_0
=\mathfrak{g}_{\mathbb{C}}^{\tau }
=(\mathfrak{g}_{\mathbb{C}}^{\mu })^{\theta _1}+(\mathfrak{g}_{\mathbb{C}}^{-\mu })^{-\theta _1}
=\mathfrak{g}^{\theta _1}+\sqrt{-1}\mathfrak{g}^{-\theta _1}$. 
\end{proof}

The restriction of the $\mathbb{C}$-linear map $\theta _1$ 
on $\mathfrak{g}_{\mathbb{C}}$ to $\mathfrak{g}_0$ 
becomes an involution on $\mathfrak{g}_0$. 
We put $\theta :=\theta _1|_{\mathfrak{g}_0}$. 
Then, $\mathfrak{k}_0:=\mathfrak{g}_0^{\theta }$ is 
a maximal compact subalgebra of $\mathfrak{g}_0$. 
This means that $\theta $ is a Cartan involution of $\mathfrak{g}_0$ 
and $\mathfrak{k}_0$ coincides with $\mathfrak{k}_1=\mathfrak{g}^{\theta _1}$. 

Similarly, $\sigma :=\theta _2|_{\mathfrak{g}_0}$ is also an involution on $\mathfrak{g}_0$. 
Clearly, $\sigma $ commutes with $\theta $. 
Therefore, 
$(\mathfrak{g}_0,\sigma ;\theta )$ 
is a non-compact semisimple symmetric pair equipped with a Cartan involution. 
In this sense, we define a map $\Phi :\mathfrak{T}\to \mathfrak{P}_c $ by 
\begin{align}
\label{eq:pair}
\Phi (\mathfrak{g},\theta _1,\theta _2)
&=(\mathfrak{g}_0,\sigma ;\theta )
=(\mathfrak{g}^{\theta _1}+\sqrt{-1}\mathfrak{g}^{-\theta _1},
	\theta _2|_{\mathfrak{g}_0};\theta _1|_{\mathfrak{g}_0}). 
\end{align}

Hereafter, 
we use the same letters $\theta _1,\theta _2$ to denote the restrictions 
$\theta _1|_{\mathfrak{g}_0},\theta _2|_{\mathfrak{g}_0}$, respectively. 
Note that the fixed point set $\mathfrak{h}_0:=\mathfrak{g}_0^{\sigma }$ is written as follows. 

\begin{lemma}
\label{lem:h0}
$\mathfrak{h}_0=\mathfrak{k}_1\cap \mathfrak{k}_2+\sqrt{-1}(\mathfrak{p}_1\cap \mathfrak{k}_2)$. 
\end{lemma}

\begin{proof}
$\mathfrak{g}_0^{\sigma }
=(\mathfrak{g}^{\theta _1}+\sqrt{-1}\mathfrak{g}^{-\theta _1})^{\sigma }
=\mathfrak{g}^{\theta _1}\cap \mathfrak{g}^{\theta _2}
	+\sqrt{-1}(\mathfrak{g}^{-\theta _1}\cap \mathfrak{g}^{\theta _2})
=\mathfrak{k}_1\cap \mathfrak{k}_2+\sqrt{-1}(\mathfrak{p}_1\cap \mathfrak{k}_2)$. 
\end{proof}

By the observation of Lemma \ref{lem:h0}, 
$\mathfrak{k}_1\cap \mathfrak{k}_2$ is a maximal compact subalgebra of $\mathfrak{h}_0$. 

Next, we give a map $\Psi :\mathfrak{P}_c\to \mathfrak{T}$. 
Let $(\mathfrak{g}_0,\sigma ;\theta )$ be a non-compact semisimple symmetric pair 
equipped with a Cartan involution. 
Retain the notation as in Section \ref{subsec:pair}. 
%, namely, 
%we write $\mathfrak{g}_0=\mathfrak{k}_0+\mathfrak{p}_0
%=\mathfrak{g}_0^{\theta }+\mathfrak{g}_0^{-\theta }$ for the corresponding Cartan decomposition 
%and $\mathfrak{g}_0=\mathfrak{h}_0+\mathfrak{q}_0$ for the eigenspace decomposition of $\sigma$. 
We extend $\theta$ and $\sigma$ to $\mathbb{C}$-linear automorphisms on 
the complexification $\mathfrak{g}_{\mathbb{C}}=\mathfrak{g}_0+\sqrt{-1}\mathfrak{g}_0$. 
As we have seen in (\ref{eq:compact real form}), 
$\mathfrak{g}:=\mathfrak{g}_0^{\theta }+\sqrt{-1}\mathfrak{g}_0^{-\theta }
=\mathfrak{k}_0+\sqrt{-1}\mathfrak{p}_0$ is a compact semisimple Lie algebra. 
Then, the restrictions of $\theta $ and $\sigma$ to $\mathfrak{g}$ become involutions on $\mathfrak{g}$. 

We set $\theta _1:=\theta |_{\mathfrak{g}}$ and $\theta _2:=\sigma |_{\mathfrak{g}}$. 
Clearly, $\theta _1$ commutes with $\theta _2$. 
Thus, we get a commutative compact semisimple symmetric triad $(\mathfrak{g},\theta _1,\theta _2)$. 
This yields the map $\Psi :\mathfrak{P}_c\to \mathfrak{T}$ defined by 
\begin{align}
\Psi (\mathfrak{g}_0,\sigma ;\theta )
&=(\mathfrak{g},\theta _1,\theta _2)
=(\mathfrak{g}_0^{\theta }+\sqrt{-1}\mathfrak{g}_0^{-\theta },
	\theta |_{\mathfrak{g}},\sigma |_{\mathfrak{g}}). 
\label{eq:triad}
\end{align}

From now, 
we use the same letters $\theta,\sigma $ to denote the restrictions 
$\theta |_{\mathfrak{g}},\sigma |_{\mathfrak{g}}$ to $\mathfrak{g}$, respectively. 

The fixed point set $\mathfrak{k}_1:=\mathfrak{g}^{\theta _1}$ equals to $\mathfrak{k}_0$, 
and $\mathfrak{k}_2:=\mathfrak{g}^{\theta _2}$ is given by 
\begin{align}
\label{eq:g2}
\mathfrak{g}^{\theta _2}=\mathfrak{k}_0^{\sigma }+\sqrt{-1}\mathfrak{p}_0^{\sigma }
=\mathfrak{k}_0\cap \mathfrak{h}_0+\sqrt{-1}(\mathfrak{p}_0\cap \mathfrak{h}_0). 
\end{align}

\subsection{One-to-one correspondence between $\mathfrak{P}/{\equiv }$ and $\mathfrak{T}/{\equiv }$}
\label{subsec:duality}

Two maps $\Phi ,\Psi $ given in Section \ref{subsec:correspondence} 
are inversed correspondence to each other 
(see the proof of Theorem \ref{thm:duality-thm} below). 
Using them, we will construct a one-to-one correspondence between 
$\mathfrak{P}_c/{\equiv }$ and $\mathfrak{T}/{\equiv }$. 

For this, we prepare: 

\begin{lemma}
\label{lem:well-defined}
The map $\Phi $ defined by (\ref{eq:pair}) induces the map 
from $\mathfrak{T}/{\equiv }$ to $\mathfrak{P}_c/{\equiv }$. 
Similarly, the map $\Psi $ defined by (\ref{eq:triad}) induces the map 
from $\mathfrak{P}_c/{\equiv }$ to $\mathfrak{T}/{\equiv }$. 
\end{lemma}

In the following, 
$[(\mathfrak{g},\theta _1,\theta _2)]\in \mathfrak{T}/{\equiv }$ denotes 
the equivalence class containing $(\mathfrak{g},\theta _1,\theta _2)\in \mathfrak{T}$. 

\begin{proof}
Let us assume $[(\mathfrak{g},\theta _1,\theta _2)]=[(\mathfrak{g}',\theta _1',\theta _2')]\in 
\mathfrak{T}/{\equiv }$. 
By definition, 
we obtain $(\mathfrak{g},\theta _1,\theta _2)\equiv (\mathfrak{g}',\theta _1',\theta _2')$, 
from which we take a Lie algebra isomorphism $\varphi :\mathfrak{g}\to \mathfrak{g}'$ 
satisfying $\varphi \theta _i=\theta _i'\varphi $ for $i=1,2$. 
We extend $\varphi$ to a complex Lie algebra isomorphism 
from $\mathfrak{g}_{\mathbb{C}}$ to $\mathfrak{g}_{\mathbb{C}}'$. 
Since 
$(\mathfrak{g}')^{\theta _1'}$ coincides with $\varphi (\mathfrak{g}^{\theta _1})$ 
and $(\mathfrak{g}')^{-\theta _1'}$ with $\varphi (\mathfrak{g}^{-\theta _1})$, 
we have 
\begin{align*}
\mathfrak{g}_0'
&=(\mathfrak{g}')^{\theta _1'}+\sqrt{-1}(\mathfrak{g}')^{-\theta _1'}
%&=\varphi (\mathfrak{g}^{\theta _1})+\sqrt{-1}\varphi (\mathfrak{g}^{-\theta _1})\\
=\varphi (\mathfrak{g}^{\theta _1}+\sqrt{-1}\mathfrak{g}^{-\theta _1})
=\varphi (\mathfrak{g}_0). 
\end{align*}
Hence, we conclude 
\begin{align*}
%\Phi (\mathfrak{g}',\theta _1',\theta _2')
%=
(\mathfrak{g}_0',\theta _2';\theta _1')
=(\varphi (\mathfrak{g}_0),\varphi \theta _2\varphi ^{-1};\varphi \theta _1\varphi ^{-1})
\equiv (\mathfrak{g}_0,\theta _2;\theta _1). 
%=\Phi (\mathfrak{g},\theta _1,\theta _2). 
\end{align*}
This means $\Phi (\mathfrak{g}',\theta _1',\theta _2')\equiv \Phi (\mathfrak{g},\theta _1,\theta _2)$, 
from which this gives rise to the map $\Phi :\mathfrak{T}/{\equiv }\to \mathfrak{P}_c/{\equiv }$. 

Similarly, 
we can induce the map $\Psi :\mathfrak{P}_c/{\equiv }\to \mathfrak{T}/{\equiv }$. 
Then, we omit its detail. 
\end{proof}

\begin{theorem}[Duality theorem]
\label{thm:duality-thm}
$\widetilde{p}\circ \Phi: \mathfrak{T}/{\equiv }\to \mathfrak{P}/{\equiv }$ is a bijective map. 
Moreover, $\Psi \circ (\widetilde{p})^{-1}$ is the inversed map of $\widetilde{p}\circ \Phi$. 
Hence, we get a one-to-one correspondence between $\mathfrak{T}/{\equiv }$ and $\mathfrak{P}/{\equiv }$. 
\end{theorem}

\begin{proof}
We take a non-compact semisimple symmetric pair 
equipped with a Cartan involution $(\mathfrak{g}_0,\sigma ;\theta )$. 
Combining (\ref{eq:pair}) with (\ref{eq:triad}), 
we have 
\begin{align*}
\Phi \circ \Psi (\mathfrak{g}_0,\sigma ;\theta )
&=\Phi (\mathfrak{g}_0^{\theta }+\sqrt{-1}\mathfrak{g}_0^{-\theta },\theta ,\sigma )\\
&=((\mathfrak{g}_0^{\theta }+\sqrt{-1}\mathfrak{g}_0^{-\theta })^{\theta }
	+\sqrt{-1}(\mathfrak{g}_0^{\theta }+\sqrt{-1}\mathfrak{g}_0^{-\theta })^{-\theta },
	\sigma ;\theta )\\
&=(\mathfrak{g}_0^{\theta }+\mathfrak{g}_0^{-\theta },\sigma ;\theta )\\
&=(\mathfrak{g}_0,\sigma ;\theta ). 
\end{align*}
This means that $\Psi$ is the inversed map of $\Phi$. 
Hence, we obtain a bijection from $\mathfrak{P}_c/{\equiv} $ to $\mathfrak{T}/{\equiv }$. 
As Lemma \ref{lem:cartan-involution}, we conclude Theorem \ref{thm:duality-thm}. 
\end{proof}

\begin{notation}
\label{notation:*}
For simplicity, 
we usually use the notation as follows: 
\begin{align*}
(\mathfrak{g},\theta _1,\theta _2)^*
&=\Phi (\mathfrak{g},\theta _1,\theta _2)\quad ((\mathfrak{g},\theta _1,\theta _2)\in \mathfrak{T}),\\
(\mathfrak{g}_0,\sigma ;\theta )^*
&=(\mathfrak{g}_0,\sigma )^*=\Psi (\mathfrak{g}_0,\sigma ;\theta )\quad 
((\mathfrak{g}_0,\sigma ;\theta )\in \mathfrak{P}_c). 
\end{align*}
\end{notation}

\begin{eg}
\label{eg:duality-so-u}
Let $(\mathfrak{g},\theta _1,\theta _2)$ be the commutative compact semisimple symmetric triad 
which has been considered in Example \ref{eg:so-so-u}, 
namely, $\mathfrak{g}=\mathfrak{so}(2p+2q)$, 
$\theta _1(X)=I_{2p,2q}XI_{2p,2q}$ and $\theta _2(X)=J_{p,q}XJ_{p,q}^{-1}$ ($X\in \mathfrak{g}$). 
We recall that $\mathfrak{g}^{\theta _1}=\mathfrak{so}(2p)+\mathfrak{so}(2q)$ 
and $\mathfrak{g}^{\theta _2}\simeq \mathfrak{u}(p+q)$. 
In this setting, we will clarify the dual 
$(\mathfrak{g}_0,\sigma ;\theta )=(\mathfrak{g},\theta _1,\theta _2)^*$ 
of $(\mathfrak{g},\theta _1,\theta _2)$. 

The complexification $\mathfrak{g}$ of $\mathfrak{so}(2p,2q)$ equals $\mathfrak{so}(2p+2q,\mathbb{C})
=\{ X\in M(2p+2q,\mathbb{C}):{}^tX=-X\} $. 
The non-compact real form 
$\mathfrak{g}_0=\mathfrak{g}^{\theta _1}+\sqrt{-1}\mathfrak{g}^{-\theta _1}$ of $\mathfrak{g}$ is given 
as follows. 
The fixed point set $\mathfrak{g}^{\theta _1}=\mathfrak{g}_0^{\theta }=\mathfrak{so}(2p)+\mathfrak{so}(2q)$ 
is given by (\ref{eq:so(2p)+so(2q)}) and $\mathfrak{g}^{-\theta _1}$ by 
\begin{align*}
\mathfrak{g}^{-\theta _1}
=\left\{ 
	\left( 
		\begin{array}{cc}
		O & -B \\
		{}^tB & O
		\end{array}
	\right) :B\in M(2p,2q\,;\mathbb{R})
\right\} . 
\end{align*}
Hence, we have 
\begin{align*}
\mathfrak{g}_0=\left\{ 
	\left( 
		\begin{array}{cc}
		A & -\sqrt{-1}B \\
		\sqrt{-1}\,{}^tB & D
		\end{array}
	\right) :
	\begin{array}{c}
	A\in \mathfrak{so}(2p),~D\in \mathfrak{so}(2q),\\
	B\in M(2p,2q\,;\mathbb{R}). 
	\end{array}
\right\} . 
\end{align*}

Here, we set 
\begin{align}
\label{eq:I_{2p,2q}'}
I_{2p,2q}':=\left( 
	\begin{array}{cc}
	I_{2p} & O \\
	O & -\sqrt{-1}I_{2q}
	\end{array}
\right) 
\end{align}
and define a map $\varphi :\mathfrak{g}_0\to \mathfrak{so}(2p+2q,\mathbb{C})$ by 
\begin{align}
\label{eq:phi}
\varphi (X):=I_{2p,2q}'X(I_{2p,2q}')^{-1}\quad (X\in \mathfrak{g}_0). 
\end{align}
Then, $\varphi$ gives rise to the Lie algebra isomorphism from $\mathfrak{g}_0$ to $\mathfrak{so}(2p,2q)$ 
which is defined by (\ref{eq:so(2p,2q)}). 

Now, we put $\mathfrak{g}_0':=\mathfrak{so}(2p,2q)$ 
and $\theta ':=\varphi \theta \varphi ^{-1}$, $\sigma ':=\varphi \sigma \varphi ^{-1}$. 
Then, $(\mathfrak{g}_0',\sigma ';\theta ')$ is a non-compact semisimple symmetric pair 
equipped with a Cartan involution 
and satisfies $(\mathfrak{g}_0',\sigma ';\theta ')\equiv (\mathfrak{g}_0,\sigma ;\theta )$. 
In particular, two involutions $\theta '$ and $\sigma '$ on $\mathfrak{g}_0'$ are 
expressed as $\theta '(X)=I_{2p,2q}XI_{2p,2q}$ and $\sigma '(X)=J_{p,q}XJ_{p,q}^{-1}$, respectively 
($X\in \mathfrak{g}_0'$). 
Hence, $(\mathfrak{g}_0',\sigma ';\theta ')$ is the same as the non-compact semisimple symmetric pair 
in Example \ref{eg:so-u}, 
namely, $(\mathfrak{g}_0')^{\theta '}=\mathfrak{so}(2p)+\mathfrak{so}(2q)$ 
and $(\mathfrak{g}_0')^{\sigma '}\simeq \mathfrak{u}(p,q)$. 
\end{eg}

\begin{eg}
\label{eg:duality-so-gl}
Let $(\mathfrak{g},\theta _1,\theta _2)$ be a commutative compact semisimple symmetric triad 
with $\mathfrak{g}=\mathfrak{so}(4p)$ and 
$\theta _1(X)=I_{2p,2p}XI_{2p,2p}$, $\theta _2(X)=J_{2p}XJ_{2p}^{-1}$ ($X\in \mathfrak{g}$). 
As mentioned in Example \ref{eg:so-so-u'}, 
we obtain $\mathfrak{g}^{\theta _1}=\mathfrak{so}(2p)+\mathfrak{so}(2p)$ 
and $\mathfrak{g}^{\theta _2}\simeq \mathfrak{u}(2p)$. 

The dual $(\mathfrak{g}_0,\sigma ;\theta ):=(\mathfrak{g},\theta _1,\theta _2)^*$ is given as follows. 
By the same argument as in Example \ref{eg:duality-so-u}, 
a non-compact semisimple real Lie algebra $\mathfrak{g}_0$ is isomorphic to $\mathfrak{so}(2p,2p)$ 
and the fixed point set $\mathfrak{g}_0^{\theta }$ of the Cartan involution 
$\theta :=\theta _1|_{\mathfrak{g}_0}$ equals $\mathfrak{so}(2p)+\mathfrak{so}(2p)$. 

Next, we define a map $\varphi $ by (\ref{eq:phi}). 
We set $\mathfrak{g}_0':=\varphi (\mathfrak{g}_0)=\mathfrak{so}(2p,2q)$ 
and $\sigma '=\varphi \sigma \varphi ^{-1}$. 
Then, an explicit description of $\sigma '$ on $\mathfrak{g}_0'$ is of the form 
$\sigma '(X)=J_{2p}'XJ_{2p}'$ where $J_{2p}'$ is given by (\ref{eq:J_{2p}'}). 
This means that $(\mathfrak{g}_0',\sigma ';\theta ')$ is the same as the non-compact semisimple symmetric 
pair equipped with a Cartan involution treated in Example \ref{eg:so-gl}, 
and then 
$(\mathfrak{g}_0')^{\sigma '}$ is isomorphic to $\mathfrak{gl}(2p,\mathbb{R})$. 
\end{eg}

The following corollary is an immediate consequence of Theorem \ref{thm:duality-thm} 
(see also Section \ref{subsec:correspondence}). 
Here is a remark that 
$\mathfrak{g}_0^{\sigma }$ is reductive since it is $\theta $-invariant, 
however, it is not always semisimple. 
In this case, we would also say that 
$\mathfrak{g}'$ is a compact real form of the complexification $(\mathfrak{g}_0^{\sigma })_{\mathbb{C}}$ 
of $\mathfrak{g}_0^{\sigma }$ 
if $\mathfrak{g}'$ is compact and a real form of $(\mathfrak{g}_0^{\sigma })_{\mathbb{C}}$. 

\begin{corollary}
\label{cor:duality-thm}
Let $(\mathfrak{g}_0,\sigma ;\theta )$ be a non-compact semisimple symmetric pair 
equipped with a Cartan involution. 
We set $(\mathfrak{g},\theta _1,\theta _2):=(\mathfrak{g}_0,\sigma ;\theta )^*$. 
Then, we have: 
\begin{enumerate}
	\item The compact semisimple Lie algebra $\mathfrak{g}$ is a compact real form 
	of $\mathfrak{g}_{\mathbb{C}}$. 
	\item The fixed point set $\mathfrak{g}^{\theta _1}$ coincides with $\mathfrak{g}_0^{\theta }$. 
	\item The fixed point set $\mathfrak{g}^{\theta _2}$ is a compact real form 
	of $(\mathfrak{g}_0^{\sigma })_{\mathbb{C}}$. 
\end{enumerate}
\end{corollary}

\subsection{Compatible condition}
\label{subsec:compatible}

Our duality theorem (Theorem \ref{thm:duality-thm}) satisfies compatible conditions 
in the sense as follows. 
Here, we will write $(\mathfrak{g},\theta _1,\theta _2)^{a*}$ for 
$\Phi ((\mathfrak{g},\theta _1,\theta _2)^a)=\Phi (\mathfrak{g},\theta _1,\theta _1\theta _2)$ and 
$(\mathfrak{g},\theta _1,\theta _2)^{*a}$ for $(\Phi (\mathfrak{g},\theta _1,\theta _2))^a$, 
and so on. 

\begin{proposition}
\label{prop:compatible-c}
The following relations hold 
for $(\mathfrak{g},\theta _1,\theta _2)\in \mathfrak{T}$: 
\begin{align*}
(\mathfrak{g},\theta _1,\theta _2)^{a*}=(\mathfrak{g},\theta _1,\theta _2)^{*a},\quad 
(\mathfrak{g},\theta _1,\theta _2)^{d*}=(\mathfrak{g},\theta _1,\theta _2)^{*d}. 
\end{align*}
\end{proposition}

\begin{proposition}
\label{prop:compatible-n}
The following relations hold 
for $(\mathfrak{g}_0,\sigma ;\theta )\in \mathfrak{P}_c$: 
\begin{align*}
(\mathfrak{g}_0,\sigma ;\theta )^{a*}=(\mathfrak{g}_0,\sigma ;\theta )^{*a},\quad 
(\mathfrak{g}_0,\sigma ;\theta )^{d*}=(\mathfrak{g}_0,\sigma ;\theta )^{*d}.  
\end{align*}\end{proposition}

\begin{proof}[Proof of Proposition \ref{prop:compatible-c}]
As $(\mathfrak{g},\theta _1,\theta _2)^a=(\mathfrak{g},\theta _1,\theta _1\theta _2)$, 
it follows from (\ref{eq:pair}) that 
$\Phi ((\mathfrak{g},\theta _1,\theta _2)^a)=\Phi (\mathfrak{g},\theta _1,\theta _1\theta _2)
=(\mathfrak{g}^{\theta _1}+\sqrt{-1}\mathfrak{g}^{-\theta _1},\theta _1\theta _2;\theta _1)$. 
On the other hand, 
the associated non-compact semisimple symmetric pair $(\Phi (\mathfrak{g},\theta _1,\theta _2))^a$ 
of $\Phi (\mathfrak{g},\theta _1,\theta _2)
=(\mathfrak{g}^{\theta _1}+\sqrt{-1}\mathfrak{g}^{-\theta _1},\theta _2;\theta _1)$ equals 
$(\mathfrak{g}^{\theta _1}+\sqrt{-1}\mathfrak{g}^{-\theta _1},\theta _1\theta _2;\theta _1)$ 
(see Definition \ref{def:associated}). 
Hence, we obtain 
\begin{align*}
(\mathfrak{g},\theta _1,\theta _2)^{a*}
=(\mathfrak{g}^{\theta _1}+\sqrt{-1}\mathfrak{g}^{-\theta _1},\theta _1\theta _2;\theta _1)
=(\mathfrak{g},\theta _1,\theta _2)^{*a}. 
\end{align*}

Next, a direct computation shows 
\begin{align*}
(\mathfrak{g},\theta _1,\theta _2)^{d*}
=(\mathfrak{g},\theta _2,\theta _1)^*
=(\mathfrak{g}^{\theta _2}+\sqrt{-1}\mathfrak{g}^{-\theta _2},\theta _1;\theta _2). 
\end{align*}
On the other hand, we have 
\begin{align*}
(\mathfrak{g},\theta _1,\theta _2)^{*d}
&=(\mathfrak{g}^{\theta _1}+\sqrt{-1}\mathfrak{g}^{-\theta _1},\theta _2;\theta _1)^d
=((\mathfrak{g}^{\theta _1}+\sqrt{-1}\mathfrak{g}^{-\theta _1})^d,\theta _1;\theta _2). 
\end{align*}
By (\ref{eq:dual}), we have 
$(\mathfrak{g}^{\theta _1}+\sqrt{-1}\mathfrak{g}^{-\theta _1})^d
=\mathfrak{g}^{\theta _2}+\sqrt{-1}\mathfrak{g}^{-\theta _2}$, 
from which 
\begin{align*}
(\mathfrak{g},\theta _1,\theta _2)^{*d}
=(\mathfrak{g}^{\theta _2}+\sqrt{-1}\mathfrak{g}^{-\theta _2},\theta _1;\theta _2). 
\end{align*}
Hence, we have proved $(\mathfrak{g},\theta _1,\theta _2)^{d*}=
(\mathfrak{g},\theta _1,\theta _2)^{*d}$. 
\end{proof}

\begin{proof}[Proof of Proposition \ref{prop:compatible-n}]
Our proof of the first equality is provided as follows. 
According to Theorem \ref{thm:duality-thm}, 
we write $(\mathfrak{g}_0,\sigma ;\theta )=(\mathfrak{g},\theta _1,\theta _2)^*$ 
for some $(\mathfrak{g},\theta _1,\theta _2)\in \mathfrak{T}$. 
Recall from the proof of Theorem \ref{thm:duality-thm} that 
$(\mathfrak{g},\theta _1,\theta _2)^{**}=\Psi \circ \Phi (\mathfrak{g},\theta _1,\theta _2)
=(\mathfrak{g},\theta _1,\theta _2)$. 
By Proposition \ref{prop:compatible-c}, we have 
\begin{align*}
(\mathfrak{g}_0,\sigma ;\theta )^{a*}
&=(\mathfrak{g},\theta _1,\theta _2)^{*a*}
=(\mathfrak{g},\theta _1,\theta _2)^{a**}
=(\mathfrak{g},\theta _1,\theta _2)^{a}. 
\end{align*}
On the other hand, 
since $(\mathfrak{g}_0,\sigma ;\theta )^*=(\mathfrak{g},\theta _1,\theta _2)$, 
we have 
$(\mathfrak{g}_0,\sigma ;\theta )^{*a}=(\mathfrak{g},\theta _1,\theta _2)^a$. 
Hence, we have proved $(\mathfrak{g}_0,\sigma ;\theta )^{a*}=(\mathfrak{g}_0,\sigma ;\theta )^{*a}$. 

The second statement can be proved similarly to the first one, 
hence we omit its proof. 
\end{proof}

\subsection{Duality for Riemannian symmetric pairs}
\label{subsec:riemannian}

The duality for Riemannian semisimple symmetric pairs \`a la \'E. Cartan 
means a one-to-one correspondence between Riemannian semisimple symmetric pairs of non-compact type 
and those of compact type. 
More precisely, 
a pair $(\mathfrak{g}_0,\theta )$ of a non-compact real semisimple Lie algebra $\mathfrak{g}_0$ 
and its Cartan involution $\theta $ corresponds to the pair 
$(\mathfrak{g}_0^{\theta }+\sqrt{-1}\mathfrak{g}_0^{-\theta },\theta )$, 
and a compact semisimple symmetric pair $(\mathfrak{g},\theta _1)$ 
to the pair $(\mathfrak{g}^{\theta _1}+\sqrt{-1}\mathfrak{g}^{-\theta _1},\theta _1)$ 
(cf. \cite[Section 2 in Chapter V]{helgason}). 
On the other hand, 
$(\mathfrak{g}_0,\theta )$ and $(\mathfrak{g},\theta _1)$ have the following correspondences 
in the sense of our duality theorem given by Theorem \ref{thm:duality-thm}: 
\begin{align*}
(\mathfrak{g}_0,\theta )^*
&=(\mathfrak{g}_0,\theta ;\theta )^*
=(\mathfrak{g}_0^{\theta }+\sqrt{-1}\mathfrak{g}_0^{-\theta },\theta ,\theta ), \\
(\mathfrak{g},\theta _1)^*
&=(\mathfrak{g},\theta _1,\theta _1)^*
=(\mathfrak{g}^{\theta _1}+\sqrt{-1}\mathfrak{g}^{-\theta _1},\theta _1;\theta _1) 
\end{align*}
where $(\mathfrak{g},\theta _1)$ is regarded as the commutative compact semisimple symmetric triad 
$(\mathfrak{g},\theta _1,\theta _1)$. 
Then, 
it is reasonable to write $(\mathfrak{g},\theta _1)=(\mathfrak{g}_0,\theta )^*$ 
as the duality for Riemannian semisimple symmetric pairs. 
In this context, 
our duality in Theorem \ref{thm:duality-thm} is a kind of generalizations of the duality 
for Riemannian symmetric pairs. 

We want to understand our duality in Theorem \ref{thm:duality-thm} 
from the viewpoint of the duality for Riemannian symmetric pairs. 
For this, we have to extend the duality for Riemannian semisimple symmetric pairs 
to that for Riemannian reductive symmetric pairs. 
Then, let us assume only in this subsection 
that a non-compact real Lie algebra $\mathfrak{g}_0$ is reductive. 

We write $\mathfrak{g}_0=[\mathfrak{g}_0,\mathfrak{g}_0]\oplus \mathfrak{z}_0$ 
for the decomposition of $\mathfrak{g}_0$ into its semisimple part and its center. 
Here, $[\mathfrak{g}_0,\mathfrak{g}_0]$ is the derived ideal and 
$\mathfrak{z}_0$ is the center of $\mathfrak{g}_0$. 

Let $\nu$ be an involution on $\mathfrak{g}_0$. 
The restriction of $\nu $ 
to the semisimple part $[\mathfrak{g}_0,\mathfrak{g}_0]$ defines an involution 
on $[\mathfrak{g}_0,\mathfrak{g}_0]$ because 
\begin{align*}
\nu ([\mathfrak{g}_0,\mathfrak{g}_0])
=[\nu (\mathfrak{g}_0),\nu (\mathfrak{g}_0)]
=[\mathfrak{g}_0,\mathfrak{g}_0]. 
\end{align*}
Moreover, $\nu |_{\mathfrak{z}_0}$ is also an involution on $\mathfrak{z}_0$ since 
\begin{align*}
[\nu (\mathfrak{z}_0),\mathfrak{g}_0]
=\nu ([\mathfrak{z}_0,\nu (\mathfrak{g}_0)])
=\nu ([\mathfrak{z}_0,\mathfrak{g}_0])
=\{ 0\} . 
\end{align*}

An involution $\theta$ on a real reductive Lie algebra $\mathfrak{g}_0$ is said to be a Cartan involution 
if the restriction $\theta |_{[\mathfrak{g}_0,\mathfrak{g}_0]}$ 
is a Cartan involution of the semisimple part $[\mathfrak{g}_0,\mathfrak{g}_0]$ (see \cite{rossmann}). 
In this sense, Cartan involutions on a real reductive Lie algebra are not unique. 

Given a Cartan involution $\theta $ of a reductive Lie algebra $\mathfrak{g}_0$, 
we say that $(\mathfrak{g}_0,\theta )$ is a Riemannian reductive symmetric pair. 
According to the decomposition $\mathfrak{g}_0=[\mathfrak{g}_0,\mathfrak{g}_0]\oplus \mathfrak{z}_0$, 
the sets $\mathfrak{g}_0^{\theta }$ and $\mathfrak{g}_0^{-\theta }$ can be written 
by $\mathfrak{g}_0^{\theta }
=[\mathfrak{g}_0,\mathfrak{g}_0]^{\theta }\oplus \mathfrak{z}_0^{\theta }$ 
and $\mathfrak{g}_0^{-\theta }
=[\mathfrak{g}_0,\mathfrak{g}_0]^{-\theta }\oplus \mathfrak{z}_0^{-\theta }$, respectively. 
Then, the Lie algebra 
$\mathfrak{g}:=\mathfrak{g}_0^{\theta }+\sqrt{-1}\mathfrak{g}_0^{-\theta }$ is given by 
\begin{align}
\label{eq:riemann-g}
\mathfrak{g}
%&=([\mathfrak{g}_0,\mathfrak{g}_0]^{\theta }\oplus \mathfrak{z}_0^{\theta })
%	+\sqrt{-1}([\mathfrak{g}_0,\mathfrak{g}_0]^{-\theta }\oplus \mathfrak{z}_0^{-\theta })\\
&=([\mathfrak{g}_0,\mathfrak{g}_0]^{\theta }+\sqrt{-1}[\mathfrak{g}_0,\mathfrak{g}_0]^{-\theta })
	\oplus (\mathfrak{z}_0^{\theta }+\sqrt{-1}\mathfrak{z}_0^{-\theta }). 
\end{align}
In particular, $\mathfrak{g}^{\theta }$ coincides with 
$\mathfrak{g}_0^{\theta }=[\mathfrak{g}_0,\mathfrak{g}_0]^{\theta }\oplus \mathfrak{z}_0^{\theta }$. 

Here, $[\mathfrak{g}_0,\mathfrak{g}_0]^{\theta }+\sqrt{-1}[\mathfrak{g}_0,\mathfrak{g}_0]^{-\theta }$ 
is a compact semisimple Lie algebra 
because $[\mathfrak{g}_0,\mathfrak{g}_0]$ is semisimple 
and $\theta $ is a Cartan involution of $[\mathfrak{g}_0,\mathfrak{g}_0]$, 
in particular, this coincides with $[\mathfrak{g},\mathfrak{g}]$. 
On the other hand, 
$\mathfrak{z}:=\mathfrak{z}_0^{\theta }+\sqrt{-1}\mathfrak{z}_0^{-\theta }$ 
is the center of $\mathfrak{g}$. 
Hence, $\mathfrak{g}$ is a compact Lie algebra. 
Further, $\mathfrak{g}$ is semisimple if and only if $\mathfrak{z}_0=\{ 0\} $. 
In this case, 
we have $\mathfrak{g}_0=[\mathfrak{g}_0,\mathfrak{g}_0]$, 
and then (\ref{eq:riemann-g}) is $\mathfrak{g}=\mathfrak{g}_0^{\theta }+\sqrt{-1}\mathfrak{g}_0^{-\theta }$. 
Henceforth, 
we may also write $(\mathfrak{g},\theta )=(\mathfrak{g}_0,\theta )^*$ 
for a real reductive Lie algebra $\mathfrak{g}_0$. 

Moreover, it follows from (\ref{eq:riemann-g}) that $(\mathfrak{g},\theta )$ is of the form 
\begin{align*}
(\mathfrak{g},\theta )=([\mathfrak{g},\mathfrak{g}],\theta )\oplus 
	(\mathfrak{z},\theta ). 
\end{align*}
The semisimple part $([\mathfrak{g},\mathfrak{g}],\theta )$ is written by 
$([\mathfrak{g},\mathfrak{g}],\theta )=([\mathfrak{g}_0,\mathfrak{g}_0],\theta )^*
%=([\mathfrak{g}_0,\mathfrak{g}_0],\theta ;\theta )^*
$ 
in the sense of Theorem \ref{thm:duality-thm}. 
Concerning the center, we set $(\mathfrak{z}_0,\theta )^*:=(\mathfrak{z},\theta )$. 
Based on the above observation, we define 
\begin{align}
\label{eq:duality-reductive}
(\mathfrak{g}_0,\theta )^*:=([\mathfrak{g}_0,\mathfrak{g}_0],\theta )^*
\oplus (\mathfrak{z}_0,\theta )^* 
\end{align}
for a real reductive Lie algebra $\mathfrak{g}_0$ and a Cartan involution $\theta $ of $\mathfrak{g}_0$. 

Now, we return to our duality theorem. 
Let $(\mathfrak{g}_0,\sigma ;\theta )\in \mathfrak{P}_c$ 
and $(\mathfrak{g},\theta _1,\theta _2)=(\mathfrak{g}_0,\sigma ;\theta )^*\in \mathfrak{T}$. 
Then, $(\mathfrak{g}_0,\theta )$ and $(\mathfrak{g}_0^{\sigma },\theta )$ are 
non-compact Riemannian symmetric pairs, 
and $(\mathfrak{g},\theta _1)$ and $(\mathfrak{g}^{\theta _2},\theta _1)$ are 
compact Riemannian symmetric pairs. 
They have the following relations: 

\begin{theorem}
\label{thm:duality-g}
Let $(\mathfrak{g}_0,\sigma ;\theta )$ be a non-compact pseudo-Riemannian semisimple symmetric pair 
equipped with a Cartan involution. 
We set $(\mathfrak{g},\theta _1,\theta _2):=(\mathfrak{g}_0,\sigma ;\theta )^*$. 
Then, $(\mathfrak{g},\theta _1)$ equals $(\mathfrak{g}_0,\theta )^*$ and 
$(\mathfrak{g}^{\theta _2},\theta _1)$ equals $(\mathfrak{g}_0^{\sigma },\theta )^*$. 
\end{theorem}

\begin{proof}
As mentioned in (\ref{eq:triad}), we have 
$\mathfrak{g}=\mathfrak{g}_0^{\theta }+\sqrt{-1}\mathfrak{g}_0^{-\theta }$ 
and 
$\mathfrak{g}^{\theta _2}=\mathfrak{g}_0^{\theta }\cap \mathfrak{g}_0^{\sigma }
+\sqrt{-1}\mathfrak{g}_0^{-\theta }\cap \mathfrak{g}_0^{\sigma }$. 
\end{proof}

\begin{eg}
Let us again consider the non-compact semisimple symmetric pair equipped with a Cartan involution 
$(\mathfrak{g}_0,\sigma ;\theta )$ 
which has been treated in Example \ref{eg:so-gl}, namely, 
$\mathfrak{g}_0=\mathfrak{so}(2p,2p)$, $\sigma (X)=J_{2p}'XJ_{2p}'$ 
and $\theta (X)=I_{2p,2p}XI_{2p,2p}$ ($X\in \mathfrak{g}_0$). 
Then, $\mathfrak{g}_0^{\theta }=\mathfrak{so}(2p)+\mathfrak{so}(2p)$, 
and $\mathfrak{g}_0^{\sigma }\simeq \mathfrak{gl}(2p,\mathbb{R})$. 
% via $\iota $ (see (\ref{eq:gl(2p)-isom})). 

We have already explained in Example \ref{eg:duality-so-gl} that 
$(\mathfrak{g}_0,\sigma ;\theta )$ is equivalent to the dual 
of the commutative compact semisimple symmetric triad $(\mathfrak{g},\theta _1,\theta _2)$ 
with $\mathfrak{g}=\mathfrak{so}(4p)$, $\theta _1(X)=I_{2p,2p}XI_{2p,2p}$ and 
$\theta _2(X)=J_{2p}XJ_{2p}^{-1}$ ($X\in \mathfrak{g}$), 
from which $\mathfrak{g}^{\theta _1}=\mathfrak{so}(2p)+\mathfrak{so}(2p)$ 
$\mathfrak{g}^{\theta _2}\simeq \mathfrak{u}(2p)$. 
Here is the observation of the relation 
$(\mathfrak{g}_0,\sigma ;\theta )^*\equiv (\mathfrak{g},\theta _1,\theta _2)$ 
from the view point of Theorem \ref{thm:duality-g}. 
In the following, 
we shall use the classification of Riemannian semisimple symmetric pairs. 

Clearly, $(\mathfrak{g},\theta _1)=(\mathfrak{so}(4p),\theta _1)$ is the dual of 
$(\mathfrak{g}_0,\theta )=(\mathfrak{so}(2p,2p),\theta )$ for Riemannian semisimple symmetric pairs. 

Next, the subalgebra $\mathfrak{g}_0^{\sigma }$ is not semisimple and decomposed as 
$\mathfrak{g}_0^{\sigma }
=[\mathfrak{g}_0^{\sigma },\mathfrak{g}_0^{\sigma }]\oplus \mathfrak{z}(\mathfrak{g}_0^{\sigma })$ 
where the semisimple part is 
\begin{align*}
[\mathfrak{g}_0^{\sigma },\mathfrak{g}_0^{\sigma }]
=\left\{ 
	\left( 
		\begin{array}{cc}
		A & B \\
		B & A 
		\end{array}
	\right) :A\in \mathfrak{so}(2p),B\in \operatorname{Sym}(2p,\mathbb{R}),\operatorname{Tr}B=0
\right\} 
\end{align*}
which is isomorphic to $\mathfrak{sl}(2p,\mathbb{R})$ 
via the isomorphism $\iota $ given by (\ref{eq:gl(2p)-isom}), and the center is 
\begin{align*}
\mathfrak{z}(\mathfrak{g}_0^{\sigma })
=\left\{ 
	\left( 
		\begin{array}{cc}
		O & rI_{2p} \\
		rI_{2p} & O 
		\end{array}
	\right) :r\in \mathbb{R}
\right\} 
\end{align*}
which is isomorphic to $\iota (\mathfrak{z}(\mathfrak{g}_0^{\sigma }))=\mathbb{R}I_{2p}$. 
According to (\ref{eq:duality-reductive}), 
the dual $(\mathfrak{g}_0^{\sigma },\theta )^*$ is given by 
\begin{align*}
(\mathfrak{g}_0^{\sigma },\theta )^*
&=([\mathfrak{g}_0^{\sigma },\mathfrak{g}_0^{\sigma }],\theta )^*
\oplus (\mathfrak{z}(\mathfrak{g}_0^{\sigma }),\theta )^*\\
&\equiv 
(\mathfrak{sl}(2p,\mathbb{R}),\iota \theta \iota ^{-1})^*\oplus (\mathbb{R}I_{2p},\iota \theta \iota ^{-1})^*. 
\end{align*}
Here, the Cartan involution $\theta ':=\iota \theta \iota ^{-1}$ 
on $\iota (\mathfrak{g}_0^{\sigma })=\mathfrak{gl}(2p,\mathbb{R})$ 
is written as $\theta '(X')=-\,{}^tX'$ ($X'\in \mathfrak{gl}(2p,\mathbb{R})$). 
The dual $(\mathfrak{sl}(2p,\mathbb{R}),\theta ')^*$ 
for Riemannian semisimple symmetric pair coincides with $(\mathfrak{su}(2p),\theta ')$. 
On the other hand, 
$\mathbb{R}I_{2p}$ coincides with $(\mathbb{R}I_{2p})^{-\theta '}$. 
By (\ref{eq:riemann-g}), we obtain $(\mathbb{R}I_{2p},\theta ')^*=(\sqrt{-1}\mathbb{R}I_{2p},\theta ')$. 
Therefore, we conclude 
\begin{align*}
(\mathfrak{g}_0^{\sigma },\theta )^*
\equiv (\mathfrak{su}(2p),\theta ')\oplus (\sqrt{-1}\mathbb{R}I_{2p},\theta ')
=(\mathfrak{u}(2p),\theta '). 
\end{align*}
\end{eg}

%%%%%%%%%%%%%%%%%%%%%%%%%%%%%%%%%%%%%%%%%%%%%%%%%%%%%%%%%%%%%%%%%%%%%%%%%%%%%%%%%%%%%%%%%%

\section{Correspondence between various properties via duality theorem}
\label{sec:property}

In the previous section, 
we have provided an explicit description of a one-to-one correspondence 
between the set of equivalence classes of non-compact semisimple symmetric pairs 
and the set of equivalence classes of commutative compact semisimple symmetric triads 
(Theorem \ref{thm:duality-thm}), 
and have observed that this correspondence is a kind of generalization of the duality 
for Riemannian symmetric pairs (see Section \ref{subsec:riemannian}). 

Our next concern on our duality theorem is to seek a correspondence between various properties of 
non-compact semisimple symmetric pairs and those of commutative compact semisimple symmetric triads. 

The first half of this section is devoted to the study 
that our duality theorem preserves irreducibility, namely, 
a non-compact semisimple symmetric pair $(\mathfrak{g}_0,\sigma )$ is irreducible 
(see Definition \ref{def:irr-pair}) if and only if 
the corresponding commutative compact semisimple symmetric triad $(\mathfrak{g},\theta _1,\theta _2)
=(\mathfrak{g}_0,\sigma )^*$ is irreducible (see Definition \ref{def:irr-triad}), 
which is given by Theorem \ref{thm:corresp-irreducible}. 
Further, 
we see in Lemma \ref{lem:duality-irr-decomp} that the irreducible decomposition of $(\mathfrak{g}_0,\sigma )$ 
corresponds to the irreducible decomposition of the dual 
$(\mathfrak{g},\theta _1,\theta _2)=(\mathfrak{g}_0,\sigma )^*$. 
Moreover, our duality also gives rise to a one-to-one correspondence 
between `types' of irreducible non-compact semisimple symmetric pairs 
and `types' of irreducible commutative compact semisimple symmetric triads. 
For example, 
if $\mathfrak{g}_0$ is simple and has no complex structures, 
then $(\mathfrak{g}_0,\sigma )$ is irreducible and the corresponding $(\mathfrak{g},\theta _1,\theta _2)$ 
satisfies that $\mathfrak{g}$ is simple, in particular, $(\mathfrak{g},\theta _1,\theta _2)$ is irreducible, 
and vice versa (see Proposition \ref{prop:simple}). 
If $\mathfrak{g}_0$ is not simple or has a complex structure, 
then we divide irreducible $(\mathfrak{g}_0,\sigma )\in \mathfrak{P}$ 
into four types 
named by (P-\ref{item:cpx-anti})--(P-\ref{item:non-linear}) (see in Corollary \ref{cor:irr-pair}), 
whereas, if a compact semisimple Lie algebra $\mathfrak{g}$ is not simple, 
then we divide irreducible $(\mathfrak{g},\theta _1,\theta _2)\in \mathfrak{T}$ into four types 
named by (T-\ref{type:i2})--(T-\ref{type:ii2d}) (see in Proposition \ref{prop:irr-triad}). 
Then, we prove that irreducible $(\mathfrak{g}_0,\sigma )$ is of type (P-i) 
if and only if the corresponding $(\mathfrak{g},\theta _1,\theta _2)$ is of type (T-i) 
(i $=$ a, b, c, d) (see Theorem \ref{thm:irreducible}). 

Owing to this result, 
one can give an alternative proof for Berger's classification of non-compact semisimple symmetric pairs 
from the viewpoint of that of compact semisimple symmetric triads via our duality theorem. 
The detail will be explained in the forthcoming paper \cite{bis1}. 

The latter half of this section is to provide a new characterization 
of non-compact semisimple symmetric pairs of type $K_{\varepsilon }$ (see Definition \ref{def:typeK}) 
in terms of the corresponding commutative compact semisimple symmetric triads $(\mathfrak{g}_0,\sigma )^*$ 
via our duality theorem (see Theorem \ref{thm:typeK}). 
To explain it, 
we shall introduce an equivalence relation between two involutions $\theta _1,\theta _2$ 
on a compact semisimple Lie algebra, denoted by $\theta _1\sim \theta _2$ (see Definition \ref{def:sim}). 

First of all, 
we fix the notation as follows. 
Let $\mathfrak{l}_0$ be a real semisimple Lie algebra. 
Then, the direct sum $\mathfrak{l}_0\oplus \mathfrak{l}_0$ is a semisimple Lie algebra. 
However, it is not simple even if $\mathfrak{l}_0$ is simple. 
For two automorphisms $\nu _1,\nu _2$ on $\mathfrak{l}_0$, 
we write $\nu _1\oplus \nu _2$ for the automorphism on $\mathfrak{l}_0\oplus \mathfrak{l}_0$ 
given by 
\begin{align}
\label{eq:oplus}
(\nu _1\oplus \nu _2)(X,Y):=(\nu _1(X),\nu _2(Y))\quad (X,Y\in \mathfrak{l}_0). 
\end{align}
On the other hand, 
it would be useful to define 
\begin{align}
\label{eq:twist}
\rho (X,Y)=(Y,X)\quad ((X,Y)\in \mathfrak{l}_0\oplus \mathfrak{l}_0). 
\end{align}
Then, $\rho $ is an involution on $\mathfrak{l}_0\oplus \mathfrak{l}_0$. 
%For an involution $\nu$ on $\mathfrak{l}_0$, 
%we define two involutions $\delta _{\nu}$ and $\tau _{\nu}$ on $\mathfrak{l}_0\oplus \mathfrak{l}_0$ by 
%\begin{align}
%\label{eq:delta-nu}
%\delta _{\nu }(X,Y)&=(\nu (X),\nu (Y)),\\
%\label{eq:tau-nu}
%\tau _{\nu }(X,Y)&=(\nu (Y),\nu (X)) 
%\end{align}
%for $(X,Y)\in \mathfrak{l}_0\oplus \mathfrak{l}_0$. 
%Note that $\tau _{\nu }$ coincides with $\tau \delta _{\nu}=\delta _{\nu }\tau $ 
%for any involution $\nu$ on $\mathfrak{l}_0$. 
%Clearly, $\tau ,\delta _{\nu}$ and $\tau _{\nu}$ commute with one another. 

\subsection{Irreducible semisimple symmetric pair}
\label{subsec:irr}

Let us begin with irreducibility for non-compact semisimple symmetric pairs 
(see \cite[Chapter XI, Section 5]{kobayashi-nomizu}). 

\begin{define}
\label{def:irr-pair}
A non-compact semisimple symmetric pair $(\mathfrak{g}_0,\sigma )$ is {\it irreducible} 
if it does not admit non-trivial $\sigma$-invariant ideals of $\mathfrak{g}_0$. 
\end{define}

When $\sigma$ is a Cartan involution $\theta$, 
then $(\mathfrak{g}_0,\theta )$ is irreducible in the sense of Definition \ref{def:irr-pair} 
if and only if $\mathfrak{g}_0^{\theta }$ acts irreducibly on $\mathfrak{g}_0^{-\theta }$, 
namely, 
$(\mathfrak{g}_0,\theta )$ is irreducible as a non-compact Riemannian symmetric pair, 
which will be discussed in Section \ref{sec:appendix} separated from this section. 

Let us consider a non-compact semisimple symmetric pair $(\mathfrak{g}_0,\sigma )$ 
which is not irreducible, namely, 
there exists a non-trivial $\sigma$-invariant ideal $\mathfrak{l}_0$ of $\mathfrak{g}_0$. 
The restriction $\sigma |_{\mathfrak{l}_0}$ of $\sigma$ to $\mathfrak{l}_0$ is 
an involution on $\mathfrak{l}_0$. 
This gives rise to a non-compact semisimple symmetric pair $(\mathfrak{l}_0,\sigma |_{\mathfrak{l}_0})$. 
Let $\mathfrak{l}_0'$ be the complementary ideal of $\mathfrak{l}_0$ in $\mathfrak{g}_0$, 
namely, 
$\mathfrak{g}_0=\mathfrak{l}_0\oplus \mathfrak{l}_0'$. 
Then, $\sigma |_{\mathfrak{l}_0'}$ becomes an involution of $\mathfrak{l}_0'$. 
Hence, $(\mathfrak{g}_0,\sigma )$ is decomposed into two semisimple symmetric pairs 
$(\mathfrak{l}_0,\sigma |_{\mathfrak{l}_0})$, $(\mathfrak{l}_0',\sigma |_{\mathfrak{l}_0'})$ as follows 
\begin{align}
\label{eq:irr-decomp}
(\mathfrak{g}_0,\sigma )
=(\mathfrak{l}_0\oplus \mathfrak{l}_0',\sigma |_{\mathfrak{l}_0}\oplus \sigma |_{\mathfrak{l}_0'})
=(\mathfrak{l}_0,\sigma |_{\mathfrak{l}_0})\oplus (\mathfrak{l}_0',\sigma |_{\mathfrak{l}_0'}). 
\end{align}

Let us see that (\ref{eq:irr-decomp}) leads us a decomposition of a non-compact semisimple symmetric pair 
equipped with a Cartan involution $(\mathfrak{g}_0,\sigma ;\theta )$. 
Let $\theta _0$ (resp. $\theta _0')$ be 
a Cartan involution of $\mathfrak{l}_0$ (resp. $\mathfrak{l}_0'$) commuting with 
$\sigma |_{\mathfrak{l}_0}$ (resp. $\sigma |_{\mathfrak{l}_0'}$). 
Then, $\theta _0\oplus \theta _0'$ 
%the map 
%\begin{align}
%\label{eq:cartan-tilde}
%\widehat{\theta }:\mathfrak{l}_0\oplus \mathfrak{l}_0'\to 
%\mathfrak{l}_0\oplus \mathfrak{l}_0',\quad (X,X')\mapsto (\theta _0(X),\theta _0'(X')) 
%\end{align}
is a Cartan involution of $\mathfrak{g}_0=\mathfrak{l}_0\oplus \mathfrak{l}_0'$ commuting with $\sigma $. 
Hence, $(\mathfrak{g}_0,\sigma ;\theta _0\oplus \theta _0')$ is 
a non-compact semisimple symmetric pair equipped with a Cartan involution 
and has a decomposition 
\begin{align*}
(\mathfrak{g}_0,\sigma ;\theta _0\oplus \theta _0')
=(\mathfrak{l}_0,\sigma |_{\mathfrak{l}_0};\theta _0)
\oplus (\mathfrak{l}_0',\sigma |_{\mathfrak{l}_0'};\theta _0'). 
\end{align*}

\begin{lemma}
\label{lem:inv-cartan}
Suppose that $(\mathfrak{g}_0,\sigma )$ is not irreducible and decomposed into (\ref{eq:irr-decomp}). 
Then, any Cartan involution $\theta $ of $\mathfrak{g}_0$ commuting with $\sigma$ satisfies 
$\theta (\mathfrak{l}_0)=\mathfrak{l}_0$ and $\theta (\mathfrak{l}_0')=\mathfrak{l}_0'$. 
\end{lemma}

\begin{proof}
Retain the notation as above. 
By Fact \ref{fact:loos}, we have 
$\theta =e^{\operatorname{ad}Z}(\theta _0\oplus \theta _0')e^{-\operatorname{ad}Z}$ 
for some $Z\in \mathfrak{g}_0^{\sigma }$. 
Here, $\mathfrak{g}_0^{\sigma }$ is of the form 
$\mathfrak{g}_0^{\sigma }=\mathfrak{l}_0^{\sigma }\oplus (\mathfrak{l}_0')^{\sigma} $, 
from which we write $Z=(Z_0,Z_0')$ for some $Z_0\in \mathfrak{l}_0^{\sigma }$ 
and $Z_0'\in (\mathfrak{l}_0')^{\sigma }$. 
Then, we have $e^{\operatorname{ad}Z}=e^{\operatorname{ad}Z_0}\oplus e^{\operatorname{ad}Z_0'}$ and 
\begin{align*}
\theta 
&=(e^{\operatorname{ad}Z_0}\oplus e^{\operatorname{ad}Z_0'})(\theta _0\oplus \theta _0')
	(e^{-\operatorname{ad}Z_0}\oplus e^{-\operatorname{ad}Z_0'})\\
&=(e^{\operatorname{ad}Z_0}\theta _0e^{-\operatorname{ad}Z_0})\oplus 
	(e^{\operatorname{ad}Z_0'}\theta _0'e^{-\operatorname{ad}Z_0'}). 
\end{align*}
Hence, $\theta |_{\mathfrak{l}_0}=e^{\operatorname{ad}Z_0}\theta _0e^{-\operatorname{ad}Z_0}$ 
and $\theta |_{\mathfrak{l}_0'}=e^{\operatorname{ad}Z_0'}\theta _0'e^{-\operatorname{ad}Z_0'}$ 
become automorphisms on $\mathfrak{l}_0$ and $\mathfrak{l}_0'$, respectively. 
This means that both $\mathfrak{l}_0$ and $\mathfrak{l}_0'$ are $\theta $-invariant. 
\end{proof}

Lemma \ref{lem:inv-cartan} implies that 
$\theta |_{\mathfrak{l}_0},\theta |_{\mathfrak{l}_0'}$ are Cartan involutions of 
$\mathfrak{l}_0,\mathfrak{l}_0'$, respectively. 
Therefore, we conclude: 

\begin{proposition}
\label{prop:irr-decomp}
Let $(\mathfrak{g}_0,\sigma ;\theta )$ be a non-compact semisimple symmetric pair 
equipped with a Cartan involution. 
If $(\mathfrak{g}_0,\sigma )$ is not irreducible, 
then there exist $\sigma$-invariant ideals $\mathfrak{l}_0,\mathfrak{l}_0'$ such that 
\begin{align*}
(\mathfrak{g}_0,\sigma ;\theta )
=(\mathfrak{l}_0,\sigma |_{\mathfrak{l}_0};\theta |_{\mathfrak{l}_0})
\oplus (\mathfrak{l}_0',\sigma |_{\mathfrak{l}_0'};\theta |_{\mathfrak{l}_0'}). 
\end{align*}
\end{proposition}

Due to Proposition \ref{prop:irr-decomp}, 
a non-compact semisimple symmetric pair equipped with a Cartan involution $(\mathfrak{g}_0,\sigma ;\theta )$ 
is decomposed into irreducible ones, namely, 
\begin{align*}
(\mathfrak{g}_0,\sigma ;\theta )
=(\mathfrak{l}_0,\sigma |_{\mathfrak{l}_0};\theta |_{\mathfrak{l}_0})\oplus 
(\mathfrak{l}_0',\sigma |_{\mathfrak{l}_0'};\theta |_{\mathfrak{l}_0'})\oplus \cdots \oplus 
(\mathfrak{l}_0^{(k)},\sigma |_{\mathfrak{l}_0^{(k)}};\theta |_{\mathfrak{l}_0^{(k)}})
\end{align*}
where $\mathfrak{l}_0,\mathfrak{l}_0',\ldots ,\mathfrak{l}_0^{(k)}$ are 
non-trivial $\sigma$-invariant ideals of $\mathfrak{g}_0$ such that 
$\mathfrak{g}_0=\mathfrak{l}_0\oplus \mathfrak{l}_0'\oplus \cdots \oplus \mathfrak{l}_0^{(k)}$ 
and all non-compact semisimple symmetric pairs $(\mathfrak{l}_0^{(i)},\sigma |_{\mathfrak{l}_0^{(i)}})$ 
are irreducible. 

In view of the following lemma, 
it suffices to consider only irreducible symmetric pairs 
in order to see a one-to-one correspondence between various properties of 
non-compact semisimple symmetric pairs and those of commutative compact semisimple symmetric triads. 

\begin{lemma}
\label{lem:duality-irr-decomp}
Retain the setting of Proposition \ref{prop:irr-decomp}. 
Then, the dual $(\mathfrak{g}_0,\sigma ;\theta )^*$ is given by 
\begin{align*}
(\mathfrak{g}_0,\sigma ;\theta )^*=(\mathfrak{l}_0,\sigma |_{\mathfrak{l}_0};\theta |_{\mathfrak{l}_0})^*
\oplus (\mathfrak{l}_0',\sigma |_{\mathfrak{l}_0'};\theta |_{\mathfrak{l}_0'})^*. 
\end{align*}
\end{lemma}

\begin{proof}
%Recall the Cartan involution $\widehat{\theta}$ on $\mathfrak{g}_0$ defined by (\ref{eq:cartan-tilde}). 
%As $(\mathfrak{g}_0,\sigma ;\theta )$ is equivalent to $(\mathfrak{g}_0,\sigma ;\widehat{\theta })$, 
%it follows from Theorem \ref{thm:duality-thm} that 
%$(\mathfrak{g}_0,\sigma ;\theta )^*\equiv (\mathfrak{g}_0,\sigma ;\widehat{\theta })^*$. 
%Then, it suffices to show Lemma \ref{lem:duality-irr-decomp} for $\theta =\widehat{\theta }$. 
We set $(\mathfrak{g},\theta _1,\theta _2):=(\mathfrak{g}_0,\sigma ;\theta )^*$. 
By (\ref{eq:triad}), we obtain 
$\mathfrak{g}=
\mathfrak{g}_0^{\theta }+\sqrt{-1}\mathfrak{g}_0^{-\theta }$, 
$\theta _1=\theta |_{\mathfrak{g}}$ and $\theta _2=\sigma |_{\mathfrak{g}}$. 
Since $\mathfrak{g}_0^{\theta }=\mathfrak{l}_0^{\theta }\oplus (\mathfrak{l}_0')^{\theta }$ 
and $\mathfrak{g}_0^{-\theta }=\mathfrak{l}_0^{-\theta }\oplus (\mathfrak{l}_0')^{-\theta }$, 
the compact Lie algebra $\mathfrak{g}$ is written as 
\begin{align*}
\mathfrak{g}
&=(\mathfrak{l}_0^{\theta }\oplus (\mathfrak{l}_0')^{\theta })
	+\sqrt{-1}(\mathfrak{l}_0^{-\theta }\oplus (\mathfrak{l}_0')^{-\theta })\\
&=(\mathfrak{l}_0^{\theta }+\sqrt{-1}\mathfrak{l}_0^{-\theta })\oplus 
	((\mathfrak{l}_0')^{\theta }+\sqrt{-1}(\mathfrak{l}_0')^{-\theta }). 
\end{align*}
We set $\mathfrak{l}:=\mathfrak{l}_0^{\theta }+\sqrt{-1}\mathfrak{l}_0^{-\theta }$ and 
$\mathfrak{l}':=(\mathfrak{l}_0')^{\theta }+\sqrt{-1}(\mathfrak{l}_0')^{-\theta }$. 
Then, we have 
\begin{align*}
\theta _1|_{\mathfrak{l}}=\theta |_{\mathfrak{l}},\quad 
\theta _1|_{\mathfrak{l}'}=\theta |_{\mathfrak{l}'},\quad 
\theta _2|_{\mathfrak{l}}=\sigma |_{\mathfrak{l}},\quad 
\theta _2|_{\mathfrak{l}'}=\sigma |_{\mathfrak{l}'}. 
\end{align*}
This implies that 
\begin{align*}
(\mathfrak{g},\theta _1,\theta _2)
&=(\mathfrak{l}\oplus \mathfrak{l}',\theta |_{\mathfrak{l}\oplus \mathfrak{l}'},
	\sigma |_{\mathfrak{l}\oplus \mathfrak{l}'})
=(\mathfrak{l},\theta |_{\mathfrak{l}},\sigma |_{\mathfrak{l}})
	\oplus (\mathfrak{l}',\theta |_{\mathfrak{l}'},\sigma |_{\mathfrak{l}'}). 
\end{align*}
As 
$(\mathfrak{l},\theta |_{\mathfrak{l}},\sigma |_{\mathfrak{l}})
=(\mathfrak{l}_0,\sigma |_{\mathfrak{l}_0};\theta |_{\mathfrak{l}_0})^*$ 
and $(\mathfrak{l}',\theta |_{\mathfrak{l}'},\sigma |_{\mathfrak{l}'})
=(\mathfrak{l}_0',\sigma |_{\mathfrak{l}_0'};\theta |_{\mathfrak{l}_0'})^*$, we obtain 
\begin{align*}
(\mathfrak{g}_0,\sigma ;\theta )^*
=(\mathfrak{l}_0,\sigma |_{\mathfrak{l}_0};\theta |_{\mathfrak{l}_0})^*
	\oplus (\mathfrak{l}_0',\sigma |_{\mathfrak{l}_0'};\theta |_{\mathfrak{l}_0'})^*. 
\end{align*}

Therefore, 
we have verified Lemma \ref{lem:duality-irr-decomp}. 
\end{proof}

Henceforth, 
we deal with irreducible non-compact semisimple symmetric pairs. 

First, let us consider a non-compact semisimple symmetric pair $(\mathfrak{g}_0,\sigma )$ 
where $\mathfrak{g}_0$ is not simple. 

We take a non-trivial simple ideal $\mathfrak{g}_0'$ in $\mathfrak{g}_0$. 
Then, $\sigma (\mathfrak{g}_0')$ is also an ideal of $\mathfrak{g}_0$ because 
\begin{align*}
[\sigma (\mathfrak{g}_0'),\mathfrak{g}_0]
=[\sigma (\mathfrak{g}_0'),\sigma (\mathfrak{g}_0)]
=\sigma ([\mathfrak{g}_0',\mathfrak{g}_0])\subset \sigma (\mathfrak{g}_0'). 
\end{align*}
Hence, 
$\mathfrak{g}_0'\cap \sigma (\mathfrak{g}_0')$ is a $\sigma$-invariant ideal of $\mathfrak{g}_0$. 
Since the sequence 
$\mathfrak{g}_0'\cap \sigma (\mathfrak{g}_0')\subset \mathfrak{g}_0'\subsetneq \mathfrak{g}_0$ holds 
and $(\mathfrak{g}_0,\sigma )$ is irreducible, 
$\mathfrak{g}_0'\cap \sigma (\mathfrak{g}_0')$ has to be $\{ 0\}$. 

Now, we consider the ideal $\mathfrak{g}_0'+\sigma (\mathfrak{g}_0')
=\{ X+\sigma (Y):X,Y\in \mathfrak{g}_0'\} $ of $\mathfrak{g}_0$. 
This is $\sigma$-invariant because $\sigma (X+\sigma (Y))=Y+\sigma (X)$ for $X,Y\in \mathfrak{g}_0'$. 
Then, $\mathfrak{g}_0'+\sigma (\mathfrak{g}_0')$ coincides with $\mathfrak{g}_0$ 
because of the inclusion $\mathfrak{g}_0'\subset \mathfrak{g}_0'+\sigma (\mathfrak{g}_0')$ 
and the irreducibility of $(\mathfrak{g}_0,\sigma )$. 
We define a map $\varphi $ from $\mathfrak{g}_0$ to $\mathfrak{g}_0'\oplus \mathfrak{g}_0'$ by 
\begin{align}
\label{eq:f-isom}
\varphi :\mathfrak{g}_0=\mathfrak{g}_0'+\sigma (\mathfrak{g}_0')
\to \mathfrak{g}_0'\oplus \mathfrak{g}_0',~X+\sigma (Y)\mapsto (X,Y). 
\end{align}
Then, $\varphi $ is a Lie algebra isomorphism. 
This yields the relation $\sigma =\varphi ^{-1}\rho \varphi $ where $\rho $ is given by (\ref{eq:twist}). 
Hence, we obtain 
\begin{align}
\label{eq:irr-not-simple}
(\mathfrak{g}_0,\sigma )\equiv (\mathfrak{g}_0'\oplus \mathfrak{g}_0',\rho ). 
\end{align}

Let $\theta '$ be a Cartan involution of a real simple Lie algebra $\mathfrak{g}_0'$. 
Then, $\theta '\oplus \theta '$ is a Cartan involution of $\mathfrak{g}_0'\oplus \mathfrak{g}_0'$. 
It is obvious that $(\theta '\oplus \theta ')\rho =\rho (\theta '\oplus \theta ')$. 
Hence, $(\mathfrak{g}_0'\oplus \mathfrak{g}_0',\rho ;\theta '\oplus \theta ')$ is a 
non-compact semisimple symmetric pair equipped with a Cartan involution. 

Here, we set $\widetilde{\theta }:=\varphi ^{-1}(\theta '\oplus \theta ')\varphi $. 
This is a Cartan involution of $\mathfrak{g}_0$ and 
\begin{align*}
\sigma \widetilde{\theta }
=(\varphi ^{-1}\rho \varphi )(\varphi ^{-1}(\theta '\oplus \theta ')\varphi )
=\varphi ^{-1}\rho (\theta '\oplus \theta ')\varphi 
=\varphi ^{-1}(\theta '\oplus \theta ')\rho \varphi 
=\widetilde{\theta }\sigma. 
\end{align*}
Hence, we obtain 
\begin{align}
\label{eq:irr-not-simple-cartan}
(\mathfrak{g}_0,\sigma ;\widetilde{\theta })
\equiv (\mathfrak{g}_0'\oplus \mathfrak{g}_0',\rho ;\theta '\oplus \theta '). 
\end{align}

Moreover, we prove: 

\begin{lemma}
\label{lem:irr-not-simple}
Suppose that $\mathfrak{g}_0$ is not simple and $(\mathfrak{g}_0,\sigma )$ is irreducible. 
Under the setting of (\ref{eq:irr-not-simple}), 
any Cartan involution $\theta $ of $\mathfrak{g}_0$ commuting with $\sigma$ 
satisfies $\theta (\mathfrak{g}_0')=\mathfrak{g}_0'$. 
\end{lemma}

\begin{proof}
By Fact \ref{fact:loos}, 
there exists $Z\in \mathfrak{g}_0^{\sigma }$ such that 
$\theta =e^{\operatorname{ad}Z}\widetilde{\theta }e^{-\operatorname{ad}Z}$. 
Since $e^{\operatorname{ad}Z}\sigma =\sigma e^{\operatorname{ad}Z}$ 
(see Lemma \ref{lem:conjugate-associated}), we have 
$(\mathfrak{g}_0,\sigma ;\theta )\equiv (\mathfrak{g}_0,\sigma ;\widetilde{\theta })$. 

As $\mathfrak{g}_0^{\sigma }=\{ X+\sigma (X):X\in \mathfrak{g}_0'\} $, 
we write $Z=Z_0+\sigma (Z_0)$ for some $Z_0\in \mathfrak{g}_0'$. 
Then, we have 
\begin{align*}
\theta (X+\sigma (Y))=(e^{\operatorname{ad}Z_0}\theta 'e^{-\operatorname{ad}Z_0})X+
	\sigma ((e^{\operatorname{ad}Z_0}\theta 'e^{-\operatorname{ad}Z_0})Y). 
\end{align*}
This means that $\mathfrak{g}_0'$ is $\theta $-invariant 
and $\theta |_{\mathfrak{g}_0'}=e^{\operatorname{ad}Z_0}\theta 'e^{-\operatorname{ad}Z_0}$ 
is its Cartan involution. 
\end{proof}

Lemma \ref{lem:irr-not-simple} implies that 
the restriction $\theta |_{\mathfrak{g}_0'}$ becomes a Cartan involution of $\mathfrak{g}_0'$. 
By replacing $\theta '$ with $\theta |_{\mathfrak{g}_0'}$ in (\ref{eq:irr-not-simple}), 
we conclude: 

\begin{proposition}
\label{prop:irr-not-simple}
Let $(\mathfrak{g}_0,\sigma ;\theta )$ be a non-compact semisimple symmetric pair 
equipped with a Cartan involution. 
If $(\mathfrak{g}_0,\sigma )$ is irreducible and $\mathfrak{g}_0$ is not simple, 
then there exists a real simple Lie algebra $\mathfrak{g}_0'$ such that 
\begin{align*}
(\mathfrak{g}_0,\sigma ;\theta )\equiv 
(\mathfrak{g}_0'\oplus \mathfrak{g}_0',\rho 
	;\theta |_{\mathfrak{g}_0'}\oplus \theta |_{\mathfrak{g}_0'}). 
\end{align*}
\end{proposition}

Next, let us consider the case where $\mathfrak{g}_0$ is simple. 
Clearly, $(\mathfrak{g}_0,\sigma )$ is irreducible for any involution $\sigma $. 
In the following, 
we focus on the setting that $\mathfrak{g}_0$ is a complex Lie algebra and the property of its involutions. 

\begin{remark}
\label{rem:c-linear}
Suppose that a non-compact real simple Lie algebra $\mathfrak{g}_0$ equips with a complex structure $J$. 
Then, any involution on a complex simple Lie algebra $(\mathfrak{g}_0,J)$ 
is either $\mathbb{C}$-linear or anti-linear. 
In fact, for an involution $\sigma$ on $\mathfrak{g}_0$, 
we set $J_{\sigma }:=\sigma J\sigma$. 
Then, $J_{\sigma}$ becomes a complex structure on $\mathfrak{g}_0$ because 
$J_{\sigma }^2 =\sigma J\sigma ^2J\sigma =\sigma J^2\sigma =-\sigma ^2=-\operatorname{id}$. 
Since $\mathfrak{g}_0$ is simple, this implies $J_{\sigma }$ has to be either $J$ or $-J$. 
If $J_{\sigma }=J$ then 
$J\sigma =\sigma J$, from which $\sigma$ is $\mathbb{C}$-linear. 
On the other hand, if $J_{\sigma }=-J$ then $J\sigma =-\sigma J$, from which $\sigma$ is anti-linear. 
If $\sigma$ coincides with a Cartan involution $\theta $, 
then it is anti-linear and its fixed point set in $\mathfrak{g}_0$ is a compact real form of 
the complex simple Lie algebra $(\mathfrak{g}_0,J)$. 
\end{remark}

As a consequence of Proposition \ref{prop:irr-not-simple} and Remark \ref{rem:c-linear}, 
we get all types of irreducible non-compact semisimple symmetric pairs. 
More precisely, 
irreducible $(\mathfrak{g}_0,\sigma )$ forms as follows 
when $\mathfrak{g}_0$ is not simple or has a complex structure: 

\begin{corollary}
\label{cor:irr-pair}
Let $(\mathfrak{g}_0,\sigma ;\theta )$ be a non-compact semisimple symmetric pair 
equipped with a Cartan involution 
except the case where $\mathfrak{g}_0$ is simple without complex structures. 
If $(\mathfrak{g}_0,\sigma )$ is irreducible, 
then $(\mathfrak{g}_0,\sigma ;\theta )$ satisfies one of the followings: 
\begin{enumerate}
	\renewcommand{\labelenumi}{\textup{(P-\theenumi)} }
	\renewcommand{\theenumi}{\alph{enumi}}
	\item 
	\label{item:cpx-anti}
	$\mathfrak{g}_0$ is real simple Lie algebra with a complex structure 
	and $\sigma$ is anti-linear on the complex Lie algebra $\mathfrak{g}_0$. 
	\item 
	\label{item:cpx-linear}
	$\mathfrak{g}_0$ is real simple Lie algebra with a complex structure 
	and $\sigma$ is $\mathbb{C}$-linear on the complex Lie algebra $\mathfrak{g}_0$. 
	\item 
	\label{item:non-anti}
	$(\mathfrak{g}_0,\sigma ;\theta )\equiv 
	(\mathfrak{g}_0'\oplus \mathfrak{g}_0',\rho ;
	\theta |_{\mathfrak{g}_0'}\oplus \theta |_{\mathfrak{g}_0'})$ 
	for some real simple Lie algebra $\mathfrak{g}_0'$ with a complex structure. 
%	and some Cartan involution $\theta '$. 
	\item 
	\label{item:non-linear}
	$(\mathfrak{g}_0,\sigma ;\theta )\equiv 
	(\mathfrak{g}_0'\oplus \mathfrak{g}_0',\rho ;
	\theta |_{\mathfrak{g}_0'}\oplus \theta |_{\mathfrak{g}_0'})$ 
	for some real simple Lie algebra $\mathfrak{g}_0'$ without complex structures. 
%	and for some Cartan involution $\theta '$ of $\mathfrak{g}_0'$. 
\end{enumerate}
\end{corollary}

\subsection{Irreducible compact semisimple symmetric triad}
\label{subsec:irr-triad}

In contrast to irreducible non-compact semisimple symmetric pairs, 
we introduce irreducibility of commutative compact semisimple symmetric triads. 

\begin{define}
\label{def:irr-triad}
A compact semisimple symmetric triad $(\mathfrak{g},\theta _1,\theta _2)$ is {\it irreducible} 
if it does not admit non-trivial $\theta _1$- and $\theta _2$-invariant 
ideals of $\mathfrak{g}$ (cf. \cite[Section 2]{matsuki}). 
\end{define}

By definition, 
$(\mathfrak{g},\theta _1,\theta _2)$ is irreducible if $\mathfrak{g}$ is simple. 
Further, we have: 

\begin{lemma}
\label{lem:irreducible}
For a commutative compact semisimple symmetric triad $(\mathfrak{g},\theta _1,\theta _2)$, 
following three conditions are equivalent: 
\begin{enumerate}
	\renewcommand{\theenumi}{\roman{enumi}}
	\item $(\mathfrak{g},\theta _1,\theta _2)$ is irreducible. 
	\label{item:pair}
	\item $(\mathfrak{g},\theta _1,\theta _2)^d$ is irreducible. 
	\label{item:dual}
	\item $(\mathfrak{g},\theta _1,\theta _2)^a$ is irreducible. 
	\label{item:associated}
\end{enumerate}
\end{lemma}

\begin{proof}
As $(\mathfrak{g},\theta _1,\theta _2)^d=(\mathfrak{g},\theta _2,\theta _1)$, 
the equivalence of (\ref{item:pair}) and (\ref{item:dual}) is clear. 
Further, an ideal of $\mathfrak{g}$ is $\theta _1$- and $(\theta _1\theta _2)$-invariant 
if and only if it is $\theta _1$- and $\theta _2$-invariant, 
from which the equivalence of (\ref{item:pair}) and (\ref{item:associated}) follows. 
\end{proof}

A compact semisimple symmetric triad $(\mathfrak{g},\theta _1,\theta _2)$ 
is decomposed into irreducible ones 
$(\mathfrak{g}_1,\theta _1|_{\mathfrak{g}_1},\theta _2|_{\mathfrak{g}_1}),
\ldots ,(\mathfrak{g}_k,\theta _1|_{\mathfrak{g}_k},\theta _2|_{\mathfrak{g}_k})$ by 
\begin{align*}
(\mathfrak{g},\theta _1,\theta _2)
=(\mathfrak{g}_1,\theta _1|_{\mathfrak{g}_1},\theta _2|_{\mathfrak{g}_1})
\oplus \cdots \oplus (\mathfrak{g}_k,\theta _1|_{\mathfrak{g}_k},\theta _2|_{\mathfrak{g}_k})
\end{align*}
where $\mathfrak{g}_1,\ldots ,\mathfrak{g}_k$ are non-trivial $\theta _1$- and $\theta _2$-invariant 
ideals of $\mathfrak{g}$ and 
$\mathfrak{g}$ has a Lie algebra decomposition 
$\mathfrak{g}=\mathfrak{g}_1\oplus \cdots \oplus \mathfrak{g}_k$. 
If $(\mathfrak{g},\theta _1,\theta _2)$ is commutative, then each irreducible component 
$(\mathfrak{g}_i,\theta _1|_{\mathfrak{g}_i},\theta _2|_{\mathfrak{g}_i})$ 
is commutative. 
By Lemma \ref{lem:duality-irr-decomp}, 
we have 
\begin{align*}
(\mathfrak{g},\theta _1,\theta _2)^*
=(\mathfrak{g}_1,\theta _1|_{\mathfrak{g}_1},\theta _2|_{\mathfrak{g}_1})^*
\oplus \cdots \oplus (\mathfrak{g}_k,\theta _1|_{\mathfrak{g}_k},\theta _2|_{\mathfrak{g}_k})^*. 
\end{align*}

Our concern here is to understand irreducible commutative compact semisimple symmetric triads 
in the setting where $\mathfrak{g}$ is semisimple but not simple. 

Let $\mathfrak{u}$ be a compact simple Lie algebra 
and $\nu $ an involution on $\mathfrak{u}$. 
Then, the direct sum $\mathfrak{u}\oplus \mathfrak{u}$ is a compact semisimple Lie algebra 
which is not simple. 

The following proposition is due to Matsuki \cite[Proposition 2.2]{matsuki}. 

\begin{proposition}[\cite{matsuki}]
\label{prop:irr-triad}
Let $(\mathfrak{g},\theta _1,\theta _2)$ 
be an irreducible commutative compact semisimple symmetric triad. 
Suppose that $\mathfrak{g}$ is not simple. 
Then, there exists a compact simple Lie algebra $\mathfrak{u}$ such that 
$(\mathfrak{g},\theta _1,\theta _2)$ satisfies one of the followings: 
\begin{enumerate}
	\renewcommand{\labelenumi}{\textup{(T-\theenumi)} }
	\renewcommand{\theenumi}{\alph{enumi}}
	\item \label{type:i2}
	$(\mathfrak{g},\theta _1,\theta _2)
	\equiv (\mathfrak{u}\oplus \mathfrak{u},\rho ,\rho \circ (\nu \oplus \nu ))$ 
	for some involution $\nu $ on $\mathfrak{u}$. 
	
	\item \label{type:ii2}
	$(\mathfrak{g},\theta _1,\theta _2)\equiv 
	(\mathfrak{u}\oplus \mathfrak{u},\rho ,\nu \oplus \nu )$ 
	for some involution $\nu $ on $\mathfrak{u}$. 
	
	\item \label{type:i4}
	$(\mathfrak{g},\theta _1,\theta _2)\equiv 
	(\mathfrak{u}\oplus \mathfrak{u}\oplus \mathfrak{u}\oplus \mathfrak{u},\rho _{(12)(34)},
	\rho _{(14)(23)})$. 
	Here, involutions $\rho _{(12)(34)}$ and $\rho _{(14)(23)}$ 
	on $\mathfrak{u}\oplus \mathfrak{u}\oplus 
	\mathfrak{u}\oplus \mathfrak{u}$ 
	are given by 
	\begin{align*}
	\rho _{(12)(34)}(X,Y,Z,W)&=(Y,X,W,Z),\\
	\rho _{(14)(23)}(X,Y,Z,W)&=(W,Z,Y,X)
	\end{align*}
	for $(X,Y,Z,W)\in \mathfrak{u}\oplus \mathfrak{u}\oplus \mathfrak{u}\oplus \mathfrak{u}$, 
	respectively. 
	
	\item \label{type:ii2d}
	$(\mathfrak{g},\theta _1,\theta _2)\equiv (\mathfrak{u}\oplus \mathfrak{u},\nu \oplus \nu ,\rho )$ 
	for some involution $\nu $ on $\mathfrak{u}$. 
	We note that this is the dual compact symmetric triad of 
	$(\mathfrak{u}\oplus \mathfrak{u},\rho ,\nu \oplus \nu )$ given in (T-\ref{type:ii2}). 
\end{enumerate}
\end{proposition}

Commutative compact semisimple symmetric triads in (T-\ref{type:i2})--(T-\ref{type:ii2d}) 
have the following properties: 

\begin{lemma}
\label{lem:non-simple}
Let $\mathfrak{u}$ be a compact simple Lie algebra and 
$\nu $ an involution on $\mathfrak{u}$. 
Then, we have: 
\begin{enumerate}
	\item $(\mathfrak{u}\oplus \mathfrak{u},\rho ,\rho \circ (\nu \oplus \nu ))$ is self-dual. 
	\item $(\mathfrak{u}\oplus \mathfrak{u},\rho ,\nu \oplus \nu )^a
	=(\mathfrak{u}\oplus \mathfrak{u},\rho ,\rho \circ (\nu \oplus \nu ))$. 
	\item $(\mathfrak{u}\oplus \mathfrak{u}\oplus \mathfrak{u}\oplus \mathfrak{u},
	\rho _{(12)(34)},\rho _{(14)(23)})$ is self-dual and self-associated. 
	\item $(\mathfrak{u}\oplus \mathfrak{u},\nu \oplus \nu ,\rho )$ is self-associated. 
\end{enumerate}
\end{lemma}

\begin{proof}
\begin{enumerate}
	\item We define a Lie algebra isomorphism $\varphi _{\nu }:
	\mathfrak{u}\oplus \mathfrak{u}\to \mathfrak{u}\oplus \mathfrak{u}$ by 
	\begin{align*}
	\varphi _{\nu }(X,Y)=(\nu (X),Y)\quad ((X,Y)\in \mathfrak{u}\oplus \mathfrak{u}). 
	\end{align*}
	Then, the following equalities hold for any $(X,Y)\in \mathfrak{u}\oplus \mathfrak{u}$: 
	\begin{align*}
	\varphi _{\nu }\circ \rho (\nu \oplus \nu ) (X,Y)&=\varphi _{\nu }(\nu (Y),\nu (X))
	=(Y,\nu (X)), \\
	\rho \circ \varphi _{\nu }(X,Y)&=\rho (\nu (X),Y)=(Y,\nu (X)). 
	\end{align*}
	This implies $(\mathfrak{u}\oplus \mathfrak{u},\rho ,\rho \circ (\nu \oplus \nu ))^d
	=(\mathfrak{u}\oplus \mathfrak{u},\rho \circ (\nu \oplus \nu ),\rho )
	\equiv (\mathfrak{u}\oplus \mathfrak{u},\rho ,\rho \circ (\nu \oplus \nu ))$. 
	
	\item This follows from the definition. 
	
	\item We define an automorphism $\rho _{(13)} $ on 
	$\mathfrak{u}\oplus \mathfrak{u}\oplus \mathfrak{u}\oplus \mathfrak{u}$ by 
	\begin{align*}
	\rho _{(13)}(X,Y,Z,W)=(Z,Y,X,W)
	\end{align*}
	for $(X,Y,Z,W)\in \mathfrak{u}\oplus \mathfrak{u}\oplus \mathfrak{u}\oplus \mathfrak{u}$. 
	Then, we have 
	\begin{align*}
	\rho _{(13)}\circ \rho _{(14)(23)}=\rho _{(12)(34)}\circ \rho _{(13)}. 
	\end{align*}
	This implies that 
	\begin{align*}
	&(\mathfrak{u}\oplus \mathfrak{u}\oplus \mathfrak{u}\oplus \mathfrak{u},
	\rho _{(12)(34)},\rho _{(14)(23)})^d\\
	&=(\mathfrak{u}\oplus \mathfrak{u}\oplus \mathfrak{u}\oplus \mathfrak{u},
	\rho _{(14)(23)},\rho _{(12)(34)})\\
	&\equiv (\mathfrak{u}\oplus \mathfrak{u}\oplus \mathfrak{u}\oplus \mathfrak{u},
	\rho _{(12)(34)},\rho _{(14)(23)}). 
	\end{align*}
	
	Next, we set $\rho _{(13)(24)}:=\rho _{(12)(34)}\circ \rho _{(14)(23)}$. 
	This is an involution given by 
	\begin{align*}
	\rho _{(13)(24)}(X,Y,Z,W)=(Z,W,X,Y)
	\end{align*}
	for $(X,Y,Z,W)\in \mathfrak{u}\oplus \mathfrak{u}\oplus \mathfrak{u}\oplus \mathfrak{u}$. 
	Then, 
	$(\mathfrak{u}\oplus \mathfrak{u}\oplus \mathfrak{u}\oplus \mathfrak{u},
	\rho _{(12)(34)},\rho _{(14)(23)})^a$ is written as 
	$(\mathfrak{u}\oplus \mathfrak{u}\oplus \mathfrak{u}\oplus \mathfrak{u},
	\rho _{(12)(34)} ,\rho _{(13)(24)})$. 
	Here, let us define an involution $\rho _{(34)}$ on 
	$\mathfrak{u}\oplus \mathfrak{u}\oplus \mathfrak{u}\oplus \mathfrak{u}$ 
	by 
	\begin{align*}
	\rho _{(34)}(X,Y,Z,W)=(X,Y,W,Z). 
	\end{align*}
	Then, we have 
	\begin{align*}
	\rho _{(34)}\circ \rho _{(12)(34)}&=\rho _{(12)(34)}\circ \rho _{(34)},\\
	\rho _{(34)}\circ \rho _{(13)(24)}&=\rho _{(14)(23)}\circ \rho _{(34)}. 
	\end{align*}
	Hence, we find out that 
	$(\mathfrak{u}\oplus \mathfrak{u}\oplus \mathfrak{u}\oplus \mathfrak{u},
	\rho _{(12)(34)},\rho _{(14)(23)})^a$ is equivalent to 
	$(\mathfrak{u}\oplus \mathfrak{u}\oplus \mathfrak{u}\oplus \mathfrak{u},
	\rho _{(12)(34)} ,\rho _{(14)(23)})$. 
	
	\item A direct computation shows that 
	\begin{align*}
	(\mathfrak{u}\oplus \mathfrak{u},\nu \oplus \nu ,\rho )^a
	&=(\mathfrak{u}\oplus \mathfrak{u},\rho ,\rho \circ (\nu \oplus \nu ))^{ada}\\
	&=(\mathfrak{u}\oplus \mathfrak{u},\rho ,\rho \circ (\nu \oplus \nu ))^{dad}. 
	\end{align*}
	As mentioned above, we obtain $(\mathfrak{u}\oplus \mathfrak{u},\rho ,\rho _{\nu })^{d}
	\equiv (\mathfrak{u}\oplus \mathfrak{u},\rho ,\rho \circ (\nu \oplus \nu ))$, 
	from which we get 
	\begin{align*}
	(\mathfrak{u}\oplus \mathfrak{u},\rho ,\rho \circ (\nu \oplus \nu ))^{dad}
	&\equiv (\mathfrak{u}\oplus \mathfrak{u},\rho ,\rho \circ (\nu \oplus \nu ))^{ad}\\
	&=(\mathfrak{u}\oplus \mathfrak{u},\rho ^2\circ (\nu \oplus \nu ),\rho)\\
	&=(\mathfrak{u}\oplus \mathfrak{u},\nu \oplus \nu ,\rho). 
	\end{align*}
	Hence, we have verified 
	$(\mathfrak{u}\oplus \mathfrak{u},\nu \oplus \nu ,\rho )^a 
	\equiv (\mathfrak{u}\oplus \mathfrak{u},\nu \oplus \nu ,\rho ) $. 
\end{enumerate}

Therefore, Lemma \ref{lem:non-simple} has been proved. 
\end{proof}

The following corollary is an immediate consequence of Lemma \ref{lem:non-simple}. 
However, we shall use it in the proof of Proposition \ref{prop:non-simple-complex}. 
Thus, we will state: 

\begin{corollary}
\label{cor:i4}
For a compact simple Lie algebra $\mathfrak{u}$, 
we have 
$(\mathfrak{u}\oplus \mathfrak{u}\oplus \mathfrak{u}\oplus \mathfrak{u},
\rho _{(12)(34)} ,\rho _{(13)(24)})\equiv 
(\mathfrak{u}\oplus \mathfrak{u}\oplus \mathfrak{u}\oplus \mathfrak{u},
\rho _{(12)(34)} ,\rho _{(14)(23)})$. 
\end{corollary}

\begin{proof}
The left-hand side is the associated compact semisimple symmetric triad of 
$(\mathfrak{u}\oplus \mathfrak{u}\oplus \mathfrak{u}\oplus \mathfrak{u},
\rho _{(12)(34)} ,\rho _{(14)(23)})$ 
which is self-associated by Lemma \ref{lem:non-simple}. 
Hence, we have obtained Corollary \ref{cor:i4}. 
\end{proof}

\begin{corollary}
\label{cor:ii2d}
For a compact simple Lie algebra $\mathfrak{u}$ and 
an involution $\nu$ on $\mathfrak{u}$. 
we have 
$(\mathfrak{u}\oplus \mathfrak{u},\nu \oplus \nu ,\rho )
\equiv (\mathfrak{u}\oplus \mathfrak{u},\nu \oplus \nu ,\rho \circ (\nu \oplus \nu ))$. 
\end{corollary}

\begin{proof}
This follows from 
$(\mathfrak{u}\oplus \mathfrak{u},\nu \oplus \nu ,\rho )^a
=(\mathfrak{u}\oplus \mathfrak{u},\nu \oplus \nu ,\rho \circ (\nu \oplus \nu ))$ and 
$(\mathfrak{u}\oplus \mathfrak{u},\nu \oplus \nu ,\rho )$ is self-associated 
by Lemma \ref{lem:non-simple}. 
\end{proof}

\subsection{Correspondence of invariant ideals via duality theorem}
\label{subsec:invariant ideal}

Let $(\mathfrak{g}_0,\sigma ;\theta )$ be a non-compact semisimple symmetric pair 
equipped with a Cartan involution 
and $(\mathfrak{g},\theta _1,\theta _2):=(\mathfrak{g}_0,\sigma ;\theta )^*$ 
the corresponding commutative compact semisimple symmetric triad. 
In this subsection, 
we consider the correspondence between the set of $\sigma$-invariant ideals of $\mathfrak{g}_0$ 
and that of $\theta _1$- and $\theta _2$-invariant ideals of $\mathfrak{g}$ 
via our duality in Theorem \ref{thm:duality-thm}. 

Thanks to Lemma \ref{lem:inv-cartan}, 
any $\sigma$-invariant ideal of $\mathfrak{g}_0$ is automatically $\theta $-invariant. 
Then, we shall consider the set $\mathcal{I}(\mathfrak{g}_0,\sigma ;\theta )$ 
of all $\sigma$- and $\theta $-invariant ideals of $\mathfrak{g}_0$. 
Similarly, we denote by $\mathcal{I}(\mathfrak{g},\theta _1,\theta _2)$ 
the set of all $\theta _1$- and $\theta _2$-invariant ideals of $\mathfrak{g}$. 

\begin{theorem}
\label{thm:ideal}
The duality given by Theorem \ref{thm:duality-thm} gives rise to a one-to-one correspondence 
between $\mathcal{I}(\mathfrak{g}_0,\sigma ;\theta )$ and $\mathcal{I}(\mathfrak{g},\theta _1,\theta _2)$. 
\end{theorem}

\begin{proof}
Let us take any $\mathfrak{l}_0\in \mathcal{I}(\mathfrak{g}_0,\sigma ;\theta )$. 
By Lemma \ref{lem:inv-cartan}, $(\mathfrak{l}_0,\sigma |_{\mathfrak{l}_0};\theta |_{\mathfrak{l}_0})$ 
is a non-compact semisimple symmetric pair equipped with a Cartan involution. 
We write $\mathfrak{l}_0=\mathfrak{l}_0^{\theta }+\mathfrak{l}_0^{-\theta }$ 
for the corresponding Cartan decomposition of $\theta |_{\mathfrak{l}_0}$. 
We set $\mathfrak{l}_0^*:=\mathfrak{l}_0^{\theta }+\sqrt{-1}\mathfrak{l}_0^{-\theta }$. 
As $\theta _1=\theta |_{\mathfrak{g}}$, this is a $\theta _1$-invariant compact semisimple Lie algebra. 
Further, the relation $\sigma \theta =\theta \sigma$ shows that 
both $\mathfrak{l}_0^{\theta }$ and $\mathfrak{l}_0^{-\theta }$ are $\sigma$-invariant. 
As $\theta _2=\sigma |_{\mathfrak{g}}$, we have 
\begin{align*}
\theta _2 (\mathfrak{l}_0^*)
=\sigma (\mathfrak{l}_0^{\theta })+\sqrt{-1}\sigma (\mathfrak{l}_0^{-\theta })
=\mathfrak{l}_0^{\theta }+\sqrt{-1}\mathfrak{l}_0^{-\theta }
=\mathfrak{l}_0^*. 
\end{align*}
Hence, $\mathfrak{l}_0^*$ is $\theta _2$-invariant. 
Therefore, $(\mathfrak{l}_0^*,\theta _1|_{\mathfrak{l}_0^*},\theta _2|_{\mathfrak{l}_0^*})$ 
is a commutative compact semisimple symmetric triad 
and it coincides with the dual $(\mathfrak{l}_0,\sigma |_{\mathfrak{l}_0};\theta |_{\mathfrak{l}_0})^*$, 
namely, 
\begin{align}
\label{eq:duality-ideal1}
(\mathfrak{l}_0^*,\theta _1|_{\mathfrak{l}_0^*},\theta _2|_{\mathfrak{l}_0^*})
=(\mathfrak{l}_0,\sigma |_{\mathfrak{l}_0};\theta |_{\mathfrak{l}_0})^*
=\Psi (\mathfrak{l}_0,\sigma |_{\mathfrak{l}_0};\theta |_{\mathfrak{l}_0}). 
\end{align}

Conversely, 
$(\mathfrak{l},\theta _1|_{\mathfrak{l}},\theta _2|_{\mathfrak{l}})$ is a commutative compact 
semisimple symmetric triad for $\mathfrak{l}\in \mathcal{I}(\mathfrak{g},\theta _1,\theta _2)$. 
We set $\mathfrak{l}^*:=\mathfrak{l}^{\theta _1}+\sqrt{-1}\mathfrak{l}^{-\theta _1}$. 
Then, it is $\theta $- and $\sigma$-invariant 
and we obtain 
\begin{align}
\label{eq:duality-ideal2}
(\mathfrak{l}^*,\sigma |_{\mathfrak{l}^*};\theta |_{\mathfrak{l}^*})
=(\mathfrak{l},\theta _1|_{\mathfrak{l}},\theta _2|_{\mathfrak{l}})^*
=\Phi (\mathfrak{l},\theta _1|_{\mathfrak{l}},\theta _2|_{\mathfrak{l}}). 
\end{align}

Next, we will verify that $\mathfrak{l}_0^*=\mathfrak{l}_0^{\theta }+\sqrt{-1}\mathfrak{l}_0^{-\theta }$ 
is an ideal of $\mathfrak{g}$ if $\mathfrak{l}_0\in \mathcal{I}(\mathfrak{g}_0,\sigma ;\theta )$ 
as follows. 
The inclusion $[\mathfrak{l}_0,\mathfrak{g}_0]\subset \mathfrak{l}_0$ implies 
\begin{align*}
([\mathfrak{l}_0^{\theta },\mathfrak{g}_0^{\theta }]+[\mathfrak{l}_0^{-\theta },\mathfrak{g}_0^{-\theta }])
+([\mathfrak{l}_0^{\theta },\mathfrak{g}_0^{-\theta }]+[\mathfrak{l}_0^{-\theta },\mathfrak{g}_0^{\theta }])
\subset \mathfrak{l}_0^{\theta }+\mathfrak{l}_0^{-\theta }. 
\end{align*}
Then, we obtain 
\begin{gather*}
[\mathfrak{l}_0^{\theta },\mathfrak{g}_0^{\theta }]+[\mathfrak{l}_0^{-\theta },\mathfrak{g}_0^{-\theta }]
\subset \mathfrak{l}_0^{\theta },\quad 
[\mathfrak{l}_0^{\theta },\mathfrak{g}_0^{-\theta }]+[\mathfrak{l}_0^{-\theta },\mathfrak{g}_0^{\theta }]
\subset \mathfrak{l}_0^{-\theta }. 
\end{gather*}
Under the setting, $[\mathfrak{l}_0^*,\mathfrak{g}]$ is written as 
\begin{align*}
[\mathfrak{l}_0^*,\mathfrak{g}]
&=
([\mathfrak{l}_0^{\theta },\mathfrak{g}_0^{\theta }]+[\mathfrak{l}_0^{-\theta },\mathfrak{g}_0^{-\theta }])
+\sqrt{-1}
([\mathfrak{l}_0^{\theta },\mathfrak{g}_0^{-\theta }]+[\mathfrak{l}_0^{\theta },\mathfrak{g}_0^{-\theta }]).
\end{align*}
Hence, $[\mathfrak{l}_0^*,\mathfrak{g}]$ is contained in 
$\mathfrak{l}_0^{\theta }+\sqrt{-1}\mathfrak{l}_0^{-\theta }=\mathfrak{l}_0^*$, from which 
$\mathfrak{l}_0^*$ is an ideal of $\mathfrak{g}$. 
By the same argument as above, 
$\mathfrak{l}^*=\mathfrak{l}^{\theta _1}+\sqrt{-1}\mathfrak{l}^{-\theta _1}$ is an ideal of $\mathfrak{g}_0$ 
if $\mathfrak{l}\in \mathcal{I}(\mathfrak{g},\theta _1,\theta _2)$. 

Therefore, 
the map 
$\Psi :\mathcal{I}(\mathfrak{g}_0,\sigma ;\theta )\to \mathcal{I}(\mathfrak{g},\theta _1,\theta _2)$, 
$\mathfrak{l}_0\mapsto \mathfrak{l}_0^*$ is defined by (\ref{eq:duality-ideal1}) 
and 
$\Phi :\mathcal{I}(\mathfrak{g},\theta _1,\theta _2)\to \mathcal{I}(\mathfrak{g}_0,\sigma ;\theta )$, 
$\mathfrak{l}\mapsto \mathfrak{l}^*$ by (\ref{eq:duality-ideal2}). 
Due to Theorem \ref{thm:duality-thm}, two maps are bijections and $\Phi ^{-1}=\Psi$. 

As a conclusion, Theorem \ref{thm:ideal} has been proved. 
\end{proof}

\subsection{Equivalence of irreducibility via duality theorem}
\label{subsec:equivalence}

In the following, 
we give a correspondence between irreducible non-compact semisimple symmetric pairs 
and irreducible commutative compact semisimple symmetric triads via our duality. 

First, our duality preserves the irreducibility, namely, we prove: 

\begin{theorem}
\label{thm:corresp-irreducible}
Let $(\mathfrak{g}_0,\sigma ;\theta )$ be a non-compact semisimple symmetric pair 
equipped with a Cartan involution 
and $(\mathfrak{g},\theta _1,\theta _2)=(\mathfrak{g}_0,\sigma ;\theta )^*$ 
the corresponding commutative compact semisimple symmetric triad via Theorem \ref{thm:duality-thm}. 
Then, $(\mathfrak{g}_0,\sigma )$ is irreducible (see Definition \ref{def:irr-pair}) 
if and only if $(\mathfrak{g},\theta _1,\theta _2)$ is irreducible (see Definition \ref{def:irr-triad}). 
\end{theorem}

\begin{proof}
Suppose $(\mathfrak{g}_0,\sigma )$ is irreducible. 
Then, the set $\mathcal{I}(\mathfrak{g}_0,\sigma ;\theta )$ contains only trivial ideals, 
namely, $\mathcal{I}(\mathfrak{g}_0,\sigma ;\theta )=\{ \{ 0\} ,\mathfrak{g}_0\} $. 
By Theorem \ref{thm:ideal}, 
the set $\mathcal{I}(\mathfrak{g},\theta _1,\theta _2)$ equals 
$\Psi (\mathcal{I}(\mathfrak{g}_0,\sigma ;\theta ))=\{ \{ 0\} ,\mathfrak{g}\} $. 
Hence, $(\mathfrak{g},\theta _1,\theta _2)$ is irreducible. 
The opposite is also true, from which we omit its proof. 
\end{proof}

In a case where a non-compact real Lie algebra $\mathfrak{g}_0$ is simple and 
has no complex structures, 
any $(\mathfrak{g}_0,\sigma )\in \mathfrak{P}$ is irreducible 
and then $(\mathfrak{g},\theta _1,\theta _2):=(\mathfrak{g}_0,\sigma )^*\in \mathfrak{T}$ 
is also irreducible. 
In particular, we know: 

\begin{proposition}
\label{prop:simple}
Let $(\mathfrak{g}_0,\sigma )$ be a non-compact semisimple symmetric pair 
and $(\mathfrak{g},\theta _1,\theta _2)$ the dual of $(\mathfrak{g}_0,\sigma )$ 
via Theorem \ref{thm:duality-thm}. 
Then, 
the compact Lie algebra $\mathfrak{g}$ is simple 
if and only if the non-compact real Lie algebra $\mathfrak{g}_0$ is simple and has no complex structures. 
\end{proposition}

\begin{proof}
This follows from the dual $(\mathfrak{g}_0,\theta )^*$ for Riemannian symmetric pair 
(see \cite[Theorem 5.4 in Chapter VIII]{helgason}). 
\end{proof}

In what follows, 
we study the irreducible commutative compact semisimple symmetric triad 
$(\mathfrak{g},\theta _1,\theta _2)=(\mathfrak{g}_0,\sigma ;\theta )^*$ 
corresponding to an irreducible non-compact semisimple symmetric pair $(\mathfrak{g}_0,\sigma ;\theta )$ 
which is one of (P-\ref{item:cpx-anti})--(P-\ref{item:non-linear}) of Corollary \ref{cor:irr-pair}. 

\begin{theorem}
\label{thm:irreducible}
Let $(\mathfrak{g}_0,\sigma ;\theta )$ be an irreducible non-compact semisimple symmetric pair 
equipped with a Cartan involution. 
Then, the dual $(\mathfrak{g}_0,\sigma ;\theta )^*$ satisfies 
(T-\ref{type:i2}) (resp. (T-\ref{type:ii2}), (T-\ref{type:i4}), (T-\ref{type:ii2d})) 
if and only if $(\mathfrak{g}_0,\sigma )$ satisfies 
(P-\ref{item:cpx-anti}) (resp. (P-\ref{item:cpx-linear}), (P-\ref{item:non-anti}), 
(P-\ref{item:non-linear})). 
\end{theorem}

Our proof of Theorem \ref{thm:irreducible} 
will be given in Sections \ref{subsec:complex} and \ref{subsec:non-simple}. 

Theorem \ref{thm:irreducible} provides an alternative proof of the classification of 
irreducible non-compact semisimple symmetric pairs $(\mathfrak{g}_0,\sigma )$ 
via Theorem \ref{thm:duality-thm} when $\mathfrak{g}_0$ is either not simple or has a complex structure. 
In fact, 
their classification can be reduced to that of compact semisimple symmetric pairs, 
which is well-known (see \cite{helgason} for example). 
More precisely, 
$(\mathfrak{u}\oplus \mathfrak{u},\rho ,\rho \circ (\nu \oplus \nu ))
\equiv (\mathfrak{u}'\oplus \mathfrak{u}',\rho ,\rho \circ (\nu '\oplus \nu '))$ 
(see (T-\ref{type:i2}) of Proposition \ref{prop:irr-triad}) 
if $(\mathfrak{u},\nu )\equiv (\mathfrak{u}',\nu ')$ and so on. 

\subsection{Complex simple case}
\label{subsec:complex}

In this subsection, 
we treat the case where 
a real Lie algebra is simple and has a complex structure. 

Let $\mathfrak{g}_0$ be a real simple Lie algebra equipped with a complex structure $J$. 
We write $\mathfrak{s}=(\mathfrak{g}_0,J)$ for the complex simple Lie algebra. 
The notation $\mathfrak{s}^{\mathbb{R}}$ stands for the real semisimple Lie algebra 
by restricting the coefficient field to $\mathbb{R}$, 
equivalently, $\mathfrak{g}_0=\mathfrak{s}^{\mathbb{R}}$. 
In this setting, 
we give a description of the dual $(\mathfrak{g},\theta _1,\theta _2)$ 
of a non-compact semisimple symmetric pair equipped with a Cartan involution 
$(\mathfrak{g}_0,\sigma ;\theta )$. 

\subsubsection{Corresponding compact Lie algebra}

First of all, 
we find a description of the corresponding compact semisimple Lie algebra $\mathfrak{g}$. 
This $\mathfrak{g}$ is determined by the Riemannian symmetric pair 
$(\mathfrak{g}_0,\theta )$ with (\ref{eq:compact real form}). 
Then, we will consider the dual $(\mathfrak{g},\theta _1)=(\mathfrak{g}_0,\theta )^*$ 
for Riemannian symmetric pairs (see Theorem \ref{thm:duality-g}). 

The fixed point set $\mathfrak{k}_0:=\mathfrak{g}_0^{\theta }=(\mathfrak{s}^{\mathbb{R}})^{\theta }$ 
of the Cartan involution $\theta $ in $\mathfrak{g}_0$ 
is a compact real form of $\mathfrak{s}$. 
Then, $\mathfrak{p}_0=\mathfrak{g}_0^{-\theta }=(\mathfrak{s}^{\mathbb{R}})^{-\theta }$ 
is written by $\mathfrak{p}_0=J\mathfrak{k}_0$, 
from which we can write $\mathfrak{g}_0=\mathfrak{k}_0+J\mathfrak{k}_0$ for the Cartan decomposition. 
We denote by $\overline{X}$ the complex conjugate of $X\in \mathfrak{s}$ 
with respect to $\mathfrak{k}_0$, namely, 
$\overline{X}=X_1-JX_2$ for $X=X_1+JX_2$ with $X_1,X_2\in \mathfrak{k}_0$. 

Let $\mathfrak{g}_{\mathbb{C}}=\mathfrak{g}_0+\sqrt{-1}\mathfrak{g}_0$ 
be the complexification of the real Lie algebra $\mathfrak{g}_0$. 
Now, we define a map 
$\eta :\mathfrak{g}_0+\sqrt{-1}\mathfrak{g}_0\to \mathfrak{s}\oplus \mathfrak{s}$ 
by 
\begin{align}
\eta (X+\sqrt{-1}Y):=(X+JY,\overline{X-JY})\quad 
(X,Y\in \mathfrak{g}_0). 
\label{eq:eta}
\end{align}
The following lemma is due to \cite[Theorem 6.94]{knapp}. 

\begin{lemma}
\label{lem:isom}
The complexification 
$\mathfrak{g}_{\mathbb{C}}=\mathfrak{g}_0+\sqrt{-1}\mathfrak{g}_0$ is 
isomorphic to $\mathfrak{s}\oplus \mathfrak{s}$ as a complex Lie algebra 
via $\eta $. 
\end{lemma}

Now, we will construct a compact real form $\mathfrak{g}$ of $\mathfrak{g}_{\mathbb{C}}$. 
Let $\tau $ be the complex conjugation on $\mathfrak{g}_{\mathbb{C}}$ 
with respect to $\mathfrak{g}_0$, namely, 
\begin{align*}
\tau (X+\sqrt{-1}Y)=X-\sqrt{-1}Y\quad (X,Y\in \mathfrak{g}_0). 
\end{align*}
We extend $\theta$ to a $\mathbb{C}$-linear involution on $\mathfrak{g}_{\mathbb{C}}$. 
Obviously, $\theta $ commutes with $\tau$, 
from which $\mu :=\tau \theta$ is an anti-linear involution on $\mathfrak{g}_{\mathbb{C}}$. 
Then, $\mathfrak{g}:=\mathfrak{g}_{\mathbb{C}}^{\mu }$ is a compact real form 
of $\mathfrak{g}_{\mathbb{C}}$. 

\begin{lemma}
\label{lem:k+k}
The compact Lie algebra $\mathfrak{g}$ is given by 
$\mathfrak{g}=\mathfrak{k}_0+\sqrt{-1}J\mathfrak{k}_0$ 
and is isomorphic to $\mathfrak{k}_0\oplus \mathfrak{k}_0$. 
\end{lemma}

\begin{proof}
A direct commutation shows that 
$\mathfrak{g}
%&=\mathfrak{g}_{\mathbb{C}}^{\tau ,\theta }+\mathfrak{g}_{\mathbb{C}}^{-\tau ,-\theta }
=\mathfrak{g}_0^{\theta }+\sqrt{-1}\mathfrak{g}_0^{-\theta }
=\mathfrak{k}_0+\sqrt{-1}J\mathfrak{k}_0$. 
For an element $X\in \mathfrak{s}$, 
the relation $\overline{X}=X$ holds if and only if $X$ lies in $\mathfrak{k}_0$. 
For $X,Y\in \mathfrak{k}_0$, we have 
\begin{align*}
\eta (X+\sqrt{-1}JY)
&=(X+J(JY),\overline{X-J(JY)})\\
&=(X-Y,\overline{X+Y}) \notag \\
&=(X-Y,X+Y). \notag
\end{align*}
Hence, $\eta (\mathfrak{g})$ is contained in $\mathfrak{k}_0\oplus \mathfrak{k}_0$. 
As $\dim \mathfrak{g}=\dim (\mathfrak{k}_0\oplus \mathfrak{k}_0)$, 
we have $\eta (\mathfrak{g})=\mathfrak{k}_0\oplus \mathfrak{k}_0$. 
Since $\eta$ is an isomorphism, 
we conclude $\mathfrak{g}\simeq \eta (\mathfrak{g})=\mathfrak{k}_0\oplus \mathfrak{k}_0$. 
\end{proof}

Second, 
the restriction of the $\mathbb{C}$-linear involution 
$\theta \in \operatorname{Aut}\mathfrak{g}_{\mathbb{C}}$ to $\mathfrak{g}$ 
becomes an involutive automorphism on $\mathfrak{g}$. 

\begin{lemma}
\label{lem:k}
The fixed point set $\mathfrak{g}^{\theta }$ coincides with $\mathfrak{k}_0$ 
and is isomorphic to 
$\operatorname{diag}(\mathfrak{k}_0)
=\{ (X,X):X\in \mathfrak{k}_0\} $ in $\mathfrak{k}_0\oplus \mathfrak{k}_0$. 
Further, $\mathfrak{g}^{-\theta }$ equals $\sqrt{-1}J\mathfrak{k}_0
\simeq \{ (X,-X):X\in \mathfrak{k}_0\}$. 
\end{lemma}

\begin{proof}
Recall that $\theta $ is anti-linear as a map on $\mathfrak{s}=(\mathfrak{g}_0,J)$, 
namely, $\theta \circ J=-J\circ \theta$ 
and $\mathbb{C}$-linear as a map on $\mathfrak{g}_{\mathbb{C}}$, 
namely, $\theta \circ \sqrt{-1}=\sqrt{-1}\circ \theta $. 
Let $X+\sqrt{-1}JY\in \mathfrak{g}=\mathfrak{k}_0+\sqrt{-1}J\mathfrak{k}_0$. 
By $\mathfrak{k}_0=\mathfrak{g}_0^{\theta }$ and Remark \ref{rem:c-linear}, we have 
\begin{align}
\theta (X+\sqrt{-1}JY)=\theta (X)-\sqrt{-1}J\theta (Y)=X-\sqrt{-1}JY. 
\label{eq:theta}
\end{align}
Then, $\theta (X+\sqrt{-1}JY)=X+\sqrt{-1}JY$ if and only if $Y=0$. 
Thus, we obtain $\mathfrak{g}^{\theta }=\mathfrak{k}_0$. 
It follows from the proof of Lemma \ref{lem:k+k} that 
$\mathfrak{k}_0\simeq \eta (\mathfrak{k}_0)
=\operatorname{diag}(\mathfrak{k}_0)$. 

Similarly, 
$\mathfrak{g}^{-\theta }=\sqrt{-1}J\mathfrak{k}_0\simeq \{ (X,-X):X\in \mathfrak{k}_0\}$ 
can be verified, which we omit its detail. 
\end{proof}

We remark that 
$\operatorname{diag}(\mathfrak{k}_0)$ is a symmetric subalgebra 
of $\mathfrak{k}_0\oplus \mathfrak{k}_0$ 
which is realized as the fixed point set of $\rho$ (see (\ref{eq:twist}) for definition). 
%Then, $\tau $ on $\mathfrak{k}_0\oplus \mathfrak{k}_0$ 
This corresponds to $\theta $ on $\mathfrak{g}$ in the sense of $\eta \theta =\tau \eta$. 
Hence, we get 

\begin{lemma}
\label{lem:complex-riemannian}
$(\mathfrak{g}_0,\theta )^*\equiv (\mathfrak{k}_0\oplus \mathfrak{k}_0,\rho )$. 
\end{lemma}

In the following, 
we consider an irreducible non-compact semisimple symmetric pair equipped with a Cartan involution 
$(\mathfrak{g}_0,\sigma ;\theta )$. 
As mentioned in Remark \ref{rem:c-linear}, 
an involution $\sigma$ on the complex simple Lie algebra $\mathfrak{s}$ 
is either $\mathbb{C}$-linear or anti-linear. 
Then, we deal with $\sigma$ as a $\mathbb{C}$-linear involution 
in Section \ref{subsubsec:complex linear} 
and as an anti-linear one in Section \ref{subsubsec:anti-linear}. 
For this, 
we extend $\theta ,\sigma \in \operatorname{Aut}\mathfrak{g}_0$ 
to $\mathbb{C}$-linear involutions on $\mathfrak{g}_{\mathbb{C}}$. 

\subsubsection{Complex linear case}
\label{subsubsec:complex linear}

Let us assume here that $\sigma$ is $\mathbb{C}$-linear on $\mathfrak{s}$, 
namely, $\sigma \circ J=J\circ \sigma $. 
As $\theta \sigma =\sigma \theta $, 
$\mathfrak{k}_0$ is $\sigma$-invariant, and then we have 
\begin{align*}
\sigma (X+\sqrt{-1}JY)=\sigma (X)+\sqrt{-1}J\sigma (Y)\quad (X,Y\in \mathfrak{k}_0). 
\end{align*}
Hence, $\mathfrak{g}^{\sigma }$ is expressed as 
$\mathfrak{g}^{\sigma }=\mathfrak{k}_0^{\sigma }+\sqrt{-1}J\mathfrak{k}_0^{\sigma}
\simeq \eta (\mathfrak{g}^{\sigma })
=\mathfrak{k}_0^{\sigma }\oplus \mathfrak{k}_0^{\sigma }$. 
Here, the subalgebra $\mathfrak{k}_0^{\sigma }\oplus \mathfrak{k}_0^{\sigma }$ 
is the fixed point set of the involution 
$\sigma \oplus \sigma $ on $\mathfrak{k}_0\oplus \mathfrak{k}_0$. 
Since $\mathfrak{s}$ is simple, 
the compact Lie algebra $\mathfrak{k}_0$ is also simple. 
Consequently, 
we have proved: 

\begin{proposition}[Proof of (P-\ref{item:cpx-linear}) $\Leftrightarrow $ (T-\ref{type:ii2}) 
of Theorem \ref{thm:irreducible}]
\label{prop:complex-linear}
Let $\mathfrak{g}_0$ be a real simple Lie algebra with a complex structure $J$. 
If an involution $\sigma$ is $\mathbb{C}$-linear on $\mathfrak{s}=(\mathfrak{g}_0,J)$, 
then we have 
\begin{align*}
(\mathfrak{g}_0,\sigma ;\theta )^*
\equiv (\mathfrak{k}_0\oplus \mathfrak{k}_0,\rho ,\sigma \oplus \sigma ). 
\end{align*}
\end{proposition}

\subsubsection{Anti-linear case}
\label{subsubsec:anti-linear}

On the other hand, we assume that $\sigma$ is anti-linear on $\mathfrak{s}$, 
namely, $\sigma \circ J=-J\circ \sigma$. 
Equivalently, 
$\mathfrak{s}^{\sigma }$ is a real form of $\mathfrak{s}$. 
Then, 
\begin{align*}
\sigma (X+\sqrt{-1}JY)=\sigma (X)-\sqrt{-1}J\sigma (Y)\quad (X,Y\in \mathfrak{k}_0). 
\end{align*}
This implies that 
$\mathfrak{g}^{\sigma }=\mathfrak{k}_0^{\sigma }+\sqrt{-1}J\mathfrak{k}_0^{-\sigma }$ 
and 
\begin{align*}
\eta (\mathfrak{g}^{\sigma })
&=\{ (X-Y,X+Y):	X\in \mathfrak{k}_0^{\sigma },Y\in \mathfrak{k}_0^{-\sigma }\} 
=\{ (X,\sigma (X)):X\in \mathfrak{k}_0\} . 
\end{align*}
This is a symmetric subalgebra of $\mathfrak{k}_0\oplus \mathfrak{k}_0$ 
which is realized as the fixed point set of $\rho \circ (\sigma \oplus \sigma )$. 
Therefore, we obtain: 

\begin{proposition}[Proof of (P-\ref{item:cpx-anti}) $\Leftrightarrow $ (T-\ref{type:i2}) 
of Theorem \ref{thm:irreducible}]
\label{prop:complex-antilinear}
Let $\mathfrak{g}_0$ be a real simple Lie algebra with a complex structure $J$. 
If an involution $\sigma$ is anti-linear on $\mathfrak{s}=(\mathfrak{g}_0,J)$, 
then we have 
\begin{align*}
(\mathfrak{g}_0,\sigma ;\theta )^*
&\equiv (\mathfrak{k}_0\oplus \mathfrak{k}_0,\rho ,\rho \circ (\sigma \oplus \sigma )). 
\end{align*}
\end{proposition}

\begin{remark}
$(\mathfrak{k}_0\oplus \mathfrak{k}_0,\rho ,\rho \circ (\sigma \oplus \sigma ))$ is not equivalent to 
$(\mathfrak{k}_0\oplus \mathfrak{k}_0,\rho ,\rho )$ 
unless $\sigma $ is a Cartan involution of $\mathfrak{g}_0$. 
\end{remark}

\subsection{Non-simple case}
\label{subsec:non-simple}

In this subsection, 
we deal with the case 
where a semisimple real Lie algebra is not simple. 

Let $(\mathfrak{g}_0,\sigma ;\theta )$ be an irreducible non-compact semisimple symmetric pair 
equipped with a Cartan involution. 
Suppose that $\mathfrak{g}_0$ is not simple. 
Recall from Proposition \ref{prop:irr-not-simple} that 
it is equivalent to $(\mathfrak{g}_0'\oplus \mathfrak{g}_0',\rho ;\theta |_{\mathfrak{g}_0'}\oplus 
\theta |_{\mathfrak{g}_0'})$ for some real simple Lie algebra $\mathfrak{g}_0'$. 
Thus, it suffices to consider $(\mathfrak{g}_0'\oplus \mathfrak{g}_0',\rho ;
\theta |_{\mathfrak{g}_0'}\oplus \theta |_{\mathfrak{g}_0'})$ in this case. 
For convenience, we shall write $\theta ':=\theta |_{\mathfrak{g}_0'}$. 

If $\mathfrak{g}_0'$ has no complex structures, we have: 

\begin{proposition}[Proof of (P-\ref{item:non-linear}) $\Leftrightarrow $ (T-\ref{type:ii2d}) 
of Theorem \ref{thm:irreducible}]
\label{prop:non-simple-real}
Let $\mathfrak{g}_0'$ be a non-compact real simple Lie algebra without complex structures. 
Then, the dual of $(\mathfrak{g}_0'\oplus \mathfrak{g}_0',\rho ;\theta '\oplus \theta ')$ 
is given by 
\begin{align*}
(\mathfrak{g}_0'\oplus \mathfrak{g}_0',\rho ;\theta '\oplus \theta ')^*
\equiv (\mathfrak{g}'\oplus \mathfrak{g}',\theta '\oplus \theta ',\rho ). 
\end{align*}
Here, the compact simple Lie algebra $\mathfrak{g}'$ is characterized by 
$(\mathfrak{g}',\theta ')=(\mathfrak{g}_0,\theta ')^*$. 
\end{proposition}

\begin{proof}
This is an immediate consequence of Theorem \ref{thm:duality-thm}. 
\end{proof}

Next, let us consider the case where $\mathfrak{g}_0'$ has a complex structure $J'$, 
namely, $\mathfrak{s}'=(\mathfrak{g}_0',J')$ is a complex simple Lie algebra. 
By Lemma \ref{lem:isom}, 
the complexification $(\mathfrak{g}_0')_{\mathbb{C}}$ of $\mathfrak{g}_0'$ is 
isomorphic to $\mathfrak{s}_0'\oplus \mathfrak{s}_0'$ as a complex Lie algebra 
via $\eta ':\mathfrak{g}_0'+\sqrt{-1}\mathfrak{g}_0'\to \mathfrak{s}_0'\oplus \mathfrak{s}_0'$ 
(see (\ref{eq:eta})). 
Since $\mathfrak{g}_{\mathbb{C}}=(\mathfrak{g}_0'\oplus \mathfrak{g}_0')_{\mathbb{C}} 
=(\mathfrak{g}_0')_{\mathbb{C}}\oplus (\mathfrak{g}_0')_{\mathbb{C}}$, 
we obtain 
\begin{align*}
\eta '\oplus \eta ':\mathfrak{g}_{\mathbb{C}}
=(\mathfrak{g}_0'\oplus \mathfrak{g}_0')_{\mathbb{C}} 
\simeq \mathfrak{s}_0'\oplus \mathfrak{s}_0'\oplus \mathfrak{s}_0'\oplus \mathfrak{s}_0'. 
\end{align*}
The maximal compact subalgebra $\mathfrak{k}_0':=(\mathfrak{g}_0)^{\theta '}$ of $\mathfrak{g}_0'$ 
is simple and a compact real form of $\mathfrak{s}'$. 
Then, it follows from Lemma \ref{lem:k+k} that 
\begin{align*}
\mathfrak{g}'\oplus \mathfrak{g}'\simeq 
\mathfrak{k}_0'\oplus \mathfrak{k}_0'\oplus \mathfrak{k}_0'\oplus \mathfrak{k}_0'. 
\end{align*}
Applying Lemma \ref{lem:k}, 
the fixed point set $(\mathfrak{g}'\oplus \mathfrak{g}')^{\theta '\oplus \theta '}
=\mathfrak{k}_0'\oplus \mathfrak{k}_0'$ is isomorphic to 
\begin{align*}
(\mathfrak{g}'\oplus \mathfrak{g}')^{\theta '\oplus \theta '}&\simeq 
\operatorname{diag}(\mathfrak{k}_0')\oplus \operatorname{diag}(\mathfrak{k}_0')
=(\mathfrak{k}_0'\oplus \mathfrak{k}_0'\oplus \mathfrak{k}_0'\oplus \mathfrak{k}_0')^{\rho _{(12)(34)}},
\end{align*}
and $(\mathfrak{g}'\oplus \mathfrak{g}')^{\rho }=\operatorname{diag}\mathfrak{g}'$ is 
\begin{align*}
(\mathfrak{g}'\oplus \mathfrak{g}')^{\rho }
&\simeq 
(\mathfrak{k}_0'\oplus \mathfrak{k}_0'\oplus \mathfrak{k}_0'\oplus \mathfrak{k}_0')^{\rho _{(13)(24)}}.
\end{align*}
This implies that 
\begin{align*}
(\mathfrak{g}'\oplus \mathfrak{g}',\theta '\oplus \theta ',\rho )
\equiv (\mathfrak{k}_0'\oplus \mathfrak{k}_0'\oplus \mathfrak{k}_0'\oplus \mathfrak{k}_0',
\rho _{(12)(34)},\rho _{(13)(24)}). 
\end{align*}
By Corollary \ref{cor:i4}, we conclude: 

\begin{proposition}[Proof of (P-\ref{item:non-anti}) $\Leftrightarrow $ (T-\ref{type:i4}) 
of Theorem \ref{thm:irreducible}]
\label{prop:non-simple-complex}
Let $\mathfrak{g}_0'$ be a real simple Lie algebra with a complex structure. 
Then, we have 
\begin{align*}
(\mathfrak{g}_0'\oplus \mathfrak{g}_0',\rho ;\theta '\oplus \theta ')^*\equiv 
(\mathfrak{k}_0'\oplus \mathfrak{k}_0'\oplus \mathfrak{k}_0'\oplus \mathfrak{k}_0',
\rho _{(12)(34)},\rho _{(14)(23)}). 
\end{align*}
with $\mathfrak{k}_0'=(\mathfrak{g}_0')^{\theta '}$. 
\end{proposition}

By Propositions \ref{prop:complex-linear}, \ref{prop:complex-antilinear}, \ref{prop:non-simple-real} and 
\ref{prop:non-simple-real}, 
Theorem \ref{thm:irreducible} has been completely proved.

\subsection{Semisimple symmetric pair of type $K_{\varepsilon }$}
\label{subsec:ke}

In this subsection, 
we deal with a certain class in non-compact semisimple symmetric pairs, 
namely, symmetric pairs of type $K_{\varepsilon}$. 

\subsubsection{Symmetric pair of type $K_{\varepsilon }$} 
\label{subsubsec:typeK}

The original definition of symmetric pairs of type $K_{\varepsilon}$ is given by 
Oshima--Sekiguchi in \cite{oshima-sekiguchi-invent}, 
and after that a necessary and sufficient condition on a non-compact semisimple symmetric pair 
to be of type $K_{\varepsilon }$ is provided by Kaneyuki in \cite{kaneyuki96}. 
This paper would adopt Kaneyuki's criterion as a definition. 

Let $\mathfrak{g}_0$ be a non-compact real semisimple Lie algebra. 
Suppose we are given a $\mathbb{Z}$-grading of $m$-th kind, 
namely, $\mathfrak{g}_0$ is decomposed into the sum of $2m+1$ subspaces for some positive integer $m$ as 
\begin{align}
\mathfrak{g}_0=\sum _{k=-m}^m \mathfrak{g}_0(k) 
\label{eq:grading}
\end{align}
under the relations 
$[\mathfrak{g}_0(k),\mathfrak{g}_0(l)]\subset \mathfrak{g}_0(k+l)$ for $-m\leq k,l\leq m$, 
$\mathfrak{g}_0(\pm m)\neq \{ 0\} $ 
and $\mathfrak{g}_0(k)=\{ 0\} $ for $|k|>m$. 
We note that 
$\mathfrak{g}_0(0)$ is a reductive Lie algebra. 

The next lemma is well-known, 
however its proof might not be written in any paper. 
Then, we will explain the proof below. 

\begin{lemma}
\label{lem:ch}
Retain the setting as above. 
Then, there exists $Z\in \mathfrak{g}_0(0)$ uniquely such that 
$\operatorname{ad}(Z)|_{\mathfrak{g}_0(k)}=k\operatorname{id}_{\mathfrak{g}_0(k)}$ 
for $-m\leq k\leq m$. 
\end{lemma}

\begin{proof}
Let $\varphi$ be a linear transformation on $\mathfrak{g}_0$ satisfying 
$\varphi |_{\mathfrak{g}_0(k)}=k\operatorname{id}_{\mathfrak{g}_0(k)}$ for any $-m\leq k\leq m$. 
We take arbitrary elements $X_k\in \mathfrak{g}_0(k)$ and $X_l\in \mathfrak{g}_0(l)$. 
In view of $[X_k,X_l]\in \mathfrak{g}_0(k+l)$, 
we have $\varphi ([X_k,X_l])=(k+l)[X_k,X_l]$. 
On the other hand, the direct computation shows 
\begin{align*}
[\varphi (X_k),X_l]+[X_k,\varphi (X_l)]
=[kX_k,X_l]+[X_k,lX_l]
=(k+l)[X_k,X_l]. 
\end{align*}
Thus, we obtain $\varphi ([X_k,X_l])=[\varphi (X_k),X_l]+[X_k,\varphi (X_l)]$ 
for any $k,l$. 
Hence, $\varphi $ is a derivation on $\mathfrak{g}_0$. 
As $\mathfrak{g}_0$ is semisimple, any derivation is an inner automorphism on $\mathfrak{g}_0$ 
(cf. \cite[Proposition 6.4 in Chapter II]{helgason}). 
Hence, there exists $Z\in \mathfrak{g}_0$ uniquely such that $\varphi =\operatorname{ad}Z$. 

Next, let us show $Z$ that lies in $\mathfrak{g}_0(0)$. 
For this, we write $Z=\sum _{k=-m}^m Z_k$ along the $\mathbb{Z}$-grading (\ref{eq:grading}) 
($Z_k\in \mathfrak{g}_0(k)$). 
Since $[Z_k,\mathfrak{g}_0(l)]$ is contained in $\mathfrak{g}_0(k+l)$, 
the following inclusion holds for $-m\leq l\leq m$: 
\begin{align*}
\varphi (\mathfrak{g}_0(l))
=(\operatorname{ad}Z)(\mathfrak{g}_0(l))
=\sum _{k=-m}^m[Z_k,\mathfrak{g}_0(l)]
\subset \sum _{k=-m}^m\mathfrak{g}_0(k+l). 
\end{align*}
On the other hand, 
the definition $\varphi |_{\mathfrak{g}_0(l)}=l\operatorname{id}_{\mathfrak{g}_0(l)}$ implies 
$\varphi (\mathfrak{g}_0(l))=\mathfrak{g}_0(l)$. 
Then, $[Z_k,\mathfrak{g}_0(l)]$ must be $\{ 0\} $ for any $l$ and $k\neq 0$. 
Thus, $[Z_k,\mathfrak{g}_0]=\{ 0\} $, for $k\neq 0$. 
As $\mathfrak{g}_0$ is semisimple, we get $Z_k=0$ for $k\neq 0$. 
Consequently, $Z=Z_0\in \mathfrak{g}_0(0)$. 
\end{proof}

\begin{define}
\label{def:ch}
We say that the element $Z\in \mathfrak{g}_0$ satisfying Lemma \ref{lem:ch} 
is the {\it characteristic element} of the $\mathbb{Z}$-grading (\ref{eq:grading}). 
\end{define}

By Lemma \ref{lem:ch}, 
the subspace $\mathfrak{g}_0(k)$ is characterized by $Z$, namely, it is of the form: 
\begin{align}
\label{eq:k-eigenspace}
\mathfrak{g}_0(k)=\{ X\in \mathfrak{g}_0:(\operatorname{ad}Z)X=kX\}\quad (-m\leq k\leq m). 
\end{align}
Then, it follows from \cite[Theorem I.2.3 in Part II]{fkklr} that 
there exists a Cartan involution $\theta $ of $\mathfrak{g}_0$ such that 
\begin{gather}
\label{eq:reversing}
\theta (Z)=-Z, 
\end{gather}
and also follows from \cite[Lemma 1.4]{kaneyuki96} that 
such $\theta $ are unique up to conjugation by inner automorphisms on $\mathfrak{g}_0(0)$. 
This implies 
$\theta (\mathfrak{g}_0(k))=\mathfrak{g}_0(-k)$ for any $k$. 
In this sense, $\theta$ is called a {\it grade-reversing} Cartan involution 
associated with (\ref{eq:grading}). 

\begin{define}
\label{def:associated pair}
We call the pair $(Z,\theta )$ with (\ref{eq:k-eigenspace}) and (\ref{eq:reversing}) 
the {\it associated pair} of the $\mathbb{Z}$-grading (\ref{eq:grading}). 
\end{define}

In this setting, we define $\sigma _Z$ by 
\begin{align}
\sigma _Z:=e^{\pi \sqrt{-1}\operatorname{ad}Z} .
\label{eq:sigma_Z}
\end{align}
Then, we have 
$\sigma _Z(X_k)=e^{\pi \sqrt{-1}k}X_k=(-1)^kX_k$ for any $X_k\in \mathfrak{g}_0(k)$. 
This shows that $\sigma _Z$ defines an involution on $\mathfrak{g}_0$. 

\begin{lemma}
\label{lem:sigma_Z}
Let $(Z,\theta )$ be the associated pair of a $\mathbb{Z}$-grading of $\mathfrak{g}_0$. 
Then, the involution $\sigma _Z$ is commutes with $\theta$. 
Hence, $\sigma _Z\theta$ is an involution on $\mathfrak{g}_0$ commuting with $\theta$. 
\end{lemma}

\begin{proof}
It suffices to show Lemma \ref{lem:sigma_Z} on each eigenspace $\mathfrak{g}_0(k)$ $(k\in \mathbb{Z})$. 

Let us give a $\mathbb{Z}$-grading of $\mathfrak{g}_0$ by (\ref{eq:grading}) 
which is characterized by $Z$. 
As $\theta (\mathfrak{g}_0(k))=\mathfrak{g}_0(-k)$, 
we have $\sigma _Z\theta (X_k)=(-1)^{-k}\theta (X_k)
=(-1)^k\theta (X_k)=\theta \sigma _Z(X_k)$, from which $\sigma _Z\theta (X_k)=\theta \sigma _Z(X_k)$ 
for any element $X_k\in \mathfrak{g}_0(k)$. 
Hence, $\sigma _Z\theta =\theta \sigma _Z$ on $\mathfrak{g}_0$. 

The second statement is obvious from the first one. 
\end{proof}

Lemma \ref{lem:sigma_Z} explains that 
$(\mathfrak{g}_0,\sigma _Z\theta ;\theta )$ is a non-compact semisimple symmetric pair 
equipped with a Cartan involution. 

\begin{define}[{\cite[Proposition 2.1]{kaneyuki96}}]
\label{def:typeK}
We say that 
a non-compact semisimple symmetric pair $(\mathfrak{g}_0,\sigma )$ 
is of {\it type $K_{\varepsilon }$} if (and only if) 
there exists a $\mathbb{Z}$-grading of $\mathfrak{g}_0$ and 
its associated pair $(Z,\theta )$ such that 
$\sigma =\sigma _Z\theta $. 
\end{define}

Our equivalence relation $\equiv $ on $\mathfrak{P}$ preserves Definition \ref{def:typeK}. 
More precisely, we show: 

\begin{lemma}
\label{lem:equiv-typeK}
Let $(\mathfrak{g}_0,\sigma ),(\mathfrak{g}_0',\sigma ')$ be non-compact semisimple symmetric pairs. 
If $(\mathfrak{g}_0,\sigma )$ is of type $K_{\varepsilon }$ and 
$(\mathfrak{g}_0,\sigma )\equiv (\mathfrak{g}_0',\sigma ')$, 
then $(\mathfrak{g}_0',\sigma ')$ is of type $K_{\varepsilon }$. 
\end{lemma}

\begin{proof}
Let us give a $\mathbb{Z}$-grading (\ref{eq:grading}) of $\mathfrak{g}_0$ 
and its associated pair $(Z,\theta )$ satisfying $\sigma =\sigma _Z\theta $. 
By Lemma \ref{lem:cartan-involution}, 
$(\mathfrak{g}_0,\sigma ;\theta )\in \mathfrak{P}_c$ is equivalent to $(\mathfrak{g}_0',\sigma ';\theta ')$ 
for any Cartan involution $\theta '$ of $\mathfrak{g}_0$ commuting with $\sigma '$. 
We take a Lie algebra isomorphism $\varphi :\mathfrak{g}\to \mathfrak{g}'$ 
satisfying $\varphi \theta =\theta '\varphi $ and $\varphi \sigma =\sigma '\varphi $. 

We set $\mathfrak{g}_0'(k):=\varphi (\mathfrak{g}_0(k))$ for each $k$. 
Then, $\mathfrak{g}_0'$ is decomposed as 
\begin{align}
\mathfrak{g}_0'
&=\varphi (\mathfrak{g}_0)
=\sum _{k=-m}^m \mathfrak{g}_0'(k), 
\label{eq:grading'}
\end{align}
%The inclusion $[\mathfrak{g}_0(k),\mathfrak{g}_0(l)]\subset \mathfrak{g}_0(k+l)$ 
%induces $[\varphi (\mathfrak{g}_0(k)),\varphi (\mathfrak{g}_0(l))]
%=\varphi ([\mathfrak{g}_0(k),\mathfrak{g}_0(l)])\subset \varphi (\mathfrak{g}_0(k+l))$ 
%for $k,l\in \mathbb{Z}$. 
which defines a $\mathbb{Z}$-grading of $\mathfrak{g}_0'$. 
Moreover, $Z':=\varphi (Z)$ is the characteristic element of (\ref{eq:grading'}) 
because $X_k\in \mathfrak{g}_0(k)$ satisfies 
$(\operatorname{ad}\varphi (Z))\varphi (X_k)=\varphi ((\operatorname{ad}Z)X_k)=\varphi (kX_k)
=k\varphi (X_k)$. 
Further, the condition $\varphi \theta =\theta '\varphi $ implies 
$\theta '(Z')=\varphi (\theta (Z))=\varphi (-Z)=-Z'$. 
Thus, $\theta '$ is a grade-reversing Cartan involution of $\mathfrak{g}_0$ 
and then $(Z',\theta ')$ is the associated pair of (\ref{eq:grading'}). 

Finally, the commutativity $\sigma '\varphi =\varphi \sigma$ and 
$\sigma =e^{\pi \sqrt{-1}\operatorname{ad}Z}\theta $ show 
\begin{align*}
\sigma '
=\varphi (e^{\pi \sqrt{-1}\operatorname{ad}Z}\theta )\varphi ^{-1}
=e^{\pi \sqrt{-1}\operatorname{ad}\varphi (Z)}\varphi \theta \varphi ^{-1}
=\sigma _{Z'}\theta '. 
\end{align*}
Therefore, the non-compact semisimple symmetric pair 
$(\mathfrak{g}_0',\sigma ')$ is of type $K_{\varepsilon }$. 
\end{proof}

The aim of this subsection is to clarify a certain class in commutative compact semisimple symmetric triads 
which corresponds to the class of non-compact semisimple symmetric pairs of type $K_{\varepsilon }$ 
via Theorem \ref{thm:duality-thm}. 

\subsubsection{Dual of symmetric pair of type $K_{\varepsilon }$}
\label{subsubsec:typeK-sim}

First, we give a characterization of the dual 
of a non-compact semisimple symmetric pair of type $K_{\varepsilon }$ 
in the sense of Theorem \ref{thm:duality-thm}, 
on which we will explain in Proposition \ref{prop:typeK-sim} after the next lemma. 

\begin{lemma}
\label{lem:e}
Let $\mathfrak{g}_0$ be a real Lie algebra and $\nu$ an involution on $\mathfrak{g}_0$. 
If an element $Y\in \mathfrak{g}_0$ satisfies $\nu (Y)=-Y$, 
then we have $e^{\operatorname{ad}Y}\nu =\nu e^{-\operatorname{ad}Y}$. 
\end{lemma}

\begin{proof}
For any $X\in \mathfrak{g}_0$, we have 
\begin{align*}
e^{\operatorname{ad}Y}\nu e^{\operatorname{ad}Y}(X)
=e^{\operatorname{ad}Y}e^{\operatorname{ad}\nu (Y)}\nu (X)
=e^{\operatorname{ad}Y}e^{-\operatorname{ad}Y}\nu (X)
=\nu (X). 
\end{align*}
\end{proof}

\begin{proposition}
\label{prop:typeK-sim}
Let $(\mathfrak{g}_0,\sigma )$ be a non-compact semisimple symmetric pair of type $K_{\varepsilon }$ 
and $(\mathfrak{g},\theta _1,\theta _2)=(\mathfrak{g}_0,\sigma )^*$ 
the corresponding commutative compact semisimple symmetric triad via Theorem \ref{thm:duality-thm}. 
Then, there exists an element $Y\in \mathfrak{g}$ such that 
\begin{align}
\label{eq:sim}
\theta _2=e^{\operatorname{ad}Y} \theta _1e^{-\operatorname{ad}Y} . 
\end{align}
\end{proposition}

\begin{proof}
We take a $\mathbb{Z}$-grading of $\mathfrak{g}_0$ and its associated pair $(Z,\theta )$ 
satisfying $\sigma =\sigma _Z\theta $. 
We set 
\begin{align*}
Y:=\frac{\pi \sqrt{-1}}{2}Z. 
\end{align*}
Then, it follows from (\ref{eq:sigma_Z}) that $\sigma _Z$ coincides with $e^{\operatorname{ad}(2Y)}$. 
Here, the relation (\ref{eq:reversing}) means 
the element $Z$ is contained in $\mathfrak{g}_0^{-\theta }$, 
in particular, $Y\in \sqrt{-1}\mathfrak{g}_0^{-\theta }
\subset \mathfrak{g}_0^{\theta }+\sqrt{-1}\mathfrak{g}_0^{-\theta }=\mathfrak{g}$. 

Two involutions $\theta _1$ and $\theta _2$ form 
$\theta _1=\theta $ and $\theta _2=\sigma =\sigma _Z\theta =e^{\operatorname{ad}(2Y)}\theta $, 
respectively. 
By Lemma \ref{lem:e}, 
$e^{\operatorname{ad}(2Y)}\theta $ equals $e^{\operatorname{ad}Y}\theta e^{-\operatorname{ad}Y}$. 
Hence, we get $\theta _2=e^{\operatorname{ad}Y}\theta _1e^{-\operatorname{ad}Y}$. 
\end{proof}

Our next concern is a commutative compact symmetric triad $(\mathfrak{g},\theta _1,\theta _2)$ 
with the property (\ref{eq:sim}). 
For our argument below, 
we introduce a relation between two involutions by (\ref{eq:sim}). 
Namely, 

\begin{define}
\label{def:sim}
Two involutions $\theta _1,\theta _2$ on $\mathfrak{g}$ satisfy $\theta _1\sim \theta _2$ 
if the condition (\ref{eq:sim}) holds for some $Y\in \mathfrak{g}$. 
\end{define}

This relation defines an equivalence relation on the set of involutions on $\mathfrak{g}$. 

\subsubsection{Compact symmetric triad with (\ref{eq:sim})}%$\theta _1\sim \theta _2$}
\label{subsubsec:property-sim}

In this subsection, 
we study commutative compact semisimple symmetric triads with property (\ref{eq:sim}). 

Let $\mathfrak{g}$ be a compact semisimple Lie algebra and $\theta _1$ an involution on $\mathfrak{g}$. 
We write $\mathfrak{g}=\mathfrak{k}_1+\mathfrak{p}_1$ for the eigenspace decomposition 
with $(+1)$-eigenspace $\mathfrak{k}_1=\mathfrak{g}^{\theta _1}$ 
and $(-1)$-eigenspace $\mathfrak{p}_1=\mathfrak{g}^{-\theta _1}$. 

\begin{lemma}
\label{lem:yz}
\begin{enumerate}
	\item $e^{\operatorname{ad}Y_1}\theta _1e^{-\operatorname{ad}Y_1}=\theta _1$ 
	for any $Y_1\in \mathfrak{k}_1$. 
	\item $e^{\operatorname{ad}Z_1}\theta _1e^{\operatorname{ad}Z_1}=\theta _1$ 
	for any $Z_1\in \mathfrak{p}_1$. 
\end{enumerate}
\end{lemma}

\begin{proof}
Lemma \ref{lem:yz} is an immediate consequence of Lemma \ref{lem:e}. 
\end{proof}

Let us denote by $\operatorname{Int}\mathfrak{g}$ the adjoint group of $\mathfrak{g}$, 
namely, the analytic subgroup of the general linear group $GL(\mathfrak{g})$ 
with Lie algebra $\operatorname{ad}\mathfrak{g}$. 
As $\mathfrak{g}$ is semisimple, we identify $\operatorname{ad}\mathfrak{g}$ with $\mathfrak{g}$. 
Since $\operatorname{Int}\mathfrak{g}$ is compact, we can write 
\begin{align}
\label{eq:intg}
\operatorname{Int}\mathfrak{g}
=\exp (\operatorname{ad}\mathfrak{g})
=\{ e^{\operatorname{ad}Y}:Y\in \mathfrak{g}\} . 
\end{align}
%Further, $\operatorname{ad}\mathfrak{g}\simeq \mathfrak{g}$ because $\mathfrak{g}$ is semisimple. 
%Hence, we shall identify $\operatorname{ad}\mathfrak{g}$ with $\mathfrak{g}$. 

%Next, let $G=\operatorname{Int}\mathfrak{g}$ be the adjoint group of $\mathfrak{g}$. 
Let $G$ be $\operatorname{Int}\mathfrak{g}$. 
For an arbitrary involution $\nu $ on $\mathfrak{g}$, 
we can lift $\nu$ to an involution on $G$ via (\ref{eq:intg}), 
which we use the same letter $\nu$ to denote. 
We set $K_1$ as the identity component of the fixed point set $G^{\theta _1}$. 
Then, $K_1$ is a connected closed subgroup of $G$ with Lie algebra $\mathfrak{k}_1$. 
Let $\mathfrak{a}_1$ be a maximal abelian subspace in $\mathfrak{p}_1$. 
It is well-known fact that 
we have a compact Lie group decomposition (see \cite[Theorem 6.7 in Chapter V]{helgason}, for example) 
\begin{align}
G=K_1(\exp \mathfrak{a}_1)K_1. 
\label{eq:group-decomp}
\end{align}

We will fix a $G$-invariant inner product $\langle \cdot ,\cdot \rangle$ on $\mathfrak{g}$. 
We extend $\operatorname{ad}A\in \operatorname{ad}\mathfrak{g}$ 
for any $A\in \mathfrak{g}$ to a $\mathbb{C}$-linear transformation 
on the complexified $\mathfrak{g}_{\mathbb{C}}$. 
This is diagonalizable and its eigenvalues are pure imaginary numbers. 
For $\lambda \in \mathfrak{a}_1$, we write 
\begin{align*}
\mathfrak{g}_{\mathbb{C}}(\mathfrak{a}_1:\lambda )
:=\{ X\in \mathfrak{g}_{\mathbb{C}}:(\operatorname{ad}A)X=\sqrt{-1}\langle \lambda ,A\rangle X~
(\forall A\in \mathfrak{a}_1)\} 
\end{align*}
for the restricted root space with restricted root $\lambda $. 
We set 
\begin{align*}
\Sigma \equiv \Sigma (\mathfrak{g}_{\mathbb{C}},\mathfrak{a}_1)
:=\{ \lambda \in \mathfrak{a}_1-\{ 0\} 
:\mathfrak{g}_{\mathbb{C}}(\mathfrak{a}_1:\lambda )\neq \{ 0\} \} . 
\end{align*}
Then, $\Sigma$ satisfies the axiom of root systems. 
In particular, if $\lambda \in \Sigma$ then $-\lambda \in \Sigma$. 
Thus, 
$\mathfrak{g}_{\mathbb{C}}$ is decomposed into the restricted root spaces as follows: 
\begin{align}
\mathfrak{g}_{\mathbb{C}}=\mathfrak{z}(\mathfrak{a}_1)+
	\sum _{\lambda \in \Sigma }\mathfrak{g}_{\mathbb{C}}(\mathfrak{a}_1:\lambda ). 
\label{eq:root-space-decomp}
\end{align}
Here, we write $\mathfrak{z}(\mathfrak{a}_1)=\mathfrak{g}_{\mathbb{C}}(\mathfrak{a}_1:0)$ 
for the centralizer of $\mathfrak{a}_1$ in $\mathfrak{g}_{\mathbb{C}}$. 

Now, 
let $(\mathfrak{g},\theta _1,\theta _2)$ be a commutative compact semisimple symmetric triad 
with $\theta _1\sim \theta _2$. 
We write $\theta _2=e^{\operatorname{ad}Y} \theta _1e^{-\operatorname{ad}Y}$ 
for some $Y \in \mathfrak{g}$. 
Due to the decomposition (\ref{eq:group-decomp}), 
we can write $e^{\operatorname{ad}Y}\in G$ as 
\begin{align}
\label{eq:e^Y}
e^{\operatorname{ad}Y}=e^{\operatorname{ad}Y_1}e^{\operatorname{ad}Z_1}e^{\operatorname{ad}Y_2}\quad 
(Y_1,Y_2\in \mathfrak{k}_1,Z_1\in \mathfrak{a}_1). 
\end{align}

By Lemma \ref{lem:yz}, the involution $\theta _2$ forms 
\begin{align}
\label{eq:theta _2}
\theta _2&=e^{\operatorname{ad}Y_1}e^{\operatorname{ad}Z_1}(e^{\operatorname{ad}Y_2}\theta _1
	e^{-\operatorname{ad}Y_2})e^{-\operatorname{ad}Z_1}e^{-\operatorname{ad}Y_1} \\
&=e^{\operatorname{ad}Y_1}e^{\operatorname{ad}Z_1}\theta _1
	e^{-\operatorname{ad}Z_1}e^{-\operatorname{ad}Y_1}\notag \\
&=\theta _1e^{\operatorname{ad}Y_1}e^{-2\operatorname{ad}Z_1}e^{-\operatorname{ad}Y_1} % \notag \\
=e^{\operatorname{ad}Y_1}e^{2\operatorname{ad}Z_1}e^{-\operatorname{ad}Y_1}\theta _1. 
\notag 
\end{align}
Then, we have 
$\theta _1\theta _2=e^{\operatorname{ad}Y_1}e^{-\operatorname{ad}(2Z_1)}e^{-\operatorname{ad}Y_1}$ 
and 
$\theta _2\theta _1
=e^{\operatorname{ad}Y_1}e^{\operatorname{ad}(2Z_1)}e^{-\operatorname{ad}Y_1}$. 
Hence, the commutativity $\theta _1\theta _2=\theta _2\theta _1$ implies 

\begin{lemma}
\label{lem:Z_1}
$e^{\operatorname{ad}(2Z_1)}=e^{-\operatorname{ad}(2Z_1)}$. 
In particular, $e^{\operatorname{ad}(4Z_1)}=\operatorname{id}_{\mathfrak{g}}$. 
\end{lemma}

Let us take $\lambda \in \Sigma$ and 
$0\neq X_{\lambda }\in \mathfrak{g}_{\mathbb{C}}(\mathfrak{a}_1:\lambda )$. 
By Lemma \ref{lem:Z_1}, we have 
\begin{align*}
X_{\lambda }=
\operatorname{id}_{\mathfrak{g}_{\mathbb{C}}}X_{\lambda }=e^{\operatorname{ad}4Z_1}X_{\lambda }
=e^{\sqrt{-1}\langle \lambda ,4Z_1\rangle }X_{\lambda }. 
\end{align*}
This means that $e^{\sqrt{-1}\langle \lambda ,4Z_1\rangle }=1$, 
which obtains $\langle \lambda ,4Z_1\rangle \in 2\pi \mathbb{Z}$. 
In view of this observation, 
we set 
\begin{align}
\label{eq:gamma}
\Gamma :=\left\{ 
	A\in \mathfrak{a}_1:\langle \lambda ,A\rangle \in \frac{\pi}{2}\mathbb{Z}~
	(\forall \lambda \in \Sigma )
\right\} . 
\end{align}
Therefore, we conclude: 

\begin{proposition}
\label{prop:sim}
Let $(\mathfrak{g},\theta _1,\theta _2)$ be a commutative compact semisimple symmetric triad. 
If the relation $\theta _1\sim \theta _2$ holds, 
then there exists $Z_1\in \Gamma$ such that 
\begin{align*}
(\mathfrak{g},\theta _1,\theta _2)\equiv 
(\mathfrak{g},\theta _1,e^{\operatorname{ad}Z_1}\theta _1e^{-\operatorname{ad}Z_1}). 
\end{align*}
\end{proposition}

%\begin{proof}
%%\footnote{この証明は不要か？}
%Under the notation as in (\ref{eq:e^Y}), 
%we define an automorphism on $\mathfrak{g}$ by 
%$\varphi (X)=e^{\operatorname{ad}Y_1}X$ 
%$(X\in \mathfrak{g})$. 
%It follows from Lemma \ref{lem:yz} that $\varphi \theta _1=\theta _1\varphi $. 
%Further, the equality (\ref{eq:theta _2}) means that 
%$\theta _2=\varphi (e^{\operatorname{ad}Z_1}\theta _1e^{-\operatorname{ad}Z_1})\varphi ^{-1}$. 
%\end{proof}

Using Proposition \ref{prop:sim}, it turns out that: 

\begin{proposition}
\label{prop:sim-dual}
Let $(\mathfrak{g},\theta _1,\theta _2)$ be a commutative compact semisimple symmetric triad. 
If $\theta _1\sim \theta _2$, then $(\mathfrak{g},\theta _1,\theta _2)$ is self-dual. 
\end{proposition}

\begin{proof}
By Proposition \ref{prop:sim} and Lemma \ref{lem:dual}, 
we shall show 
\begin{align*}
(\mathfrak{g},\theta _1,e^{\operatorname{ad}Z_1}\theta _1e^{-\operatorname{ad}Z_1})^d
\equiv (\mathfrak{g},\theta _1,e^{\operatorname{ad}Z_1}\theta _1e^{-\operatorname{ad}Z_1})
\end{align*}
for $Z_1\in \Gamma $. 
By Lemmas \ref{lem:yz} and \ref{lem:Z_1}, 
the involution $e^{\operatorname{ad}Z_1}\theta _1e^{-\operatorname{ad}Z_1}$ equals 
\begin{align*}
e^{\operatorname{ad}Z_1}\theta _1e^{-\operatorname{ad}Z_1}
&=e^{\operatorname{ad}2Z_1}\theta _1
=e^{-\operatorname{ad}2Z_1}\theta _1
=e^{-\operatorname{ad}Z_1}\theta _1e^{\operatorname{ad}Z_1}. 
\end{align*}
Then, we have 
\begin{align*}
(\mathfrak{g},\theta _1,e^{\operatorname{ad}Z_1}\theta _1e^{-\operatorname{ad}Z_1})^d
&=(\mathfrak{g},e^{\operatorname{ad}Z_1}\theta _1e^{-\operatorname{ad}Z_1},\theta _1)\\
&=(\mathfrak{g},e^{-\operatorname{ad}Z_1}\theta _1e^{\operatorname{ad}Z_1},\theta _1)\\
&\equiv (\mathfrak{g},\theta _1,e^{\operatorname{ad}Z_1}\theta _1e^{-\operatorname{ad}Z_1}). 
\end{align*}

Hence, Proposition \ref{prop:sim-dual} has been proved. 
\end{proof}

\subsubsection{Dual of commutative compact symmetric triad with $\theta _1\sim \theta _2$}
\label{subsubsec:duality-sim}

The next theorem is a key for the study of the dual 
of a commutative compact semisimple symmetric triad $(\mathfrak{g},\theta _1,\theta _2)$ 
with $\theta _1\sim \theta _2$. 
The following theorem is converse to Proposition \ref{prop:typeK-sim}. 

\begin{theorem}
\label{thm:sim-typeK}
Let $\mathfrak{g}$ be a compact semisimple Lie algebra, 
$\theta _1$ an involution on $\mathfrak{g}$ and $Z_1$ an element of $\Gamma$ (see (\ref{eq:gamma})). 
For the commutative compact semisimple symmetric pair 
$(\mathfrak{g},\theta _1,e^{\operatorname{ad}Z_1}\theta _1e^{-\operatorname{ad}Z_1})$, 
the dual $(\mathfrak{g},\theta _1,e^{\operatorname{ad}Z_1}\theta _1e^{-\operatorname{ad}Z_1})^*$ 
is of type $K_{\varepsilon }$. 
\end{theorem}

In order to prove Theorem \ref{thm:sim-typeK}, 
we begin with a general setup, based on \cite{takeuchi}. 
For this, we keep the setting as in Section \ref{subsubsec:property-sim}. 

Let us take and fix $Z_1\in \Gamma$. 
We choose a positive system $\Sigma ^+$ of $\Sigma $ 
characterized by $Z_1$, namely, 
\begin{align}
\label{eq:positive system}
\Sigma ^+:=\{ \lambda \in \Sigma :\langle \lambda ,Z_1\rangle >0\} , 
\end{align}
$\Sigma ^-:=-\Sigma ^+=\{ -\lambda :\lambda \in \Sigma ^+\} $ and 
$\Sigma ^0:=\{ \lambda \in \Sigma :\langle \lambda ,Z_1\rangle =0\} $. 
Then, the restricted root $\Sigma =\Sigma (\mathfrak{g}_{\mathbb{C}},\mathfrak{a}_1)$ 
is decomposed into the disjoint union $\Sigma =\Sigma ^+\sqcup \Sigma ^-\sqcup \Sigma ^0$. 
Obviously, $-\lambda \in \Sigma ^0$ if $\lambda \in \Sigma ^0$. 

For each $\lambda \in \Sigma$, 
we define a subspace $V(\lambda )$ in $\mathfrak{g}$ by 
\begin{align}
V(\lambda ):=\{ X\in \mathfrak{g}:(\operatorname{ad}A)^2X=-\langle \lambda ,A\rangle ^2X~
(\forall A\in \mathfrak{a}_1)\} . 
\label{eq:v}
\end{align}
By definition, $V(-\lambda )=V(\lambda )$ for any $\lambda \in \Sigma$. 
As $\mathfrak{a}_1\subset \mathfrak{p}_1$, $V(\lambda )$ is $\theta _1$-invariant. 
We set $\mathfrak{k}_1(\lambda ):=V(\lambda )\cap \mathfrak{k}_1$ and 
$\mathfrak{p}_1(\lambda ):=V(\lambda )\cap \mathfrak{p}_1$. 
Then, $V(\lambda )=\mathfrak{k}_1(\lambda )+\mathfrak{p}_1(\lambda )$ 
is the eigenspace decomposition of $\theta _1$. 
Further, 
the complexification $V(\lambda )_{\mathbb{C}}
=V(\lambda )+\sqrt{-1}V(\lambda )\subset \mathfrak{g}_{\mathbb{C}}$ 
coincides with the sum of two restricted root spaces 
$\mathfrak{g}_{\mathbb{C}}(\mathfrak{a}_1:\lambda )+\mathfrak{g}_{\mathbb{C}}(\mathfrak{a}_1:-\lambda )$. 
On the other hand, 
the centralizer $\mathfrak{z}(\mathfrak{a}_1)$ of $\mathfrak{a}_1$ in $\mathfrak{g}_{\mathbb{C}}$ 
is of the form 
$\mathfrak{z}(\mathfrak{a}_1)
=(\mathfrak{z}_{\mathfrak{k}_1}(\mathfrak{a}_1)+\mathfrak{a}_1)_{\mathbb{C}}$ 
where $\mathfrak{z}_{\mathfrak{k}_1}(\mathfrak{a}_1)$ denotes the centralizer of $\mathfrak{a}_1$ 
in $\mathfrak{k}_1$. 
Thus, the restricted root space decomposition (\ref{eq:root-space-decomp}) of $\mathfrak{g}_{\mathbb{C}}$ 
is the complexification of 
$
%\mathfrak{g}_{\mathbb{C}}
%&=(\mathfrak{z}_{\mathfrak{k}_1}(\mathfrak{a}_1)+\mathfrak{a}_1)_{\mathbb{C}}
%	+\sum _{\lambda \in \Sigma ^+\cup \Sigma ^0}V(\lambda )_{\mathbb{C}}
%=\left( 
\mathfrak{z}_{\mathfrak{k}_1}(\mathfrak{a}_1)+\mathfrak{a}_1
	+\sum _{\lambda \in \Sigma ^+\sqcup \Sigma ^0}V(\lambda )
%\right) _{\mathbb{C}}
$. 
Hence, we get a decomposition of $\mathfrak{g}$ as follows: 
\begin{align}
\mathfrak{g}=\mathfrak{z}_{\mathfrak{k}_1}(\mathfrak{a}_1)+\mathfrak{a}_1+
	\sum _{\lambda \in \Sigma ^+\sqcup \Sigma ^0}(\mathfrak{k}_1(\lambda )+\mathfrak{p}_1(\lambda )), 
\label{eq:compact-decomp}
\end{align}
and the ones of $\mathfrak{k}_1$ and $\mathfrak{p}_1$, respectively, as follows: 
\begin{align}
\label{eq:compact-decomp-k}
\mathfrak{k}_1&=\mathfrak{z}_{\mathfrak{k}_1}(\mathfrak{a}_1)+
	\sum _{\lambda \in \Sigma ^+\sqcup \Sigma ^0}\mathfrak{k}_1(\lambda ), \\
\label{eq:compact-decomp-p}
\mathfrak{p}_1&=\mathfrak{a}_1+
	\sum _{\lambda \in \Sigma ^+\sqcup \Sigma ^0}\mathfrak{p}_1(\lambda ). 
\end{align}

Next, we consider a linear transformation $f_{\lambda }$ on $\mathfrak{g}$ 
for each $\lambda \in \Sigma $ defined by 
\begin{align*}
f_{\lambda }(X)=\langle \lambda ,\lambda \rangle ^{-1}(\operatorname{ad}\lambda )(X) \quad 
(X\in \mathfrak{g}). 
\end{align*}
This induces a linear transformation on $V(\lambda )=V(-\lambda )$. 
Indeed, the following equality holds for any $A\in \mathfrak{a}_1$: 
\begin{align*}
(\operatorname{ad}A)^2((\operatorname{ad}\lambda )X)
&=(\operatorname{ad}\lambda )(\operatorname{ad}A)^2X\\
&=(\operatorname{ad}\lambda )(-\langle \lambda ,A\rangle ^2X)\\
&=-\langle \lambda ,A\rangle ^2((\operatorname{ad}\lambda )X). 
\end{align*}

\begin{lemma}
\label{lem:f}
$\langle f_{\lambda }(X_1),f_{\lambda }(X_2)\rangle =\langle X_1,X_2\rangle $ 
for any $X_1,X_2\in V(\lambda )$. 
\end{lemma}

\begin{proof}
Since the inner product $\langle \cdot ,\cdot \rangle $ is $G$-invariant, 
we have $\langle (\operatorname{ad}\lambda )Y_1,Y_2\rangle 
=-\langle Y_1,(\operatorname{ad}\lambda )Y_2\rangle $ for $Y_1,Y_2\in \mathfrak{g}$. 
We compute 
\begin{align*}
\langle f_{\lambda }(X_1),f_{\lambda }(X_2)\rangle 
&=\langle \lambda ,\lambda \rangle ^{-2}
	\langle (\operatorname{ad}\lambda )X_1,(\operatorname{ad}\lambda )X_2\rangle \\
&=-\langle \lambda ,\lambda \rangle ^{-2}
	\langle X_1,(\operatorname{ad}\lambda )^2X_2\rangle \\
&=-\langle \lambda ,\lambda \rangle ^{-2}
	\langle X_1,-\langle \lambda ,\lambda \rangle ^2X_2\rangle \\
&=\langle X_1,X_2\rangle . 
\end{align*}
Hence, Lemma \ref{lem:f} has been verified. 
\end{proof}

This implies that $f_{\lambda }$ is regular on $V(\lambda )$, 
and the inverse 
$(f_{\lambda }|_{V(\lambda )})^{-1}$ of the restriction of $f_{\lambda }$ to $V(\lambda )$ 
coincides with $f_{-\lambda }|_{V(\lambda )}$. 

Here, the image $f_{\lambda }(\mathfrak{k}_1(\lambda ))$ is contained in $\mathfrak{p}_1(\lambda )$ 
because for $X\in \mathfrak{k}_1(\lambda )$ 
\begin{align*}
\theta _1(f_{\lambda }(X))
&=\langle \lambda ,\lambda \rangle ^{-1}(\operatorname{ad}\theta _1(\lambda ))\theta _1(X)
=-\langle \lambda ,\lambda \rangle ^{-1}(\operatorname{ad}\lambda )X
=-f_{\lambda }(X). 
\end{align*}
Similarly, we have $f_{\lambda }(\mathfrak{p}_1(\lambda ))\subset \mathfrak{k}_1(\lambda )$. 
Hence, $f_{\lambda }$ yields a linear isomorphism 
from $\mathfrak{k}_1(\lambda )$ to $\mathfrak{p}_1(\lambda )$. 
Now, we set $d_{\lambda }:=\dim \mathfrak{k}_1(\lambda )=\dim \mathfrak{p}_1(\lambda )$. 

Let $\{ S_i^{\lambda }: 1\leq i\leq d_{\lambda }\} $ 
be an orthonormal basis of $\mathfrak{k}_1(\lambda )$. 
We set $T_i^{\lambda }:=f_{\lambda }(S_i^{\lambda })$ for $1\leq i\leq d_{\lambda }$. 
By Lemma \ref{lem:f}, $\{ T_i^{\lambda }:1\leq i\leq d_{\lambda }\} $ is an 
orthonormal basis of $\mathfrak{p}_1(\lambda )$. 

\begin{lemma}
\label{lem:st}
The following equalities hold for each $i=1,2,\ldots .d_{\lambda }$ and any $A\in \mathfrak{a}_1$: 
\begin{align}
\label{eq:st}
(\operatorname{ad}A)S_i^{\lambda }&=\langle \lambda ,A\rangle T_i^{\lambda }
%\quad (\forall A\in \mathfrak{a}_1)
,\\
\label{eq:ts}
(\operatorname{ad}A)T_i^{\lambda }&=-\langle \lambda ,A\rangle S_i^{\lambda }
%\quad (\forall A\in \mathfrak{a}_1)
%,\\
%\label{eq:triplet-3}
%[S_i^{\lambda },T_i^{\lambda }]&=\lambda 
. 
\end{align}
\end{lemma}

\begin{proof}
Let us verify (\ref{eq:st}). 
Suppose $A\in \mathfrak{a}_1$ satisfies $\langle \lambda ,A\rangle =0$. 
Then, $\langle \lambda ,A\rangle T_i^{\lambda }$ equals zero. 
On the other hand, 
we compute 
\begin{align*}
\langle (\operatorname{ad}A)S_i^{\lambda },(\operatorname{ad}A)S_i^{\lambda }\rangle 
&=-\langle S_i^{\lambda },(\operatorname{ad}A)^2S_i^{\lambda }\rangle \\
&=-\langle S_i^{\lambda },-\langle \lambda ,A\rangle ^2S_i^{\lambda }\rangle \\
&=0. 
\end{align*}
Thus, we have $(\operatorname{ad}A)S_i^{\lambda }=0$, 
which coincides with $\langle \lambda ,A\rangle T_i^{\lambda }$. 

In a general case where $A\in \mathfrak{a}_1$ satisfies $\langle \lambda ,A\rangle \neq 0$, 
we put $A'=A-\langle \lambda ,\lambda \rangle ^{-1}\langle \lambda ,A\rangle \lambda \in \mathfrak{a}_1$. 
Then, we have $\langle \lambda ,A'\rangle =0$. 
This relation explains $(\operatorname{ad}A')S_i^{\lambda }=0$. 
Hence, we obtain 
\begin{align*}
(\operatorname{ad}A)S_i^{\lambda }
&=(\operatorname{ad}A')S_i^{\lambda }
	+\langle \lambda ,\lambda \rangle ^{-1}\langle \lambda ,A\rangle 
	(\operatorname{ad}\lambda )S_i^{\lambda }
=\langle \lambda ,A\rangle T_i^{\lambda }. 
\end{align*}
This implies (\ref{eq:st}) for any $A\in \mathfrak{a}_1$. 

The equality (\ref{eq:ts}) follows from $f_{\lambda }^{-1}((\operatorname{ad}A)S_i^{\lambda })
=f_{\lambda }^{-1}(\langle \lambda ,A\rangle T_i^{\lambda })$. 
Indeed, 
we have $f_{\lambda }^{-1}((\operatorname{ad}A)S_i^{\lambda })
=f_{-\lambda }((\operatorname{ad}A)S_i^{\lambda })
=-(\operatorname{ad}A)T_i^{\lambda }$ and 
$f_{\lambda }^{-1}(\langle \lambda ,A\rangle T_i^{\lambda }=\langle \lambda ,A\rangle S_i^{\lambda }$. 

Therefore, Lemma \ref{lem:st} has been proved. 
\end{proof}

We are ready to give a proof of Theorem \ref{thm:sim-typeK}. 

\begin{proof}[Proof of Theorem \ref{thm:sim-typeK}]
We set $(\mathfrak{g}_0,\sigma ;\theta )
=(\mathfrak{g},\theta _1,e^{\operatorname{ad}Z_1}\theta _1e^{-\operatorname{ad}Z_1})^*$. 
By (\ref{eq:pair}), 
we write $\mathfrak{g}_0=\mathfrak{k}_1+\sqrt{-1}\mathfrak{p}_1$, 
$\theta =\theta _1$ and $\sigma =e^{\operatorname{ad}Z_1}\theta _1e^{-\operatorname{ad}Z_1}$. 
By the observation of (\ref{eq:compact-decomp}), 
$\mathfrak{g}_0$ is decomposed as follows 
\begin{align}
\label{eq:non-compact-decomp}
\mathfrak{g}_0=\mathfrak{z}_{\mathfrak{k}_1}(\mathfrak{a}_1)+\sqrt{-1}\mathfrak{a}_1
+\sum _{\lambda \in \Sigma ^+\sqcup \Sigma ^0}(\mathfrak{k}_1(\lambda )+\sqrt{-1}\mathfrak{p}_1(\lambda )). 
\end{align}

Using an element $Z_1\in \Gamma$, 
we define $Z\in \sqrt{-1}\mathfrak{a}_1$ by 
\begin{align}
Z=\frac{2Z_1}{\pi \sqrt{-1}}. 
\label{eq:ch}
\end{align}
Then, we will show that 
one can find a $\mathbb{Z}$-grading of $\mathfrak{g}_0$ 
whose characteristic element is $Z$. 
Namely, $\mathfrak{g}_0$ is of the form $\mathfrak{g}_0=\sum _{k\in \mathbb{Z}}\mathfrak{g}_0(k)$ 
under 
$\mathfrak{g}_0(k)=\{ X\in \mathfrak{g}_0:(\operatorname{ad}Z)X=kX\} $. 

Now, we fix $\lambda \in \Sigma ^+$ 
and take a basis of $\mathfrak{k}_1(\lambda )+\sqrt{-1}\mathfrak{p}_1(\lambda )$ as 
$\{ S_i^{\lambda }\pm \sqrt{-1}T_i^{\lambda }:1\leq i\leq d_{\lambda }\} $. 
%We put $k_{\lambda }:=\sqrt{-1}\langle \lambda ,Z\rangle =2\pi ^{-1}\langle \lambda ,Z_1\rangle $. 
Then, 
it follows from the relations (\ref{eq:st}) and (\ref{eq:ts}) that 
for each $i=1,2,\ldots ,d_{\lambda }$ we have 
\begin{align*}
\operatorname{ad}(Z)(S_i^{\lambda }\pm \sqrt{-1}T_i^{\lambda })
%&=\frac{\langle \lambda ,2Z_1\rangle }{\pi \sqrt{-1}}(T_i^{\lambda }\mp \sqrt{-1}S_i^{\lambda })
=\mp \sqrt{-1}\langle \lambda ,Z\rangle (S_i^{\lambda }\pm \sqrt{-1}T_i^{\lambda }). 
\end{align*}
Since $Z_1\in \Gamma $ satisfies $2\langle \lambda ,Z_1\rangle \in \pi \mathbb{Z}$ 
(see (\ref{eq:gamma})) and our choice of the positive system $\Sigma ^+$ is characterized by $Z_1$ 
(see (\ref{eq:positive system})), 
the number $\sqrt{-1}\langle \lambda ,Z\rangle =2\pi ^{-1}\langle \lambda ,Z_1\rangle $ 
is a positive integer. 
Hence, we get the following inclusion: 
%$S_i^{\lambda }-\sqrt{-1}T_i^{\lambda }\in \mathfrak{g}_0(k_{\lambda })$ and 
%$S_i^{\lambda }+\sqrt{-1}T_i^{\lambda }\in \mathfrak{g}_0(-k_{\lambda })$. 
%Thus, we get 
\begin{align*}
\mathfrak{k}_1(\lambda )+\sqrt{-1}\mathfrak{p}_1(\lambda )\subset 
\mathfrak{g}_0(\sqrt{-1}\langle \lambda ,Z\rangle )+\mathfrak{g}_0(-\sqrt{-1}\langle \lambda ,Z\rangle ). 
\end{align*}
Consequently, the subspace $\mathfrak{g}_0(k)+\mathfrak{g}_0(-k)$ for a positive integer $k$ 
is decomposed into as follows: 
\begin{align*}
\mathfrak{g}_0(k)+\mathfrak{g}_0(-k)
=\sum _{
	\begin{subarray}{c}
	\lambda \in \Sigma ^+\\
	\sqrt{-1}\langle \lambda ,Z\rangle =k
	\end{subarray}
}(\mathfrak{k}_1(\lambda )+\sqrt{-1}\mathfrak{p}_1(\lambda )). 
\end{align*}
Clearly, the eigenspace $\mathfrak{g}_0(0)$ with eigenvalue $0$ coincides with 
$\mathfrak{z}_{\mathfrak{k}_1}(\mathfrak{a}_1)+\sqrt{-1}\mathfrak{a}_1
+\sum _{\lambda \in \Sigma ^0}(\mathfrak{k}_1(\lambda )+\sqrt{-1}\mathfrak{p}_1(\lambda ))$. 
Therefore, it follows from (\ref{eq:non-compact-decomp}) that 
\begin{align}
\mathfrak{g}_0
=\mathfrak{g}_0(0)+\sum _{k\in \mathbb{Z}_+}(\mathfrak{g}_0(k)+\mathfrak{g}_0(-k))
=\sum _{k\in \mathbb{Z}}\mathfrak{g}_0(k) 
\label{eq:grading-typeK}
\end{align}
which is a $\mathbb{Z}$-grading of $\mathfrak{g}_0$ characterized by $Z$. 

As $Z_1\in \mathfrak{a}_1$, we have $\theta (Z)=-Z$. 
Hence, $\theta $ is a grade-reversing Cartan involution of $\mathfrak{g}_0$, from which 
$(Z,\theta )$ is the associated pair of the $\mathbb{Z}$-grading (\ref{eq:grading-typeK}). 

Finally, 
combining Lemma \ref{lem:yz} with (\ref{eq:ch}), 
$\sigma =e^{\operatorname{ad}Z_1}\theta _1e^{-\operatorname{ad}Z_1}
=e^{\operatorname{ad}(2Z_1)}\theta _1
=e^{\pi \sqrt{-1}\operatorname{ad}Z}\theta 
=\sigma _Z\theta $. 
Thus, 
$(\mathfrak{g}_0,\sigma )$ is of type $K_{\varepsilon }$ (see Definition \ref{def:typeK}). 

As a result, 
Theorem \ref{thm:sim-typeK} has been proved. 
\end{proof}

\subsubsection{A characterization of symmetric pair of type $K_{\varepsilon}$}
\label{subsubsec:characterization}

As a result, we establish a new characterization for 
a non-compact semisimple symmetric pair to be of type $K_{\varepsilon }$ 
by Theorem \ref{thm:duality-thm} as follows: 

\begin{theorem}
\label{thm:typeK}
Let $(\mathfrak{g}_0,\sigma )$ be a non-compact semisimple symmetric pair 
and $(\mathfrak{g},\theta _1,\theta _2):=(\mathfrak{g}_0,\sigma )^*$ 
the corresponding commutative compact semisimple symmetric triad. 
Then, the following conditions are equivalent: 
\begin{enumerate}
	\renewcommand{\theenumi}{\roman{enumi}}
	\item $(\mathfrak{g}_0,\sigma )$ is of type $K_{\varepsilon }$ (see Definition \ref{def:typeK}). 
	\label{item:typeK}
	\item $\theta _1\sim \theta _2$ (see Definition \ref{def:sim}). 
%	namely, 
%	$\theta _2=e^{\operatorname{ad}Y}\theta _1e^{-\operatorname{ad}Y}$ for some $Y\in \mathfrak{g}$. 
	\label{item:sim}
\end{enumerate}
\end{theorem}

\begin{proof}
The implication (\ref{item:typeK}) $\Rightarrow$ (\ref{item:sim}) is an immediate consequence 
of Proposition \ref{prop:typeK-sim}. 

Let us assume $(\mathfrak{g},\theta _1,\theta _2)$ satisfies $\theta _1\sim \theta _2$. 
By Proposition \ref{prop:sim}, 
there exists $Z_1\in \Gamma$ such that 
$(\mathfrak{g},\theta _1,\theta _2)\equiv 
(\mathfrak{g},\theta _1,e^{\operatorname{ad}Z_1}\theta _1e^{-\operatorname{ad}Z_1})$. 
Thanks to Theorem \ref{thm:sim-typeK}, 
the dual $(\mathfrak{g},\theta _1,e^{\operatorname{ad}Z_1}\theta _1e^{-\operatorname{ad}Z_1})^*$ 
is of type $K_{\varepsilon}$. 
As Lemma \ref{lem:equiv-typeK}, 
$(\mathfrak{g}_0,\sigma )
=(\mathfrak{g},\theta _1,\theta _2)^*
\equiv (\mathfrak{g},\theta _1,e^{\operatorname{ad}Z_1}\theta _1e^{-\operatorname{ad}Z_1})^*$ 
is of type $K_{\varepsilon }$. 
Hence, the implication (\ref{item:sim}) $\Rightarrow $ (\ref{item:typeK}) also holds. 
\end{proof}

\begin{eg}
\label{eg:typeK}
Let us consider a non-compact semisimple symmetric pair $(\mathfrak{g}_0,\sigma )$ 
where $\mathfrak{g}_0=\mathfrak{sl}(n,\mathbb{R})=\{ X\in M(n,\mathbb{R}):\operatorname{Tr}X=0\} $ 
and $\sigma (X)=-I_{2p,n-2p}{{}^tX}I_{2p,n-2p}$ for $0<2p<n$. 
Here, $I_{2p,n-2p}\in M(n,\mathbb{R})$ stands for the diagonal matrix given by (\ref{eq:I_{m,n}}). 
Then, the fixed point set $\mathfrak{g}_0^{\sigma }$ equals $\mathfrak{so}(2p,n-2p)$ 
(see (\ref{eq:so(2p,2q)}) for realization). 

%Let us see that $(\mathfrak{g}_0,\sigma )$ is of type $K_{\varepsilon }$. 
%For this, 
We take an element $Z$ of $\mathfrak{g}_0$ as 
\begin{align*}
Z=\operatorname{diag}(\underbrace{1,\ldots ,1}_{p},\underbrace{-1,\ldots ,-1}_{p},
	\underbrace{0,\ldots ,0}_{n-2p}). 
\end{align*}
Then, the Lie algebra $\mathfrak{g}_0$ has a $\mathbb{Z}$-grading of second kind. 
Let us choose a Cartan involution $\theta$ of $\mathfrak{g}_0$ commuting with $\sigma$ 
as $\theta (X)=-{}\,^tX$ $(X\in \mathfrak{g}_0)$. 
Clearly, 
$\theta $ is a grade-reversing Cartan involution, 
and then $(Z,\theta )$ is the associated pair of this $\mathbb{Z}$-grading. 
Moreover, the direct computation shows $\sigma =\sigma _Z\theta $. 

On the other hand, 
the commutative compact semisimple symmetric triad $(\mathfrak{g},\theta _1,\theta _2)$ 
corresponding to $(\mathfrak{g}_0,\sigma )$ via Theorem \ref{thm:duality-thm} is characterized as follows. 
The compact simple Lie algebra $\mathfrak{g}$ is 
$\mathfrak{su}(n)=\{ X\in M(n,\mathbb{C}):\overline{{}^tX}+X=O\} $, 
and the fixed point sets $\mathfrak{g}^{\theta _1}$ and $\mathfrak{g}^{\theta _2}$ are given by 
$\mathfrak{g}^{\theta _1}=\{ X\in \mathfrak{su}(n):\overline{X}=X\} =\mathfrak{so}(n)$ and 
$\mathfrak{g}^{\theta _2}=\{ X\in \mathfrak{su}(n):I_{2p,n-2p}\overline{X}I_{2p,n-2p}=X\} $, respectively. 
Thus, we obtain $\mathfrak{g}^{\theta _1}\simeq \mathfrak{g}^{\theta _2}$ 
via the Lie algebra isomorphism $\mathfrak{g}^{\theta _1}\to \mathfrak{g}^{\theta _2}$, 
$X\mapsto I_{2p,n-2p}XI_{2p,n-2p}$. 
Furthermore, we have 
$\theta _2=e^{\operatorname{ad}Z_1}\theta _1e^{-\operatorname{ad}Z_1}$ 
for $Z_1=\frac{\pi \sqrt{-1}}{2}Z$. 
Hence, $\theta _1\sim \theta _2$. 
\end{eg}

%It should be pointed out that 
%\begin{align*}
%(\mathfrak{g}_0,\sigma ;\theta )^*=(\mathfrak{g},\theta _1,\theta _2)
%\not \equiv (\mathfrak{g},\theta _1,\theta _1)=(\mathfrak{g}_0,\theta ;\theta )^* 
%\end{align*}
%happens 
%even though $\mathfrak{g}_0^{\theta }=\mathfrak{g}^{\theta _1}\simeq \mathfrak{g}^{\theta _2}$. 
%Indeed, this follows from $(\mathfrak{g}_0,\sigma )\not \equiv (\mathfrak{g}_0,\theta )$ 
%and Theorem \ref{thm:duality-thm}. 

Finally, we end this subsection by the following corollary. 

\begin{corollary}
\label{cor:tykeK-dual}
A non-compact semisimple symmetric pair $(\mathfrak{g}_0,\sigma )$ of type $K_{\varepsilon }$ 
is self-dual. 
\end{corollary}

\begin{proof}
Let $\theta $ be a Cartan involution of $\mathfrak{g}_0$ commuting with $\sigma$ and 
$(\mathfrak{g},\theta _1,\theta _2)=(\mathfrak{g}_0,\sigma )^*$ the 
commutative compact semisimple symmetric triad corresponding to $(\mathfrak{g}_0,\sigma )$. 
It follows from Theorem \ref{thm:typeK} 
(or Proposition \ref{prop:typeK-sim}) that $\theta _1\sim \theta _2$. 
By Proposition \ref{prop:sim-dual}, $(\mathfrak{g},\theta _1,\theta _2)^d
\equiv (\mathfrak{g},\theta _1,\theta _2)$. 
Using Proposition \ref{prop:compatible-n}, we have 
\begin{align*}
(\mathfrak{g}_0,\sigma ;\theta )^d
&=(\mathfrak{g}_0,\sigma ;\theta )^{*d*}
=(\mathfrak{g},\theta _1,\theta _2)^{d*}
\equiv (\mathfrak{g},\theta _1,\theta _2)^{*}
=(\mathfrak{g}_0,\sigma ;\theta ). 
\end{align*}
Hence, Corollary \ref{cor:tykeK-dual} has been proved. 
\end{proof}

As we have seen in Example \ref{eg:typeK}, 
it happens that 
$(\mathfrak{g},\theta _1,\theta _2):=(\mathfrak{g}_0,\sigma ;\theta )^*
\not \equiv (\mathfrak{g}_0,\theta ;\theta )^*=(\mathfrak{g},\theta _1,\theta _1)$ 
for any non-compact semisimple symmetric pair $(\mathfrak{g}_0,\sigma ;\theta )$ 
of type $K_{\varepsilon }$. 
This fact exhibits the difficulty 
in specifying the dual of a commutative compact semisimple symmetric triad 
$(\mathfrak{g},\theta _1,\theta _2)$. 
Then, we need to find a systematic description of $(\mathfrak{g},\theta _1,\theta _2)^*$ 
whether $\theta _1\sim \theta _2$ or not, 
which is an essential part of our method to classify pseudo-Riemannian simple symmetric pairs. 

Loosely speaking, our method consists of two parts. 
One is that we shall give a characterization of commutative compact semisimple symmetric triads 
in terms of symmetric triads with multiplicities 
which has been introduced by the second author in \cite{ikawa-jms}. 
The other is to determine the intersection $\mathfrak{g}^{\theta _1}\cap \mathfrak{g}^{\theta _2}$ 
from $(\mathfrak{g},\theta _1,\theta _2)\in \mathfrak{T}$. 
In fact, 
we will classify commutative compact simple symmetric triads up to the equivalence relation $\equiv$. 
Then, 
we shall gain the classification of non-compact pseudo-Riemannian simple symmetric pairs 
as the dual of commutative compact simple symmetric triads, 
which provides an alternative proof of it due to Berger. 
The detail will be explained in the forthcoming paper \cite{bis1}. 

%%%%%%%%%%%%%%%%%%%%%%%%%%%%%%%%%%%%%%%%%%%%%%%%%%%%%%%%%%%%%%%%%%%%%%%%%%%%%%%%%%%%%%%%%%%%%%%%%%%%

\section{Appendix}
\label{sec:appendix}

This appendix concentrates on the notion of irreducible non-compact 
pseudo-Riemannian semisimple symmetric pairs. 
Throughout this section, 
let $(\mathfrak{g}_0,\sigma )$ be a non-compact semisimple symmetric pair 
and we set $\mathfrak{h}_0:=\mathfrak{g}_0^{\sigma }$, $\mathfrak{q}_0:=\mathfrak{g}_0^{-\sigma }$. 

We adopt the definition for $(\mathfrak{g}_0,\sigma )$ to be irreducible 
if Definition \ref{def:irr-pair} holds, namely, 
\begin{enumerate}
	\renewcommand{\labelenumi}{(N\theenumi) }
	\item 
	\label{item:def}
	there does not exist non-trivial $\sigma$-invariant ideals of $\mathfrak{g}_0$. 
\end{enumerate}

On the other hand, 
the reference \cite[p.435]{oshima-sekiguchi} says that 
$(\mathfrak{g}_0,\sigma )$ is irreducible 
if 
\begin{enumerate}
	\renewcommand{\labelenumi}{(N\theenumi) }
	\addtocounter{enumi}{1}
	\item 
	\label{item:adjoint}
	the adjoint action of $\mathfrak{h}_0$ on $\mathfrak{q}_0$ is irreducible. 
\end{enumerate}

We will compare two notions (N\ref{item:def}) and (N\ref{item:adjoint}) for $(\mathfrak{g}_0,\sigma )$. 
First, 
we show that (N\ref{item:adjoint}) implies (N\ref{item:def}) in Section \ref{subsec:adjoint->irreducible}. 
Second, we prove that the opposite is also true for Riemannian semisimple symmetric pairs 
in Section \ref{subsec:riemannian setting}. 
Third, we provide a counterexample of the implication (N\ref{item:def}) $\Rightarrow$ (N\ref{item:adjoint}) 
in Section \ref{subsec:counterexample}.

\subsection{Effective semisimple symmetric pair}
\label{subsec:effective}

The studies of (N\ref{item:def}) and (N\ref{item:adjoint}) will be carried out 
under the setting that $(\mathfrak{g}_0,\sigma )$ is effective without loss of generality. 
Here is a brief summary on effective non-compact semisimple symmetric pairs. 
For this, we consider a homogeneous space of a Lie group as follows. 

Let $G$ be a Lie group and $H$ a closed subgroup of $G$. 
The group $G$ acts on $G/H$ by the left transformation, namely, $g\cdot xH:=(gx)H$ ($g,x\in G$). 
Then, the isotropy subgroup $G_{xH}=\{ g\in G:g\cdot xH=xH\} $ of $G$ at $xH\in G/H$ equals $xHx^{-1}$. 
We set $S_H:=\bigcap _{x\in G}xHx^{-1}$. 
Then, the $G$-action on $G/H$ is called {\it effective} if $S_H$ coincides with $\{ e\} $. 

We observe that $S_H$ is a normal subgroup of $G$ and is contained in $H$. 
On the other hand, 
an arbitrary normal subgroup $N$ of $G$ with $N\subset H$ has to be also contained in $S_H$. 
Indeed, this follows from $N=xNx^{-1}\subset xHx^{-1}$ for any $x\in G$. 
Hence, the $G$-action on $G/H$ to be effective if and only if 
any normal subgroup of $G$ contained in $H$ equals $\{ e\} $. 

In this context, 
the notion of effective non-compact semisimple symmetric pairs 
is a Lie algebra version of the notion of effective Lie group actions on homogeneous spaces. 
More precisely, we define: 

\begin{define}
\label{def:effective}
A non-compact semisimple symmetric pair $(\mathfrak{g}_0,\sigma )$ is {\it effective} 
if any ideal of $\mathfrak{g}_0$ contained in $\mathfrak{h}_0$ equals $\{ 0\} $. 
\end{define}

It is without loss of generality 
for the study on the relation between (N\ref{item:def}) and (N\ref{item:adjoint}) 
that $(\mathfrak{g}_0,\sigma )$ is assumed to be effective. 

\subsection{Implication from (N\ref{item:adjoint}) to (N\ref{item:def})}
\label{subsec:adjoint->irreducible}

%In this subsection, 
%we consider the irreducibility of $(\mathfrak{g}_0,\sigma )$ 
%under the assumption that the adjoint $\mathfrak{h}_0$-action on $\mathfrak{q}_0 $ is irreducible. 
%
%Before the explanation, 
We begin with a general setup for the consideration of the implication 
(N\ref{item:adjoint}) $\Rightarrow$ (N\ref{item:def}) as follows. 
We denote by $B$ the Killing form of $\mathfrak{g}_0$. 
For a subspace $\mathfrak{l}_0$ of $\mathfrak{g}_0$, 
we write $\mathfrak{l}_0^{\bot }:=\{ X\in \mathfrak{g}_0:B(X,Y)=0~(\forall Y\in \mathfrak{l}_0)\} $ 
for the orthogonal complement of $\mathfrak{l}_0$ in $\mathfrak{g}_0 $ with respect to $B$. 

From now, we assume that $\mathfrak{l}_0$ is an ideal of $\mathfrak{g}_0$. 

\begin{lemma}
\label{lem:bracket}
If $\mathfrak{l}_0$ is an ideal of $\mathfrak{g}_0$, 
then $[\mathfrak{l}_0,\mathfrak{l}_0^{\bot}]=\{ 0\} $. 
\end{lemma}

\begin{proof}
We observe 
$B(\mathfrak{g}_0,[\mathfrak{l}_0,\mathfrak{l}_0^{\bot }])=\{ 0\} $ because 
\begin{align*}
B(\mathfrak{g}_0,[\mathfrak{l}_0,\mathfrak{l}_0^{\bot }])
=B([\mathfrak{g}_0,\mathfrak{l}_0],\mathfrak{l}_0^{\bot })\subset B(\mathfrak{l}_0,\mathfrak{l}_0^{\bot })
=\{ 0\} . 
\end{align*}
Since $B$ is non-degenerate, we obtain $[\mathfrak{l}_0,\mathfrak{l}_0^{\bot}]=\{ 0\} $. 
\end{proof}

The orthogonal complement $\mathfrak{l}_0^{\bot }$ becomes an ideal of $\mathfrak{g}_0$, 
and then so is $\mathfrak{b}_0:=\mathfrak{l}_0\cap \mathfrak{l}_0^{\bot } $. 
By Lemma \ref{lem:bracket}, 
we obtain $[\mathfrak{b}_0,\mathfrak{b}_0]
\subset [\mathfrak{l}_0,\mathfrak{l}_0^{\bot } ]=\{ 0\} $. 
Thus, $\mathfrak{b}_0$ is an abelian ideal. 
Let $\mathfrak{c}_0$ be a complementary subspace to $\mathfrak{b}_0$ in $\mathfrak{g}_0$. 
For any $X\in \mathfrak{g}_0$ and $A\in \mathfrak{b}_0$, 
the endomorphism $\operatorname{ad}A\operatorname{ad}X$ on $\mathfrak{g}_0$ 
maps $\mathfrak{b}_0$ into $\{ 0\} $ since $\mathfrak{b}_0$ is an abelian 
and $\mathfrak{c}_0$ into $\mathfrak{b}_0$ since $\mathfrak{b}_0$ is an ideal of $\mathfrak{g}_0$. 
Then, we have $B(\mathfrak{b}_0,\mathfrak{g}_0)
=\operatorname{Tr}(\operatorname{ad}\mathfrak{b}_0\operatorname{ad}\mathfrak{g}_0)=\{ 0\} $. 
Hence, $\mathfrak{b}_0=\mathfrak{l}_0\cap \mathfrak{l}_0^{\bot } =\{ 0\} $. 
Therefore, 
$\mathfrak{g}_0$ is decomposed into the direct sum of $\mathfrak{l}_0$ and $\mathfrak{l}_0^{\bot }$, 
namely, $\mathfrak{g}_0=\mathfrak{l}_0+\mathfrak{l}_0^{\bot }$ 
(see \cite[Proposition 6.1 in Chapter II]{helgason}).

\begin{lemma}
\label{lem:bot-invariant}
If $\mathfrak{l}_0$ is $\sigma$-invariant, 
then $\mathfrak{l}_0^{\bot }$ is $\sigma$-invariant
\end{lemma}

\begin{proof}
As $\sigma (\mathfrak{l}_0)=\mathfrak{l}_0$, 
we have 
$B(\sigma (\mathfrak{l}_0^{\bot}),\mathfrak{l}_0)
=B(\mathfrak{l}_0^{\bot },\sigma (\mathfrak{l}_0))=B(\mathfrak{l}_0^{\bot },\mathfrak{l}_0)=\{ 0\} $. 
Hence, we have verified $\sigma (\mathfrak{l}_0^{\bot })\subset \mathfrak{l}_0^{\bot }$. 
\end{proof}

Now, let $(\mathfrak{g}_0,\sigma )$ be an effective non-compact semisimple symmetric pair 
for which the condition (N\ref{item:def}) does not hold. 
Then, we have: 

\begin{proposition}
\label{prop:non-trivial-invariant}
Let $(\mathfrak{g}_0,\sigma )$ be an effective non-compact semisimple symmetric pair. 
If there exists a non-trivial $\sigma$-invariant ideal $\mathfrak{l}_0$ of $\mathfrak{g}_0$, 
then $\mathfrak{q}_0\cap \mathfrak{l}_0$ 
is a non-trivial $(\operatorname{ad}\mathfrak{h}_0)$-invariant subspace of $\mathfrak{q}_0$. 
\end{proposition}

\begin{proof}
Retain the notation as above. 
Since $\mathfrak{l}_0$ is $\sigma$-invariant, 
we write $\mathfrak{l}_0=\mathfrak{h}_{\mathfrak{l}_0}+\mathfrak{q}_{\mathfrak{l}_0}$ 
for the $\sigma $-eigenspace decomposition of $\mathfrak{l}_0$ 
with $\mathfrak{h}_{\mathfrak{l}_0}=\mathfrak{l}_0^{\sigma }=\mathfrak{h}_0\cap \mathfrak{l}_0$ 
and $\mathfrak{q}_{\mathfrak{l}_0}=\mathfrak{l}_0^{-\sigma }=\mathfrak{q}_0\cap \mathfrak{l}_0$. 
On the other hand, $\mathfrak{l}_0^{\bot} $ is also a $\sigma$-invariant ideal of $\mathfrak{g}_0$ 
(see Lemma \ref{lem:bot-invariant}). 
Then, $\mathfrak{l}_0^{\bot }=\mathfrak{h}_{\mathfrak{l}_0^{\bot }}+\mathfrak{q}_{\mathfrak{l}_0^{\bot }}$ 
is a $\sigma$-eigenspace decomposition of $\mathfrak{l}_0^{\bot}$ 
with $\mathfrak{h}_{\mathfrak{l}_0^{\bot }}=(\mathfrak{l}_0^{\bot })^{\sigma }
=\mathfrak{h}_0\cap \mathfrak{l}_0^{\bot }$ 
and $\mathfrak{q}_{\mathfrak{l}_0^{\bot }}=(\mathfrak{l}_0^{\bot })^{-\sigma }
=\mathfrak{q}_0\cap \mathfrak{l}_0^{\bot }$. 
Using them, $\mathfrak{h}_0$ can be written as 
$\mathfrak{h}_0=\mathfrak{h}_{\mathfrak{l}_0}+\mathfrak{h}_{\mathfrak{l}_0^{\bot }}$ 
and $\mathfrak{q}_0$ as 
$\mathfrak{q}_0=\mathfrak{q}_{\mathfrak{l}_0}+\mathfrak{q}_{\mathfrak{l}_0^{\bot }}$ 
along the decomposition $\mathfrak{g}_0=\mathfrak{l}_0+\mathfrak{l}_0^{\bot}$. 

The subspace $\mathfrak{q}_{\mathfrak{l}_0}$ of $\mathfrak{q}_0$ has to be non-zero. 
Indeed, if $\mathfrak{q}_{\mathfrak{l}_0}=\{ 0\} $, 
then the ideal $\mathfrak{l}_0=\mathfrak{h}_{\mathfrak{l}_0}$ is contained in $\mathfrak{h}_0$. 
Since $(\mathfrak{g}_0,\sigma )$ is effective, 
$\mathfrak{l}_0$ equals $\{ 0\} $, 
which contradicts to $\mathfrak{l}_0\neq \{ 0\} $. 
%Similarly, we have $\mathfrak{q}_{\mathfrak{l}_0^{\bot }}\neq \{ 0\} $.

Here, it follows from Lemma \ref{lem:bracket} that 
$[\mathfrak{h}_{\mathfrak{l}_0^{\bot }},\mathfrak{q}_{\mathfrak{l}_0}]
\subset [\mathfrak{l}_0^{\bot },\mathfrak{l}_0]=\{ 0\} $. 
Combining it with the relation 
$[\mathfrak{h}_{\mathfrak{l}_0},\mathfrak{q}_{\mathfrak{l}_0}]
=[\mathfrak{h}_0,\mathfrak{q}_0]\cap \mathfrak{l}_0
\subset \mathfrak{q}_0\cap \mathfrak{l}_0=\mathfrak{q}_{\mathfrak{l}_0}$, 
we have shown 
$[\mathfrak{h}_0,\mathfrak{q}_{\mathfrak{l}_0}]
=[\mathfrak{h}_{\mathfrak{l}_0},\mathfrak{q}_{\mathfrak{l}_0}]
	+[\mathfrak{h}_{\mathfrak{l}_0^{\bot }},\mathfrak{q}_{\mathfrak{l}_0}]
\subset \mathfrak{q}_{\mathfrak{l}_0}$. 
Hence, $\mathfrak{q}_{\mathfrak{l}_0}$ is 
a $(\operatorname{ad}\mathfrak{h}_0)$-invariant subspace of $\mathfrak{q}_0$. 
\end{proof}

Therefore, we get the implication (N\ref{item:adjoint}) $\Rightarrow$ (N\ref{item:def}) 
as a contraposition to Proposition \ref{prop:non-trivial-invariant}. 

\begin{theorem}[(N\ref{item:adjoint}) $\Rightarrow$ (N\ref{item:def})]
\label{thm:action->irreducible}
Let $(\mathfrak{g}_0,\sigma )$ be an effective non-compact semisimple symmetric pair. 
If the $\mathfrak{h}_0$-action on $\mathfrak{q}_0$ is irreducible 
then there does not exist non-trivial $\sigma $-invariant ideals of $\mathfrak{g}_0$. 
\end{theorem}

\subsection{Equivalence of (N\ref{item:def}) and (N\ref{item:adjoint}) for Riemannian semisimple symmetric pair}
\label{subsec:riemannian setting}

In this subsection, 
we will treat a special case where $(\mathfrak{g}_0,\sigma )$ is a Riemannian semisimple symmetric pair, 
namely, $\sigma$ is a Cartan involution $\theta $ of $\mathfrak{g}_0$. 
Then, we shall replace $\mathfrak{h}_0$ by $\mathfrak{k}_0=\mathfrak{g}_0^{\theta }$ 
and $\mathfrak{q}_0$ by $\mathfrak{p}_0=\mathfrak{g}_0^{-\theta }$. 
From now, 
let us consider the opposite implication (N\ref{item:def}) $\Rightarrow$ (N\ref{item:adjoint}) 
for $(\mathfrak{g}_0,\theta )$. 

Let $\mathfrak{p}_1\neq \{ 0\} $ be 
a $(\operatorname{ad}\mathfrak{k}_0)$-invariant subspace of $\mathfrak{p}_0$. 
As $\theta |_{\mathfrak{p}_0}=-\operatorname{id}_{\mathfrak{p}_0}$, 
$\mathfrak{p}_1$ is $\theta $-invariant. 
We write $\mathfrak{p}_2=\{ Y\in \mathfrak{p}_0:B(Y,Y_1)=0~(\forall Y_1\in \mathfrak{p}_1)\} $ 
for the orthogonal complement of $\mathfrak{p}_1$ in $\mathfrak{p}_0$ with respect to $B$. 
Since $B|_{\mathfrak{p}_0\times \mathfrak{p}_0}$ is positive definite, 
$\mathfrak{p}_0$ is decomposed into the direct sum of $\mathfrak{p}_1$ and $\mathfrak{p}_2$, 
namely, $\mathfrak{p}_0=\mathfrak{p}_1+\mathfrak{p}_2$ and $\mathfrak{p}_1\cap \mathfrak{p}_2=\{ 0\} $. 
Further, we have: 

\begin{lemma}
\label{lem:p1p2}
$[\mathfrak{p}_1,\mathfrak{p}_2]=\{ 0\} $. 
\end{lemma}

\begin{proof}
We recall that $[\mathfrak{k}_0,\mathfrak{p}_1]\subset \mathfrak{p}_1$ 
and $[\mathfrak{p}_1,\mathfrak{p}_2]\subset [\mathfrak{p}_0,\mathfrak{p}_0]\subset \mathfrak{k}_0$. 
For $Y_1\in \mathfrak{p}_1$ and $Y_2\in \mathfrak{p}_2$, we have 
$B(X,[Y_1,Y_2])=B([X,Y_1],Y_2)=0$ for any $X\in \mathfrak{k}_0$. 
Since $B|_{\mathfrak{k}_0\times \mathfrak{k}_0}$ is negative definite, 
we obtain $[Y_1,Y_2]=0$. 
\end{proof}

We define a subspace $\mathfrak{l}_0\neq \{ 0\} $ of $\mathfrak{g}_0$ by 
\begin{align}
\label{eq:l0}
\mathfrak{l}_0:=[\mathfrak{p}_1,\mathfrak{p}_1]+\mathfrak{p}_1. 
\end{align}
Then, $\mathfrak{l}_0$ is $\theta $-invariant because 
$\theta (\mathfrak{l}_0)
=[\theta (\mathfrak{p}_1),\theta (\mathfrak{p}_1)]+\theta (\mathfrak{p}_1)
=[\mathfrak{p}_1,\mathfrak{p}_1]+\mathfrak{p}_1=\mathfrak{l}_0$. 
Now, we show: 

\begin{proposition}
\label{prop:ideal}
$\mathfrak{l}_0=[\mathfrak{p}_1,\mathfrak{p}_1]+\mathfrak{p}_1$ is an ideal of $\mathfrak{g}_0$. 
\end{proposition}

\begin{proof}
We write $\mathfrak{g}_0=\mathfrak{k}_0+\mathfrak{p}_0$ for the corresponding Cartan decomposition. 
According to the decomposition $\mathfrak{p}_0=\mathfrak{p}_1+\mathfrak{p}_2$, 
we obtain $\mathfrak{g}_0=\mathfrak{k}_0+\mathfrak{p}_1+\mathfrak{p}_2$. 
Then, it is necessary to show following three relations: 
\begin{align}
[\mathfrak{k}_0,\mathfrak{l}_0]\subset \mathfrak{l}_0,\quad 
[\mathfrak{p}_1,\mathfrak{l}_0]\subset \mathfrak{l}_0,\quad 
[\mathfrak{p}_2,\mathfrak{l}_0]=\{ 0\} . 
\end{align}

First, the Jacobi identity shows $[\mathfrak{k}_0,[\mathfrak{p}_1,\mathfrak{p}_1]]
=[\mathfrak{p}_1,[\mathfrak{k}_0,\mathfrak{p}_1]]$. 
Further, the relation $[\mathfrak{k}_0,\mathfrak{p}_1]\subset \mathfrak{p}_1$ 
gives the inclusion $[\mathfrak{p}_1,[\mathfrak{k}_0,\mathfrak{p}_1]]\subset [\mathfrak{p}_1,\mathfrak{p}_1]$. 
Thus, we have 
$[\mathfrak{k}_0,\mathfrak{l}_0]
=[\mathfrak{k}_0,[\mathfrak{p}_1,\mathfrak{p}_1]]+[\mathfrak{k}_0,\mathfrak{p}_1]
\subset [\mathfrak{p}_1,\mathfrak{p}_1]+\mathfrak{p}_1=\mathfrak{l}_0$. 

Second, as $[\mathfrak{p}_1,\mathfrak{p}_1]\subset \mathfrak{k}_0$ 
and $[\mathfrak{k}_0,\mathfrak{p}_1]\subset \mathfrak{p}_1$, 
we have $[\mathfrak{p}_1,\mathfrak{l}_0]
=[\mathfrak{p}_1,[\mathfrak{p}_1,\mathfrak{p}_1]]+[\mathfrak{p}_1,\mathfrak{p}_1]
\subset [\mathfrak{p}_1,\mathfrak{k}_0]+[\mathfrak{p}_1,\mathfrak{p}_1]
\subset \mathfrak{p}_1+[\mathfrak{p}_1,\mathfrak{p}_1]=\mathfrak{l}_0$. 

Third, the Jacobi identity implies that 
$[\mathfrak{p}_2,[\mathfrak{p}_1,\mathfrak{p}_1]]=[\mathfrak{p}_1,[\mathfrak{p}_1,\mathfrak{p}_2]]$. 
By Lemma \ref{lem:p1p2}, 
this equals $[\mathfrak{p}_1,\{ 0\} ]=\{ 0\} $. 
Thus, we obtain $[\mathfrak{p}_2,\mathfrak{l}_0]
=[\mathfrak{p}_2,[\mathfrak{p}_1,\mathfrak{p}_1]]+[\mathfrak{p}_2,\mathfrak{p}_1]=\{ 0\} $. 

Hence, we conclude 
\begin{align*}
[\mathfrak{g}_0,\mathfrak{l}_0]
=[\mathfrak{k}_0,\mathfrak{l}_0]+[\mathfrak{p}_1,\mathfrak{l}_0]+[\mathfrak{p}_2,\mathfrak{l}_0]
\subset \mathfrak{l}_0. 
\end{align*}
Therefore, we have proved Proposition \ref{prop:ideal}. 
\end{proof}

Using the $\theta $-invariant ideal $\mathfrak{l}_0$ given by (\ref{eq:l0}), we prove: 

\begin{theorem}[(N\ref{item:def}) $\Leftrightarrow$ (N\ref{item:adjoint}) for Riemannian semisimple symmetric pair]
Let $(\mathfrak{g}_0,\theta )$ be a non-compact Riemannian semisimple symmetric pair. 
Then, the adjoint $\mathfrak{k}_0$-action on $\mathfrak{p}_0$ is irreducible 
if and only if there does not exist non-trivial $\sigma$-invariant ideals of $\mathfrak{g}_0$. 
\end{theorem}

\begin{proof}
The necessary condition is a direct consequence of Theorem \ref{thm:action->irreducible}. 
Then, it suffices to show the sufficient condition. 

For a $(\operatorname{ad}\mathfrak{k}_0)$-invariant subspace $\mathfrak{p}_1\neq \{ 0\} $ in $\mathfrak{p}_0$, 
the subspace $\mathfrak{l}_0\neq \{ 0\} $ defined by (\ref{eq:l0}) 
is a $\theta $-invariant ideal of $\mathfrak{g}_0$ 
(see Proposition \ref{prop:ideal}). 
If there does not exist non-trivial $\sigma$-invariant ideals of $\mathfrak{g}_0$, 
then $\mathfrak{l}_0$ must to be $\mathfrak{g}_0$. 
Thus, we get 
$[\mathfrak{p}_1,\mathfrak{p}_1]+\mathfrak{p}_1=\mathfrak{k}_0+\mathfrak{p}_0$. 
This implies $\mathfrak{p}_1=\mathfrak{p}_0$. 
Hence, the adjoint $\mathfrak{k}_0$-action on $\mathfrak{p}_0$ is irreducible. 
\end{proof}

\subsection{The implication (N\ref{item:def}) $\Rightarrow $ (N\ref{item:adjoint}) for pseudo-Riemannian symmetric pair}
\label{subsec:counterexample}

Finally, we give an example of non-compact pseudo-Riemannian semisimple symmetric pairs 
which satisfy (N\ref{item:def}) but do not satisfy (N\ref{item:adjoint}). 

\begin{eg}
\label{eg:counterexample}
Let $m,n$ be positive integers. 
We take a non-compact real semisimple Lie algebra $\mathfrak{g}_0$ 
as $\mathfrak{sl}(m+n,\mathbb{R})=\{ X\in M(m+n,\mathbb{R}):\operatorname{tr}X=0\} $ 
and an involution $\sigma$ on $\mathfrak{g}_0$ as $\sigma (X)=I_{m,n}XI_{m,n}$ $(X\in \mathfrak{g}_0)$ 
where $I_{m,n}$ is defined by (\ref{eq:I_{m,n}}). 
Then, $(\mathfrak{g}_0,\sigma )$ is a non-compact pseudo-Riemannian semisimple symmetric pair, 
in particular, $\sigma$ is not a Cartan involution of $\mathfrak{g}_0$. 
Since $\mathfrak{g}_0$ is simple, $(\mathfrak{g}_0,\sigma )$ satisfies (N\ref{item:def}). 

On the other hand, the fixed point set $\mathfrak{h}_0=\mathfrak{g}_0^{\sigma }$ is of the form 
\begin{align*}
\mathfrak{h}_0
&=\left\{ 
	\left( 
		\begin{array}{cc}
		A & O \\
		O & D
		\end{array}
	\right) :
	\begin{array}{c}
	A\in M(m,\mathbb{R}),~D\in M(n,\mathbb{R}),\\
	\operatorname{tr}A+\operatorname{tr}D=0
	\end{array}
\right\} ,
\end{align*}
which is $\mathfrak{s}(\mathfrak{gl}(m,\mathbb{R})+\mathfrak{gl}(n,\mathbb{R}))$, 
and the fixed point set $\mathfrak{q}_0=\mathfrak{g}_0^{-\sigma}$ is 
\begin{align*}
\mathfrak{q}_0
&=\left\{ 
	\left( 
		\begin{array}{cc}
		O & X \\
		Y & O
		\end{array}
	\right) :
	X\in M(m,n\,;\mathbb{R}),~Y\in M(n,m\,;\mathbb{R})
\right\} . 
\end{align*}
Here, we take two subspaces $\mathfrak{q}_1,\mathfrak{q}_2$ of $\mathfrak{q}_0$ as 
\begin{align*}
\mathfrak{q}_1
&=\left\{ 
	\left( 
		\begin{array}{cc}
		O & X \\
		O & O
		\end{array}
	\right) :
	X\in M(m,n\,;\mathbb{R})
\right\} ,\\
\mathfrak{q}_2
&=\left\{ 
	\left( 
		\begin{array}{cc}
		O & O \\
		Y & O
		\end{array}
	\right) :
	Y\in M(n,m\,;\mathbb{R})
\right\} . 
\end{align*}
Then, they are $(\operatorname{ad}\mathfrak{h}_0)$-invariant, 
and then $\mathfrak{q}_0$ is decomposed into two $(\operatorname{ad}\mathfrak{h}_0)$-invariant subspaces as 
$\mathfrak{q}_0=\mathfrak{q}_1+\mathfrak{q}_2$. 
Hence, the adjoint $\mathfrak{h}_0$-action on $\mathfrak{q}_0$ is not irreducible, 
from which (N\ref{item:adjoint}) does not hold. 
\end{eg}

As mentioned in Section \ref{subsec:irr}, 
any non-compact semisimple symmetric pair $(\mathfrak{g}_0,\sigma )$ 
can be decomposed into the direct sum 
\begin{align*}
(\mathfrak{g}_0,\sigma )
=(\mathfrak{l}_0^{(1)},\sigma _1)\oplus \cdots \oplus 
	(\mathfrak{l}_0^{(k)},\sigma _k)
\end{align*}
of non-compact semisimple symmetric pairs $(\mathfrak{l}_0^{(1)},\sigma _1)$, 
\ldots , $(\mathfrak{l}_0^{(k)},\sigma _k)$ 
which satisfy (N\ref{item:def}). 
On the other hand, 
even though the adjoint $\mathfrak{h}_0$-action on $\mathfrak{q}_0$ is not irreducible, 
$(\mathfrak{g}_0,\sigma )$ does not always have a non-trivial $\sigma$-invariant ideal 
(see Example \ref{eg:counterexample}). 
In this context, 
Definition \ref{def:irr-pair} (that is, (N\ref{item:def})) would be appropriate 
to a definition for a non-compact semisimple symmetric pair to be irreducible. 
%for our study of various properties via Theorem \ref{thm:duality-thm}. 

%%%%%%%%%%%%%%%%%%%%%%%%%%%%%%%%%%%%%%%%%%%%%%%%%%%%%%%%%%%%%%%%%%%%%%%%%%%%%%%%%%%%%%%%%%%%%%%%%%%%%

%%% \section{References}

\end{document}